\documentclass[a4paper,12pt]{article}

\usepackage{setspace} 
\usepackage{amsmath}
\usepackage{amssymb}
\usepackage{ascmac}
\usepackage{amsthm}
\usepackage{amscd}
\usepackage{color}
\usepackage{enumerate}
\usepackage[dvips]{graphicx}
\usepackage{float}
\usepackage{wrapfig}

  \makeatletter
    
    \@addtoreset{equation}{section}
    \makeatother

\theoremstyle{definition}
\newtheorem{definition}{Definition}[section]
\theoremstyle{plain}
\newtheorem{theorem}[definition]{Theorem}
\newtheorem{corollary}[definition]{Corollary}
\newtheorem{lemma}[definition]{Lemma}
\newtheorem{proposition}[definition]{Proposition}
\theoremstyle{remark}
\newtheorem{remark}[definition]{Remark}
\newcommand{\Do}{\partial\!\!\!/}

\begin{document}

\title{Current algebras on \(S^3\) of complex Lie algebras\\( Revised version )}
\author{Tosiaki Kori 
\\Department of Mathematics\\
Graduate School 
of Science and Engineering\\
Waseda University,\\Tokyo, Japan
\\email{ kori@waseda.jp}}
\date{}
\maketitle

\begin{abstract}
This is a full revised version of the previous same titled article.   The 2-cocycle of the central extension of the current algebra on \(S^3\) is taken the place of a new quarternion valued one, that is defined by the boundary Dirac operator.   And we added a discussion on the symmetric invariant bilinear forms associated to the 
central extension \(\widehat{\mathfrak{g}}\).   \\
A current algebra is a Lie algebra of smooth mappings of a given manifold into a Lie algebra, or its central extension.     Affine Kac-Moody algebra is an example where the manifold is \(S^1\subset \mathbf{C}\setminus\{0\}\).   
    We shall extend this theory to a Lie algebra of smooth mappings of \(S^3\subset \mathbf{C}^2\setminus\{0\}\) into a complex simple Lie algebra \(\mathfrak{g}\).  
    Let  \(\mathcal{L}\,\) be the  \(\mathbf{C}-\)algebra generated by  the Laurent polynomial type harmonic spinors over \(\mathbf{C}^2\setminus \{0\}\).    Here we mean "harmonic" as the zero-mode of the Dirac operator on \(\mathbf{C}^2\).      
  The real part \(\mathcal{K}\) of \(\mathcal{L}\,\)  becomes a commutative real subalgebra of \(\mathcal{L}\,\).  
    For a simple Lie algebra \(\mathfrak{g}\),  the {\it \(\mathfrak{g}\)-current algebra}  is defined as the ( real ) Lie algebra \(\mathcal{L}\mathfrak{g}\)  that is generated by \(\mathcal{L}\otimes_{\mathbf{R}}\mathfrak{g}\).      For the Cartan subalgebra  \(\mathfrak{h}\) of \(\mathfrak{g}\,\),   \(\,\mathcal{K}\mathfrak{h}=\mathcal{K}\otimes_{\mathbf{R}}\mathfrak{h}\) becomes a   Cartan subalgebra of  \(\mathcal{L}\mathfrak{g}\).   We have the weight space decomposition of \(\mathcal{L}\mathfrak{g}\) with respect to the Cartan subalgebra \(\mathcal{K}\mathfrak{h}\).    We introduce a quarternion valued 2-cocycle on the \(\mathfrak{g}\)-current algebra \(\mathcal{L}\mathfrak{g}\). 
   Then we have the central extenstion \(\mathcal{L}\mathfrak{g}\oplus \mathbf{H}{\rm c}\) associated to the 2-cocycle.   Adjoining a derivation 
 coming from the radial vector field \(\frac{\partial}{\partial n}\) on \(S^3\) we obtain the second central extension  \(\,\widehat{\mathfrak{g}}=
 \mathcal{L}\mathfrak{g}\oplus \mathbf{H}{\rm c} \oplus \mathbf{R}{\rm d}\).     We shall investigate the root space decomposition and the Chevalley generators  of \(\,\widehat{\mathfrak{g}}\,\).  
\end{abstract}

Mathematics Subject Classification.    17B65,  17B67, 22E67,  81R10,  81R25,  15B33.\\
{\bf Key Words }    Infinite dimensional Lie algebras,  Current algebras, 
	Kac-Moody Lie algebra, Spinor analysis.

\medskip

\section{Introduction}

 An affine Kac-Moody algebra of untwisted type can be realized in terms of a central extension of the loop algebra of a simple Lie algebra, \cite{K}.     Let \(L=\mathbf{C}[t,t^{-1}]\) be the algebra of Laurent polynomials \(\sum_i\,a_it^i\).   The residue function 
\(Res\,:\,L\,\longrightarrow\,\mathbf{C}\,\) is given by \(Res\,(\sum_i\,a_it^i\,)=a_{-1}\).    Given a simple Lie algebra \(\mathfrak{g}\), then \(L\mathfrak{g}=L\otimes_{\mathbf{C}}\mathfrak{g}\) may be made into a Lie algebra in a unique way satisfying 
\[[P\otimes x,\,Q\otimes y]=PQ\,\left[\, x\,,\,y\,\right] \,,\quad\mbox{ for } P,Q\in L\,,\,x,y\in\mathfrak{g}\,.\]
\(L\mathfrak{g}\) is called the loop algebra of \(\mathfrak{g}\).   Let \( \langle\,\cdot \vert \,\cdot \rangle\) be the invariant bilinear form on \(\mathfrak{g}\).   
We define a bilinear form 
\[
\langle\,\cdot\,\vert\, \cdot\,\,\rangle\,:\,L\mathfrak{g}\times L\mathfrak{g}\,\longrightarrow\,\mathbf{C}[t,t^{-1}]\,\]
by \(\langle\,P\otimes x\,\vert \,Q\otimes y\,\rangle \,=PQ\,\langle\,x\,\vert \,y\,\rangle \,\).   
  We define a 
   2-cocycle \(\kappa\) on the Lie algebra \(L\mathfrak{g}\, \) by the formula 
\[\kappa (P\otimes x\,, \,Q\otimes y)\,=\,Res\,(\frac{dP}{dt}\,Q\,) \, \langle x\vert y\rangle\,.\]
   Let \(L\mathfrak{g} \oplus \mathbf{C}{\rm c}\) be the extension of \(L\mathfrak{g}\) by the 1-dimensional center \(\mathbf{C}{\rm c}\) associated to the 2-cocycle \(\kappa \) whose Lie multiplication is given by 
   \[[\,a+\lambda {\rm c}\,,\,b+\mu{\rm c}\,]\,=\, [\,a\,, b\,]\,+\,\kappa(a,b){\rm c}\,.\]
    The Euler derivation \(t\frac{d}{dt}\) acts on \(L\mathfrak{g} \oplus \mathbf{C}{\rm c} \) as an outer derivation; \(\Delta(a+\lambda {\rm c})=t\frac{da}{dt}\) for \(a\in L\mathfrak{g}\), \(\lambda\in\mathbf{C}\).    
Then  adjoining the derivation \(\Delta\) to \(L\mathfrak{g} \oplus \mathbf{C}{\rm c}\) we have the Lie algebra 
\[
\widehat{\mathfrak{g}}\,=\,L\mathfrak{g} \oplus \mathbf{C}{\rm c}\oplus \mathbf{C}{\rm d}\]
by defining the Lie product as 
\[[\,a+\gamma {\rm c} + \lambda {\rm d}\,,\,b+\delta {\rm c} +\mu {\rm d}\,]\,=\,[\,a\,,\,b\,]+ \kappa(a,b){\rm c} + \lambda\,\Delta(b)\,+\mu\,\Delta(a)\,.\]
We follow this procedure to have a central extension of the current algebra on \(S^3\).    While the Laurent polynomial type harmonic spinors over \(\mathbf{C}^2\setminus \{0\}\) do not constitute an algebra we shall consider the \(\mathbf{C}-\)algebra generated by the Laurent polynomial type harmonic spinors which  we call {\it the algebra of current  over \(S^3\)} and denote by \(\mathcal{L}\,\).   It plays the same role as the algebra \(L=\mathbf{C}[t,\,t^{-1}]\) of Laurent polynomials over \(S^1\) does.     The  {\it current algebra of \(\,\mathfrak{g}\)} is the real Lie algebra \(\mathcal{L}\mathfrak{g}\,\) that is generated by \(\,\mathcal{L}\otimes_{\mathbf{C}}\mathfrak{g}\) .    We shall introduce a 2-cocycle on   \(\mathcal{L}\,\) that takes values in the quarternions \(\mathbf{H}\).    Then it is extended to a 2-cocycle on the current algebra \(\mathcal{L}\mathfrak{g}\).   
For this purpose we prepare in section 2 
 a rather long introduction to our previous results on analysis of harmonic  spinors on \(\mathbf{C}^2\), \cite{ F,G-M, Ko1, Ko2, Ko3} and \cite{K-I}, that is, we develop some parallel results as  in classical analysis;  the separation of variable method for Dirichlet problem, the expansion by  eigenfunctions of Laplacian, Cauchy integral formula for holomorphic functions  and Laurent expansion of meromorphic functions etc..   For example, the Dirac operator on spinors corresponds to the Cauchy-Riemann operator on complex functions.
Let  \(\Delta=\mathbf{H}^2\) be the 4-dimensional spinor space, that is, an irreducible representation of the complex Clifford algebra \(\, {\rm Clif }^c_4=End(\Delta)\).  
The algebraic basis of \({\rm Clif }^c_4\) is given by 
the Dirac matrices:
\(\gamma_k\,=\,\left(\begin{array}{cc}0&-i\sigma_k\\ i\sigma_k&0\end{array}\right)\,,\, k=1,2,3\), and
\(\gamma_4\,=\,\left(\begin{array}{cc}0&-I\\ -I&0\end{array}\right)\,.
\)
Where \(\sigma_k\) are Pauli matrices.     Let \(S= \mathbf{R}^4\times \Delta\) be the spinor bundle.   The Dirac operator is defined by  the following formula:
\begin{equation*}
\mathcal{D}=\,-\,\frac{\partial}{\partial x_1}\gamma_4\,-\,\frac{\partial}{\partial x_2}\gamma_3\,-\,\frac{\partial}{\partial x_3}\gamma_2\,-\,\frac{\partial}{\partial x_4}\gamma_1\,:\,C^{\infty}( \mathbf{R}^4, S)\longrightarrow\,
C^{\infty}( \mathbf{R}^4, S)\,.
\end{equation*}  
Let 
\(S^{\pm} = \mathbf{R}^4\times \Delta^{\pm}\) be the ( even and odd ) half spinor bundle  corresponding to the decomposition \(\Delta=\Delta^+\oplus\Delta^-\): \(\Delta^{\pm}\simeq \mathbf{H}\).    
   The half spinor Dirac operator  $D=\mathcal D\vert S^+$  has the polar decomposition:
\(
D = \gamma_+ \left( \frac{\partial}{\partial n} - \Do \right) 
\) 
with the tangential (nonchiral) component $\Do\,$ and the radial vector \(\frac{\partial}{\partial n}\).   The tangential Dirac operator \(\Do \) on \(S^3\) is a self adjoint elliptic differential operator.     
The eigenvalues of \(\Do\) are  \(\{\frac{m}{2},\,\,-\frac{m+3}{2}\,;\,m=0,1,\cdots \} \) with multiplicity \((m+1)(m+2)\).   We have an explicitly written polynomial formula of eigenspinors \(\left\{ \phi^{+(m,l,k)},\,\phi^{-(m,l,k)}\right\}_{0\leq l\leq m,\,0\leq k\leq m+1}\)  corresponding to the eigenvalues \(\frac{m}{2}\) and  \(-\frac{m+3}{2}\,\) respectively that give rise to a complete orthonormal system in \(L^2(S^3, S^+)\), \cite{Ko1, Ko2}.      A spinor \(\phi\) on a domain \(G\subset \mathbf{C}^2\) is called a {\it harmonic spinor} on \(G\) if \(D\phi=0\).   Each \(\phi^{+(m,l,k)}\) is extended to a harmonic spinor on \(\mathbf{C}^2\), while each \(\phi^{-(m,l,k)}\) is extended to a harmonic spinor on \(\mathbf{C}^2\setminus \{0\}\) that is regular at infinity.    
Every harmonic spinor \(\varphi\) on \(\mathbf{C}^2\setminus \{0\}\) has an expansion by the basis \(\phi^{\pm(m,l,k)}\):
 \begin{equation*}
 \varphi(z)=\sum_{m,l,k}\,C_{+(m,l,k)} \phi^{+(m,l,k)}(z)+\sum_{m,l,k}\,C_{-(m,l,k)}\phi^{-(m,l,k)}(z).
 \end{equation*}
 We call the above series a {\it Laurent polynomial type spinor} if finitely many coefficients \(C_{-(m,l,k)}\) are non-zero, and  the  coefficients 
 \[\left(\begin{array}{c}C_{-(0,0,1)}\\ C_{-(0,0,0)}\end{array}\right)\]
 is called the {\it quarternion residue } of \(\varphi\) and is denoted by \(qRes\,[\varphi ]\).     The space of Laurent polynomial type spinors  is denoted by   \(\mathbf{C}[\phi ^{\pm}] \).      Let \(\mathbf{H}\) be the algebra of quarternion numbers.   
 We look an even spinor also as a \(\mathbf{H}\)-valued smooth function: \(C^{\infty}(S^3,S^+)=C^{\infty}(S^3,\mathbf{H})\), so that 
the space of spinors \(C^{\infty}(S^3, S^+)\) is endowed with a multiplication rule:
  \begin{equation}
  \phi_1\cdot \phi_2\,=\,
   \begin{pmatrix} \,u_1 u _2 - \Bar{v}_1 v_2 \,\\[0.2cm]
\,v_1u _ 2 + \Bar {u} _1v _2 \, \end{pmatrix} \,,  \quad \mbox{ for }\,\phi_i=   \begin{pmatrix} \,u_i  \\[0.2cm]
\,v_i \end{pmatrix},\,i=1,2\,.
\label{multiple}
\end{equation}
 \(\mathbf{C}[\phi ^{\pm}] \) does not constitute an algebra, as we see by  the example \(\phi^{+(1,0,0)}\cdot\phi^{-(0,0,0)}\) which is not a harmonic spinor.     The  {\it algebra of current  \(\, \mathcal{L}\,\) on \(S^3\) } is a \(\mathbf{C}-\)subalgebra of \(C^{\infty}(S^3, S^+)\) ( or rather \(C^{\infty}(\mathbf{C}^2\setminus \{0\}, S^+)\) ) that is generated by  \(\mathbf{C}[\phi ^{\pm}] \) .   In section 3 we introduce a 2-cocycle on \(\mathcal{L}\,\).     For spinors \(\varphi,\, \psi\,\in \mathcal{L} \,\) we put 
\begin{equation}
A(\varphi\,,\,\psi)\,=\,q\,Res\,[\,\Do\varphi\cdot \psi\,-\,\Do\psi\cdot\varphi\,]\,.
\end{equation}
Then \(A\) gives a \(\mathbf{H}-\)valued sqew-symmetric bilinear form on \(\mathcal{L}\) and satisfies the cocycle condition:
\begin{equation}
A(\phi_1\phi_2, \phi_3)+A(\phi_2\phi_3, \phi_1)+A(\phi_3\phi_1, \phi_2)=0\,.\end{equation}
Hitherto we have prepared  the space of spinors \(C^{\infty}(S^3,S^+)\) and the algebra of current \(\mathcal{L}\,\) that will play the role of coefficients of our {\it current algebra } discussed below.   These are {\it complex} algebras.     On the other hand  \(C^{\infty}(S^3,S^+)\simeq C^{\infty}(S^3,\mathbf{H})\) has a \(\mathbf{H}\)-module structure, while our basic interest is on the  real Lie algebra generated by \(\mathcal{L}\otimes_{\mathbf{C}}\mathfrak{g}\,\).    In such a way  it is frequent that we deal with the fields \(\mathbf{H}\),  \(\mathbf{C}\) and \(\mathbf{R}\) in one formula.   
So to prove a steady point of view for our subjects we shall introduce the concept of {\it quartrernion Lie algebras}, \cite{Kq}.      First we note that a quarternion module \(V=\mathbf{H}\otimes_{\mathbf{C}}V_o=V_o+JV_o\,\), \(\,V_o\) being a \(\mathbf{C}\)-module, has two  involutions \(\sigma\) and \(\tau\):
\[\sigma(u+Jv)=u-Jv\,,\quad \tau(u+Jv)=\overline u+J\overline v\,,\quad u,v\in V_o\,.\]
  A {\it quarternion Lie algebra} \(\,\mathfrak{q}\) is defined as a real submodule of a   quarternion module \(V\) that is endowed with a real Lie algebra structure compatible with the involutions \(\sigma\) and \(\tau\):
\begin{eqnarray*}
&\sigma\mathfrak{q}\,\subset\mathfrak{q}\,,\\[0.2cm]
&\sigma [x\,,y]\,=[\sigma x\,,\sigma y]\,,\quad \tau [x\,,y]\,=[\tau x\,,\tau y]\,\quad \mbox{ for }\, x,y\in \mathfrak{q} . \label{involutions}
\end{eqnarray*}
 For a complex Lie algebra \(\mathfrak{g}\) the {\it quarternification} of \(\mathfrak{g}\) is a quarternion Lie algebra \(\mathfrak{g}^q\) that is generated ( as a real Lie algebra ) by  \(\mathbf{H}\otimes_{\mathbf{C}}\mathfrak{g}\).     For example, 
\(\mathfrak{so}^{\ast}(2n)=\mathbf{H}\otimes_{\mathbf{C}}\mathfrak{so}(n,\mathbf{C})\) is the quarternification of \(\mathfrak{so}(n,\mathbf{C})\).     
 \(\mathfrak{sl}(n,\mathbf{H})\) is the quarternification of  \(\mathfrak{sl}(n,\mathbf{C})\) though  \(\mathbf{H}\otimes_{\mathbf{C}}\mathfrak{sl}(n,\mathbf{C})\) is not a Lie algebra.      The algebra of current \(\mathcal{L}\,\) is a quarternion Lie algebra.    In fact \(\mathcal{L}\) is a real submodule of \(C^{\infty}(S^3,\,\mathbf{H})\) that is invariant under the involutions \(\sigma\) and \(\tau\).    The real part  \(\mathcal{K}\,=\{\phi\in\mathcal{L};\,\sigma{\phi}=\phi, \,\tau\phi=\phi\,\}\) plays an important role.   
\(\mathcal{K}\) is a commutative normal subalgebra of \(\mathcal{L}\), and satisfies the condition \([\mathcal{K},\,\mathcal{L}]=0\).
  
   Let \(\mathfrak{g}\) be a simple Lie algebra that we suppose to be a subalgebra of \(\mathfrak{gl}(n,\mathbf{C})\).     Let  \(\mathcal{L}\mathfrak{g}\) be the quarternion Lie algebra generated by   \(\mathcal{L}\otimes_{\mathbf{C}}\mathfrak{g}\) with the Lie bracket defined by
 \begin{equation}
  [\phi_1\otimes X_1\,,\,\phi_2\otimes X_2\,]\,=\,(\phi_1\cdot \phi_2)\otimes (X_1X_2)\,-\,(\phi_2\cdot \phi_1)\otimes (X_2X_1)
  \end{equation}
for \(\phi_1,\,\phi_2\in \mathcal{L},\,X_1,X_2\in\mathfrak{g}\,.\)
Here the right hand side is the bracket of the tensor product of the associative algebra \(\mathcal{L}\) and the matrix algebra \(\mathfrak{g}\).    \(\mathcal{L}\mathfrak{g}\)  is called the  \(\,\mathfrak{g}\)-{\it current algebra}.
Let \(\mathfrak{h}\) be the Cartan subalgebra of \(\mathfrak{g}\).      Let \(\mathcal{K}\mathfrak{h}=\mathcal{K}\otimes_{\mathbf{R}}\mathfrak{h}\).   We find that \(\mathcal{K}\mathfrak{h}\) is a Cartan subalgebra of  \(\mathcal{L}\mathfrak{g}\).   
 It extends the adjoint representation \(ad_{\mathfrak{h}}:\mathfrak{h}\longrightarrow End_{\mathbf{C}}(\mathfrak{g})\) to the  adjoint representation \(ad_{\mathcal{K}\mathfrak{h}}:\mathcal{K}\mathfrak{h}\longrightarrow End_{\mathcal{L}}(\mathcal{L}\mathfrak{g})\).       The associated weight space decomposition of  \(\mathcal{L}\mathfrak{g}\) with respect to  \(\mathcal{K}\mathfrak{h}\) will be given.   We find that the space of non-zero weights of \(\mathcal{L}\mathfrak{g}\) corresponds bijectively to the root space of \(\mathfrak{g}\).    
Let \(\mathfrak{g}_{\lambda}\) be the root space of root \(\lambda\) and let \(\Phi^{\pm}\) be the set of positive (respectively negative ) roots of \(\mathfrak{g}\).  
 Then we have the triangular decomposition of the \(\mathfrak{g}\)-current algebra:
\begin{eqnarray*}
 \mathcal{L}\mathfrak{g}&=&\mathcal{L}\mathfrak{h}\,+\,
\mathcal{L}\mathfrak{e}\,+\,\mathcal{L}\mathfrak{f}\,,\quad\mbox{( direct sum )}\,,
  \\[0.2cm]
with  \qquad \mathcal{L}\mathfrak{e}\,&=&\sum_{\lambda\in\Phi^+}\mathcal{L}\otimes_{\mathbf{R}}\mathfrak{g}_{\lambda},\quad
  \mathcal{L}\mathfrak{f}\,=\sum_{\lambda\in\Phi^-}\mathcal{L}\otimes_{\mathbf{R}}\mathfrak{g}_{\lambda}\,.
  \end{eqnarray*}
     \(\,\mathcal{L}\mathfrak{h}\) has the weight \(0\): \([\mathcal{K}\mathfrak{h},\,  \mathcal{L}\mathfrak{h}\,]=0\,\).    
     
  We discuss in section 5 
our central subject to give the central extension of  \(\mathfrak{g}\)-current algebra.    We extend the \(\mathbf{H}-\)valued 2-cocycle \(\,A\) on \(\mathcal{L}\) to a 2-cocycle on \(\mathcal{L}\mathfrak{g}\,\) by the formula 
\begin{equation}
A(\phi\otimes X,\,\psi\otimes Y)\,=\,(X \vert Y)\,A(\phi,\psi),\quad \phi,\,\psi\in \mathcal{L}\,,\,X,Y\in \mathfrak{g},
\end{equation}
where \((X\vert Y)=Trace(XY)\) is the Killing form of \(\mathfrak{g}\).     Then we have the associated central extension:
 \[\,\mathcal{L}\mathfrak{g}({\rm c})=
\mathcal{L}\mathfrak{g}\,\oplus\,\mathbf{H}{\rm c} ,\]
 which is a quarternion Lie algebra.  The radial vector field \(\frac{\partial}{\partial n}\) on \(\mathbf{C}^2\setminus 0\) acts on \(\mathcal{L}\mathfrak{g}({\rm c})\) as an outer derivation.     Then, adjoining the derivation \(\frac{\partial}{\partial n}\), we have the second central extension: 
\begin{equation*}
\widehat{\mathfrak{g}}=\mathcal{L}\mathfrak{g}({\rm c})\oplus\mathbf{C}{\rm d}.
\end{equation*}
( Actually we adopt the prolonged radial derivation ; \(2\vert z\vert^3\frac{\partial}{\partial n}\) ).
We shall investigate the root space decomposition of \(\,\widehat{\mathfrak{g}}\,\).  For a root \(\alpha\in\Phi\),  let \(\mathfrak{g}_{\alpha}=\{x\in \mathfrak{g};\,[\,h,\,x\,]=\alpha(h)x\,,\,\forall h\in \mathfrak{h}\,\}\) denote the root space of \(\alpha\).   
Put 
\begin{equation*}
\widehat{\mathfrak{h}}\,=\,
 \,\mathfrak{h} \,\oplus  \mathbf{H}{\rm c}\oplus \mathbf{C}\,{\rm d}\,
\end{equation*}
  \(\widehat{\mathfrak{h}}\) is a commutative subalgebra of \(\widehat{\mathfrak{g}}\,\) and   
 \(\,\widehat{\mathfrak{g}}\) is decomposed into a direct sum of the simultaneous eigenspaces of \(ad\,(\hat h)\), \(\,\hat h\in \widehat{\mathfrak{h}}\,\), and  
 $\Phi \subset \mathfrak{h}^{\ast}$ is regarded as a subset of $\,\widehat{\mathfrak{h}}^{\ast}$.    
 
We introduce  \(\Lambda\in \widehat{\mathfrak{h}}^{\ast}; \) as the dual elements of \(c \), and \( \delta\in \widehat{\mathfrak{h}}^{\ast}\)  as the dual element of \(d\).     
Then \(\alpha_1,\,\cdots\,,\alpha_l,\,\delta,\,\Lambda\,\) give a basis of \(\widehat{\mathfrak{h}}^{\ast}\).   
The set of simple root are 
\begin{equation*}
\widehat{\Phi} = \left\{ \frac{m}{2} \delta + \alpha\,;\quad \alpha \in \Phi\,,\,m\in\mathbf{Z}\,\right\} \bigcup \left\{ \frac{m}{2} \delta ;\quad  m\in \mathbf{Z} \, \right\}  \,.
\end{equation*}
\(\widehat{ \mathfrak{g}}\) has the weight space decomposition:
\begin{equation*}
\widehat{ \mathfrak{g}}\,=\, \oplus_{m\in\mathbf{Z} }\, \widehat{ \mathfrak{g}}_{\frac{m}{2}\delta }\,\oplus\,\,\oplus_{\alpha\in \Phi,\, m\in\mathbf{Z} }\, 
\widehat{ \mathfrak{g}}_{\frac{m}{2}\delta +\alpha}\,.
\end{equation*}
Each weight space is given as follows.    
 \begin{eqnarray*}  
\widehat{ \mathfrak{g}}_{\frac{m}{2}\delta + \alpha}\,&=&\mathcal{L}[m] \otimes_{\mathbf{C}} \mathfrak{g}_{ \alpha}\,,\quad\mbox{ for \(\alpha\neq 0\) and and \(m\in \mathbf{Z}\)}, \,,\\[0.2cm] 
\widehat{ \mathfrak{g}}_{0\delta }&=& (\,\mathcal{L}[0] \mathfrak{h}\,)\oplus \mathbf{H}{\rm c}\oplus\,\mathbf{C}{\rm d}\,\supset\,\widehat{\mathfrak{h}}\,, 
\\[0.2cm]
 \widehat{ \mathfrak{g}}_{\frac{m}{2}\delta }&=&  \,\mathcal{L}[m] \otimes_{\mathbf{C}} \mathfrak{h}\,\,, \quad\mbox{for  \(0\neq  m\in\mathbf{Z} \) . }\,
 \end{eqnarray*}
 Where \(\mathcal{L}[m]\) is the subspace of \(\mathcal{L}\) constituting of those elements 
 \(\phi\in \mathcal{L}\,\) that are of homogeneous degree \(m\).   \(\mathcal{L}[0]\mathfrak{h}\) is the Lie subalgebra generated by 
\(\mathcal{L}[0]\otimes_{\mathbf{C}}\mathfrak{h}\).

  
\section{Spinor analysis on $S^3\subset \mathbf{C}^2$.}

Here we prepare a fairly long preliminary of spinor analysis on \(\mathbf{R}^4\) because  I think various subjects belonging to quarternion analysis or detailed properties of harmonic spinors of the Dirac operator on \(\mathbf{R}^4\) are not so familiar to the readers.    
We refer to \cite{F,  Ko1} for the exposition on Dirac operators on \(\mathbf{R}^4\) and to  \cite{D-S-Sc, G-M, Ko2} for the function theory of harmonic spinors. 
Subsections 2.1, 2.2, 2.3 are to remember the theory of harmonic spinors.      

 \subsection{ Spinors and the Dirac operator on \(\mathbf{R}^4\).}

\subsubsection{}   
Let \(\mathbf{K}\) be the field \(\mathbf{R}\) or \(\mathbf{C}\).   
Let \(V\) be a \(\mathbf{K}\)-vector space  equipped with a  quadratic form \(q\) over the field \(\mathbf{K}\).   The Clifford algebra 
\(C_{\mathbf{K}}(V,q)\) is a \(\mathbf{K}\)-algebra which contains \(V\) as a sub-vector space and is generated by 
the elements of \(V\) subject to the relations
\[v_1v_2+v_2v_1=2q(v_1,v_2)\,,\]
for \(v_1,\,v_2\in V\).     
  In the sequel we denote 
 \(
{\rm Clif}_n=C_{\mathbf{R}}(\mathbf{R}^n,\,-x_1^2-\cdots-x_n^2)\) and \({\rm Clif}_{n}^c\,=  C_{\mathbf{C}}(\mathbf{C}^n,z_1^2+\cdots+z_n^2)\).  
It holds \({\rm Clif}_n^c={\rm Clif}_n\otimes_{\mathbf{R}}\mathbf{C}\).   
    We have an important isomorphism:
\begin{equation}\label{reduction}
{\rm Clif}_{n+2}^c={\rm Clif}_n^c\otimes_{\mathbf{C}}\mathbf{C}(2)\, 
\,.
\end{equation}
Here \(\mathbf{K}(m)\) denotes the algebra of \(m\times m\)-matrices with entries in the field \(\,\mathbf{K}\).     Let \(\mathbf{H}\) be the algebra of quarternion numbers that are formed on \(\mathbf{R}\) by the symbols \(i,\,j,\, k\).      
The left multiplication of \(\mathbf{H}\) yields an endomorphism of \(\mathbf{H}\,\); \(\mathbf{H}\simeq End_{\mathbf{H}}\mathbf{H}\simeq \mathbf{C}(2)\).   Then the corresponding matrices to \(i\,,j\,,k\in \mathbf{H}\,\) are given by \( i\sigma_3,\, i\sigma_2,\, i\sigma_1\).   Where \(\sigma_k\,\), \(k=1,2,3\), are 
 Pauli matrices: 
\begin{equation*}
\sigma_1=\left(\begin{array}{cc}
0&1\\[0.2cm] 1&0\end{array}\right)\,,
\, \sigma_2=\left(\begin{array}{cc}
0&-i\\[0.2cm] i&0\end{array}\right)\,,\, 
\sigma_3=
\left(\begin{array}{cc}
1&0\\[0.2cm] 0&-1\end{array}\right)\,.
\end{equation*}
The relations \(\sigma_i^2=-1\), \(i=1,2,3\), and \(\sigma_1\sigma_3+\sigma_3\sigma_1=0\) 
shows that \(\{\sigma_1,\,\sigma_3\}\) generate \({\rm Clif}^c_2\), so that \({\rm Clif}^c_2=\mathbf{H}\).   Let 
\(\Delta=\mathbf{C}^2\otimes_{\mathbf{C}}\mathbf{C}^2\) be the vector space of complex 4-spinors that gives 
the spinor representation of Clifford algebra \({\rm Clif }^c_4\,\):
\[{\rm Clif }^c_4={\rm End} _{\mathbf{C}}(\Delta)=
\mathbf{C}(4)\,.\]
Then   \(\,{\rm Clif }_4^c\) is 
generated by  the following Dirac matrices:
\begin{equation*}
\gamma_k\,=\,\left(\begin{array}{cc}0&-i\sigma_k\\ i\sigma_k&0\end{array}\right)\,,\quad k=1,2,3,\quad
\gamma_4\,=\,\left(\begin{array}{cc}0&-I\\ -I&0\end{array}\right)\,.
\end{equation*}
The set 
\begin{equation}\label{Dmatrices}
\left\{\gamma_p, \quad \gamma_p\gamma_q,\quad \gamma_p\gamma_q\gamma_r,\quad\gamma_p\gamma_q\gamma_r\gamma_s\,;\quad 1\leq p,q,r,s\leq 4\,\right\}
\end{equation}
gives a 16-dimensional basis of the representation \({\rm Clif}^c_4\,\simeq\, {\rm End}_{\mathbf{C}}(\Delta)\,\) with the following relations:
\begin{equation*}
\gamma_p\gamma_q+\gamma_q\gamma_p=2\delta_{pq}\,.
\end{equation*}
The representation \(\Delta\) decomposes into irreducible representations \(\Delta^{\pm}=\mathbf{C}^2\) of \(\,{\rm Spin}(4)\). 

Let 
\(S= \mathbf{C} ^2\times \Delta\) be the trivial spinor bundle on \( \mathbf{C} ^2\).   The corresponding bundle 
\(S^+= \mathbf{C} ^2\times \Delta^+\) ( respectively  \(S^-= \mathbf{C} ^2\times \Delta^-\) )  is called the even ( respectively  odd ) half spinor bundle and the sections are called even ( respectively odd ) spinors.   
On the other hand, since \({\rm Clif}^c_4=\mathbf{H}(2)\otimes_{\mathbf{R}}\mathbf{C}\) and \(\Delta=\mathbf{H}^2=\mathbf{H}\oplus\mathbf{H}\),   we may look an even spinor on \(M\subset \mathbf{R}^4\) as a \( \mathbf{H}\) valued smooth function:
 \(
C^{\infty}(M, \mathbf{H})\,=\, C^{\infty}(M, S^+)\).   We feel free to use the alternative notation to write a spinor:
\begin{equation*}\label{shidentification}
C^{\infty}(M, \mathbf{H})\ni u+jv=p+qi+rj+sk\,
\longleftrightarrow \,\left(\begin{array}{c}u\\v\end{array}\right)\,\in\,C^{\infty}(M, S^+),\quad
\begin{array}{c}u=p+qi\\ v=-r+si\end{array}.
\end{equation*}
We have also the following alternative notation:
\begin{equation}\label{shidentification}
C^{\infty}(M, \mathbf{H})\,\ni\,p+qi+rj+sk\,\longleftrightarrow\,
p\,I+q\,i\sigma_3 \,+ r\,i\sigma_2\,+s\,i\sigma_1\,\in\,C^{\infty}(M, S^+)\,.
\end{equation}
The conjugate quarternion \(\overline x\) of \(x=p+qi+rj+sk\) is defined by 
\(\overline x=p-qi-rj-sk\) and conjugation is an anti-involution; \(\overline{(xy)}=\overline{y}\overline{x}\).   

The multiplication of two even spinors is defined by
\begin{equation}\label{spinormultip}
\phi_1\cdot \phi_2\,=\,
\,\left(\begin{array}{c}u_1u_2-\overline v_1v_2\\  v_1u_2+\overline u_1v_2\end{array}\right)\,\end{equation}
for \(\phi_i=\left(\begin{array}{c}u_i\\ v_i\end{array}\right)\), \(i=1,2\).
It corresponds to the quarternion multiplication:
\[(u_1+jv_1)(u_2+jv_2)=(u_1u_2-\overline v_1v_2)+j( v_1u_2+\overline u_1v_2)\,.\]

\subsubsection{}   

 The Dirac operator is defined by
\begin{equation*}
\mathcal{D} = c \circ d\,:\, C^{\infty}(M, S)\,\longrightarrow\,  C^{\infty}(M, S)\,.
\end{equation*}
where $d : S \rightarrow  T^{*} \mathbf{C}^2\otimes S \simeq  T \mathbf{C}^2\otimes S $ is the covariant derivative which is the exterior differential in this case, 
and $c: T \mathbf{C}^2 \otimes S \rightarrow S$ is the bundle 
homomorphism coming from the Clifford multiplication.   
With respect to the Dirac matrices \(\{\gamma_{j}\}_{j=1,2,3,4}\,\), (\ref{Dmatrices}), the Dirac operator has the expression:
\begin{equation*}
\mathcal{D}=\,-\,\frac{\partial}{\partial x_1}\gamma_4\,-\,\frac{\partial}{\partial x_2}\gamma_3\,-\,\frac{\partial}{\partial x_3}\gamma_2\,-\,\frac{\partial}{\partial x_4}\gamma_1\,.
\end{equation*}
By means of the decomposition $S = S^{+} \oplus S^{-}$ the Dirac operator has 
the chiral decomposition:
\begin{equation*}
\mathcal{D} = 
\begin{pmatrix}
0 & D^{\dagger} \\
D & 0
\end{pmatrix}
: C^{\infty}(\mathbf{C}^2, S^{+} \oplus S^{-}) \rightarrow C^{\infty}(\mathbf{C}^2, S^{+} \oplus S^{-}).
\end{equation*}
If we adopt the notation
\[\frac{\partial}{\partial z_1}=\frac{\partial}{\partial x_1}-i\frac{\partial}{\partial x_2}\,,\quad
\frac{\partial}{\partial z_2}=\frac{\partial}{\partial x_3}-i\frac{\partial}{\partial x_4}\,,\]

 $D$ and \(D^{\dagger}\) have the following coordinate expressions;
\begin{equation*}
D =  \begin{pmatrix} \frac{\partial}{\partial z_1} & - \frac{\partial}{\partial \Bar{z_2}} 
\\ \\ \frac{\partial}{\partial z_2} & \frac{\partial}{\partial \Bar{z_1}} \end{pmatrix} , 
\quad
D^{\dagger} = \begin{pmatrix} \frac{\partial}{\partial \Bar{z_1}} & \frac{\partial}{\partial \Bar{z_2}} 
\\ \\ - \frac{\partial}{\partial z_2} & \frac{\partial}{\partial z_1} \end{pmatrix}.
\end{equation*}

\subsubsection{}
The right action  of \( SU(2)\) on \(\mathbf{C}^2\) is written by 
\[R_gz\,=\,\left(\begin{array}{c}\, az_1-b\overline z_2\, \\ az_2+b\overline  z_1 \end{array}\right),
 \quad g= \left(\begin{array}{cc}a& -\overline b \\ b&\overline a\end{array}\right)\in\, SU(2),\quad z=\left(\begin{array}{c}z_1\\z_2\end{array}\right)\in \mathbf{C}^2.\]
Then the infinitesimal action of \(su(2)\) on \(\mathbf{C}^2\) is   \[((dR_e)X)F=\frac{d}{dt}\vert_{t=0}R_{\exp\,tX}F\,,\quad X\in su(2)\,.\]
It yields the following basis of  vector fields \((\theta_3,\theta_1,\theta_2)\) on 
 $\{|z|= 1\}\simeq S^3$ :
 \begin{equation}
 \theta_1=\frac{1}{2\sqrt{-1}}dR(\sigma_2)\,,\,\, \theta_2=\frac{1}{2\sqrt{-1}}dR(\sigma_1)\,,\,\,\theta_3=-\frac{1}{2\sqrt{-1}}dR(\sigma_3)\,.
 \end{equation}
We prefer often the following basis \((e_+,e_-,\theta)\,\) given by 
\begin{equation}\label{vectbasis}
2\theta_3=-\sqrt{-1}\theta\,, \quad 2\theta_1=e_++e_-\,,\quad 2\theta_2=\sqrt{-1}(e_+-e_-)\,.
\end{equation}
The local coordinate expression of these vector fields becomes:
\begin{eqnarray}\label{bycoordinate}
e_+ &=& -z_2 \frac{\partial}{\partial \Bar{z_1}} +z_1 \frac{\partial}{\partial \Bar{z_2}}\, ,
\quad
e_- = - \Bar{z_2} \frac{\partial}{\partial z_1} + \Bar{z_1} \frac{\partial}{\partial z_2}\,\,,
\\[0.2cm]
\theta &=& \left(z_1 \frac{\partial}{\partial z_1} + z_2 \frac{\partial}{\partial z_2}
- \Bar {z_1} \frac{\partial}{\partial \Bar{z_1}} - \Bar{z_2} \frac{\partial}{\partial \Bar{z_2}}\right)\,,
\end{eqnarray}
and the following commutation relations hold;
\begin{equation*}
[\,\theta , \,e_+\,] \,=\, 2e_+\,,\quad [\,\theta , \,e_-\,] \,=\, -2e_-\,, \quad [\,e_+ ,\,e_-\,]=\,- \theta\,. 
\end{equation*}
The dual basis are given by the differential 1-forms:
\begin{eqnarray*}
\theta_3^{\ast}&=&\frac{1}{2\sqrt{-1}|z|^2}( \overline z_1 d z_1+\overline z_2 dz_2 -z_1 d\overline z_1- z_2 d \overline z_2  ),\\[0.2cm]
\theta_1^{\ast}&=&\frac{1}{2|z|^2}(e_+ ^{\ast}+e_- ^{\ast})\,,\qquad
\theta_2^{\ast}=\frac{1}{2\sqrt{-1}|z|^2}(e_+ ^{\ast}-e_- ^{\ast})\,,
\end{eqnarray*}
where
\begin{equation*}
e_+ ^{\ast}=( -\overline z_2 d\overline z_1+\overline z_1 d\overline z_2 )\,,
\quad
e_- ^{\ast}=( -z_2 d z_1+ z_1 d z_2 )\,,
\end{equation*}
here we wrote the formulae in the form extended to \(\mathbf{C}^2\setminus \{0\}\).     

\noindent \(\theta^{\ast}_k\,\), \(k=1,2, 3\), are real 1-forms:
\(\,
\overline\theta_k^{\ast}=\theta_k^{\ast}\,
\).      It holds that \(\theta_j^{\ast}(\theta_k)=\delta_{jk}\,\) for  \(\,j,k\,=\,1,2,3\,\).  
The 
integrability condition becomes
\begin{equation}\label{integrable}
\frac{\sqrt{-1}}{2}d\theta_3^{\ast}=\theta_1^{\ast}\wedge \theta_2^{\ast}\,,\quad
\frac{\sqrt{-1}}{2}d \theta_1^{\ast}=\theta_2^{\ast}\wedge \theta_3^{\ast}\,,\quad
\frac{\sqrt{-1}}{2}d \theta_2^{\ast}=\theta_3^{\ast}\wedge \theta_1^{\ast}\,.\end{equation}
And 
\(\,\theta_0^{\ast}\wedge \theta_1^{\ast}\wedge \theta_2^{\ast}=d\sigma_{S^3}\,
\) is the volume form on \(S^3\).

\begin{lemma}\label{vanishing}
\begin{equation}\int_{S^3}\,\theta_{k} f\,d\sigma\,=\,0\,,\quad k=1,2,3\,,\end{equation}
for any function \(f\) on \(S^3\). 
\end{lemma}
   This is proved as follows.   We consider the 2-form \(\beta=f\theta_1^{\ast}\wedge\theta_2^{\ast}\).   By virtue of the integrable condition (\ref{integrable}) we have 
\[d\beta=(\theta_3f)\,\theta_3^{\ast}\wedge\theta_1^{\ast}\wedge\theta_2^{\ast}=\theta_3f\,d\sigma\,.\]
Hence 
\[0=\int_{S^3}\,d\beta\,=\,\int_{S^3}\theta_3f\,d\sigma.\]
Similarly for the integral of \(\theta_kf\), \(k=1,2\), of the base vector fields \(\theta_k\,;\,k=1,2\), on \(S^3\). 

On the other hand the lemma is an immediate consequence of the \(SO(4)-\)invariance of the volume form.    This is a remark due to Professor T. Iwai of Kyoto university.

\subsubsection{}

We introduce the normal vector \(\mathbf{n}\) to \(S^3\) and its \((1,0)\)-components \(\nu\) as follows:
\begin{equation}\label{normal}
\mathbf{n}=\,\nu+\overline{\nu}\,\,,\quad \nu=
z_1 \frac{\partial}{\partial z_1} + z_2 \frac{\partial}{\partial z_2}\,.
\end{equation}
Then \(\theta=\nu-\overline{\nu}\) and we have the following commutation relations:
\begin{equation}\label{relation}
\mathbf{ n} \,e_{\pm}\,=\,e_{\pm}\,\mathbf{ n} \,,\quad
\mathbf{ n} \, \theta\,=\,\theta\, \mathbf{ n} \, .
\end{equation}

The normal derivation is defined by
\begin{equation}
\frac{\partial}{\partial n} =\frac{1}{2 |z|}\mathbf{n}\,= \frac{1}{2 |z|}(\nu + \Bar{\nu}), \qquad .\label{normalder}
\end{equation}

We shall denote by \(\gamma\) the Clifford multiplication of the normal vector $\frac{\partial}{\partial n}\,$.   
The multiplication  $\gamma$ changes the chirality:
\(\gamma=
\gamma_+\oplus\gamma_- : S^+ \oplus S^- \longrightarrow S^- \oplus S^+ \), and \(\gamma^2 = 1
\).
The matrix forms of \(\gamma_{\pm}\) are  
\begin{equation}
\gamma_+\,=\,\frac{1}{\vert z\vert}\left(\begin{array}{cc}\overline{z}_1&-z_2\\ \overline{z}_2& z_1\end{array}\right)\,,\,
\quad 
\gamma_-\,=\,\frac{1}{\vert z\vert}\left(\begin{array}{cc}z_1& z_2\\ -\overline{z}_2&\overline{z}_1\end{array}\right)\,.
\end{equation}
\begin{proposition}\cite{Ko1}~~~
The Dirac operators $D$ and \(D^{\dagger}\) have the following polar decompositions:
\begin{eqnarray}\label{tgDirac}
D &= \gamma_+ \left( \frac{\partial}{\partial n} - \Do \right)\,:& \,S^+\longrightarrow\,S^- ,\\[0.2cm]
D^\dagger &= \left( \frac{\partial}{\partial n} + \Do + \frac{3}{2|z|} \right)\gamma_- \,\,:& \, S^-\longrightarrow\,S^+\,,\nonumber
\end{eqnarray}
where the non-chiral Dirac operator $\Do$ is given by
\begin{equation}\label{tgDirac2}
\Do = - \left[ \sum^{3} _{i = 1} \left( \frac{1}{|z|} \theta_i \right) \cdot \nabla_{\frac{1}{|z|} \theta_i} \right]
= \frac{1}{|z|} 
\begin{pmatrix}
-\,i \,\theta_3 & \,\,\theta_1-\,i\theta_2\, \\[0.2cm]
\,-\theta_1-\,i\theta_2\,& \,i\, \theta_3
\end{pmatrix}.
\end{equation}
\end{proposition} 

\(\Do\) restricted is called the {\it tangential } Dirac operator:  
 \begin{equation*}
 \Do | S^3 : C^{\infty} (S^3, S^+) \longrightarrow C^{\infty} (S^3, S^+)
 \end{equation*}
The tangential Dirac operator on \(S^3 \) 
is a self adjoint elliptic differential operator.
 The tangential Dirac operator (\ref{tgDirac2}) is written in the formula:
 \begin{equation*}
 \Do\,=\frac{\sqrt{-1}}{\vert z\vert}(\,-\,\theta_3\,\sigma_3+\theta_1\sigma_2-\,\theta_2\sigma_1\,),
 \end{equation*}
and by the quarternion notation it becomes:
 \begin{equation}
 \Do\,=\frac{1}{\vert z\vert}(-\theta_3\,i+\theta_1 j-\theta_2 k\,).\end{equation}
The following commutation relation holds:
 \begin{equation}\label{DNcommute}
 {\bf n}\,\Do\,=\,\Do\, {\bf n}\,-\,\Do\,.
 \end{equation}
Lemma \ref{vanishing} yields the following
 \begin{proposition}\label{spvanishing}
 \begin{equation}
 \int_{S^3}\,\Do\,\phi\,d\sigma\,=\,0\,,\end{equation}
for any \(\phi\in C^{\infty}(S^3, S^+)\simeq  C^{\infty}(S^3, \mathbf{H})\). 
\end{proposition}

\subsection{Harmonic spinors}

\subsubsection{Harmonic polynomials on \(S^3\subset\mathbf{C}^2\)}

In the following we denote a function $f(z, \Bar{z})$ of variables $z, \Bar{z}\,$ simply by $f(z)$.
  \begin{definition}~~~
   For \(m = 0, 1, 2, \cdots\), and \(\, l\,, k = 0, 1, \cdots , m\),  we define the {\it monomials}:
\begin{eqnarray}
v ^{k} _{(l,m-l)} &=& (e_-)^k z^{l}_{1} z^{m-l}_{2}.\label{v}\\[0.2cm]  
 w^{k} _{(l,m-l)}&=& (-1)^k\frac{l!}{(m-k)!}\,v^{m-l}_{(k,m-k)}\,. \label{wtov}
\end{eqnarray}
\end{definition}
Actually these are polynomials of \(z_1,z_2,\overline{z}_1, \overline{z}_2\) but they play the role of {\it monomials} in our study.   
The monomials  \(v ^{k} _{(l,m-l)} \) in (\ref{v}) come naturally from the right action of \(SU(2)\) on \(\mathbf{C}^2\), so as the monomials 
\( w^{k} _{(l,m-l)}\) from the left action of \(SU(2)\) on \(\mathbf{C}^2\setminus \{0\}\).     \(v ^{k} _{(l,m-l)}\) are harmonic polynomials on $\mathbf{C}^2$;   
\(\,\Delta v ^{k} _{(l,m-l)}=0\,\),  where \(\,\Delta
= \frac{\partial ^2}{\partial z _1 \partial \Bar{z}_1} + \frac{\partial ^2}{\partial z _2 \partial \Bar{z}_2}
\).      It holds that 
\begin{equation}\label{complexconj}
\overline{ v^k_{(l,m-l)}}=(-1)^{m-l-k}\frac{k!}{(m-k)!}v^{m-k}_{(m-l,l)}\,.
\end{equation}

 In \cite{Ko0,Ko1} we saw that  
  \(\left\{\,\frac{1}{\sqrt{2}\pi}
  \sqrt{\frac{(m+1)(m-k)}{l!(m-l)!k!}}v ^{k} _{(l,m-l)}\, ;  \,m = 0, 1, \cdots,\,  0\leq k,l\leq m\,\right\} \)  forms a \(L^2(S^3)\)-complete orthonormal system of the space of harmonic polynomials.    
The similar assertions hold for   \(\left\{\,\frac{1}{\sqrt{2}\pi}\sqrt{\frac{(m+1)(m-k)}{l!(m-l)!k!}}w ^{k} _{(l,m-l)}\, ;  \,m = 0, 1, \cdots,\,  0\leq k,l\leq m\,\right\} \).

Let \(H\) be the space of harmonic polynomials on \(S^3\).   For each pair \((m,l)\), \(0\leq l\leq m\),  let 
  \(H_{(m,l)}\) be the linear subspace generated by the vectors \(\{ v ^{k} _{(l,m-l)}\,; 0\leq k\leq m+1\}\,\).   
  \(H\) is the direct sum of \( H_{(m,l)}\), \(0\leq m,\,0\leq l\leq m\), and each \( H_{(m,l)}\) gives a \((m+1)\)-dimensional right representation of \(su(2)\) with the highest weight \(\frac m2\).    Similarly 
 the subspace  \(H^{\dag}_{(m,l)}=\{w^{k} _{(l,m-l)}\,; 0\leq k\leq m+1\}\) gives a \((m+1)\)-dimensional left representation of \(su(2)\) with the highest weight \(\frac m2\).

In Lemma 4.1 of \cite{Ko0} we proved the  following product formula for the harmonic polynomials \(v^k_{(a.b)}\,\):   
\begin{proposition}\label{productformula}
\begin{equation}\label{multiv}
v ^{k_1} _{(a_1,b_1)} v ^{k_2} _{(a_2,b_2)}=\sum_{j=0}^{a_1+a_2+b_1+b_2}C_j\vert z\vert^{2j}\,v ^{k_1+k_2-j} _{(a_1+a_2-j,\,b_1+b_2-j)} \, ,
\end{equation}
for some rational numbers  \(C_j=C_j(a_1,a_2,b_1,b_2,k_1,k_2)\).     
\end{proposition}

\begin{proposition}~~~\\
The space of harmonic polynomials \(H\) on \(S^3\) is given a graded \(\mathbf{C}\)-algebra structure.  
\end{proposition}
In fact, let \(k=k_1+k_2\), \(a=a_1+a_2\) and \(b=b_1+b_2\).   Restricted to \(S^3\), the harmonic  polynomial \(v^k_{(a,b)}\) is equal to  a constant multiple of 
\(\,v^{k_1}_{(a_1,b_1)}\cdot v^{k_2}_{(a_2,b_2)}\) modulo a linear combination of polynomials \(v^{k-j}_{(a-j,b-j)}\,\), \(1\leq j\leq min(k,a,b)\).   So the set of harmonic polynomials becomes a graded \(\mathbf{C}\)-algebra.

\subsubsection{Harmonic spinors on \(S^3\subset\mathbf{C}^2\)}

Now we introduce a basis of the space of even harmonic spinors.   
\begin{definition}~~~
For $m = 0,1,2, \cdots ; l = 0,1, \cdots , m$ and $k=0,1, \cdots , m+1$, we put 
\begin{eqnarray}\label{basespinor}
\phi^{+(m,l,k)} (z) &=& \sqrt{\frac{(m+1-k)!}{k!l!(m-l)!}} \begin{pmatrix} k v^{k-1} _{(l, m-l)}\\ \\ -v^{k}_{(l, m-l)} \end{pmatrix}\nonumber,\\ \notag \\
\phi^{-(m,l,k)} (z) &=& \sqrt{ \frac{(m+1-k)!}{k!l!(m-l)!}} \left(\frac{1}{\vert z\vert^2}\right)^{m+2}\begin{pmatrix} w^{k} _{(m+1-l,l)}\\ \\ w^{k}_{(m-l,l+1)} \end{pmatrix}.
\end{eqnarray}
\end{definition}

{\bf Examples}

\begin{eqnarray*}
\phi^{+(0,0,1)}&=\left( \begin{array}{c}1\\0 \end{array} \right) \,,\quad
 \phi^{+(0,0,0)}&=\left( \begin{array}{c}0\\-1 \end{array} \right)\,,\quad
\phi^{+(1,0,0)}=\left(\begin{array}{c}0\\-\sqrt{2}\,z_2  \end{array} \right)\,\\[0.2cm]
 \phi^{+(1,0,1)}&=\left(\begin{array}{c}z_2\\ -\overline z_1 \end{array} \right)\,,\quad 
\phi^{+(1,1,1)}&=\left(\begin{array}{c}z_1\\ \overline z_2 \end{array}\right)\,,\quad
 \phi^{+(2,0,0)}=\left(\begin{array}{c}0\\-\sqrt{3}\,z_2^2 \end{array}\right)\,.
 \end{eqnarray*}
 \begin{eqnarray*}
 \phi^{-(0,0,0)}&=\frac{1}{|z|^4}\left( \begin{array}{c}z_2\\\overline{z}_1 \end{array} \right) \,,\quad
 \phi^{-(0,0,1)}&=\frac{1}{|z|^4}\left( \begin{array}{c}-z_1\\\overline{z}_2 \end{array} \right)\,,\,\\[0.2cm]
 \phi^{-(1,0,0)}&=\frac{\sqrt{2}}{|z|^6}\left(\begin{array}{c}z_2^2\\ 2z_2\overline{z}_1  \end{array} \right)\,,\,
 \phi^{-(1,0,1)}&=\frac{-2}{|z|^6}\left(\begin{array}{c}z_1z_2\\  |z_1|^2-|z_2|^2 \end{array} \right)\,
.
\end{eqnarray*}

\begin{proposition}\cite{Ko1}~~~
\begin{enumerate}
\item
\(\phi^{+(m,l,k)}\)  is a harmonic spinor on \(\mathbf{C}^2\) and  \(\phi^{-(m,l,k)}\) is a harmonic spinor on \(\mathbf{C}^2 \backslash \{0\}\) that is regular  at infinity.   
\item
On $S^3 = \{|z| = 1\}$ we have:
 \begin{equation*}
 \Do \phi^{+(m,l,k)} =\frac{m}{2}\, \phi^{+(m,l,k)} \,,\qquad 
 \Do \phi^{-(m,l,k)} = -\frac{m+3}{2}\, \phi^{-(m,l,k)} \, .
 \end{equation*}
 \item
The eigenvalues of $\,\Do$ are 
 \begin{equation*}
 \frac{m}{2} \,,\quad - \frac{m+3}{2} \, ; \quad  m = 0, 1, \cdots,
 \end{equation*}
and the multiplicity of each eigenvalue is equal to $(m+1)(m+2)$.
\item
The set of eigenspinors
 \begin{equation*}
 \left\{ \frac{1}{\sqrt{2}\pi }\phi^{+(m,l,k)}, \quad \frac{1}{\sqrt{2}\pi }\phi^{-(m,l,k)} \,
 ; \quad m = 0, 1, \cdots , \,  0\leq l\leq  m,\, 0\leq k\leq m+1\right\}
 \end{equation*}
forms a complete orthonormal system of $L^2 (S^3, S^+)$.
\end{enumerate}
\end{proposition}

\begin{remark}~~~
On \(\mathbf{C}^2\setminus\{0\}\) we have 
\[\Do\phi^{+(m,l,k)}=\frac{m}{2\vert z\vert}\,\phi^{+(m,l,k)}\,,\quad 
\Do\phi^{-(m,l,k)}=\,-\frac{m+3}{2\vert z\vert}\,\phi^{-(m,l,k)}\,.\]
\end{remark}

The Bergman kernel on the space of harmonic spinors is associated to the basis spinors \(\left\{\,\phi^{+(m,l,k)}\,,\,\phi^{-(m,l,k)}\,\right\}\).    It is given by the following formula:
\begin{equation*}
B(\,z\,,\,\zeta\,)\,=\, \,\frac{1}{2\pi^2}\,
 \sum_{m,l,k}\,\phi^{+(m,l,k)}(\zeta)\otimes \phi^{+(m,l,k)}(z)\,+\phi^{-(m,l,k)}(\zeta)\otimes \phi^{-(m,l,k)}(z)\,,
 \end{equation*}
for \( z\in G,\, \zeta\in \partial G\).  

The Cauchy kernel ( fundamental solution )  of the half Dirac operator \(D: C^{\infty}(\mathbf{C}^2,\,S^+)\longrightarrow 
 C^{\infty}(\mathbf{C}^2,\,S^-)\) is given by 
\begin{equation*}
K^{\dag}(z,\zeta)=\frac{1}{\vert\zeta-z\vert^3}\gamma_-(\zeta-z)\,: C^{\infty}(\mathbf{C}^2,\,S^-)\longrightarrow 
C^{\infty}(\mathbf{C}^2,\,S^+)
,\quad |z|< |\zeta|\,.
\end{equation*}

 In \cite{Ko2} we showed that the the Bergman kernel coincides with the Cauchy kernel:
  \begin{equation*}
 K^{\dag}(z,\zeta)\,=\,B(\,z\,,\,\zeta\,)\quad\mbox{ for \( z\in G,\, \zeta\in \partial G\)}
 \end{equation*}
 for any bounded domain \(G\subset\mathbf{C}^2\).

We have the following integral representation of spinors:
\begin{theorem}\cite{Ko1}~~~
Let \(G\) be a domain of \(\mathbf{C}^2\) and let \(\varphi\in C^{\infty}(\overline G,\,S^+)\).   
We have 
\begin{equation*}
\varphi(z)=-\frac{1}{2\pi^2}\int_G\,K^{\dag}(z,\zeta)D\varphi(\zeta)dv(\zeta)+\frac{1}{2\pi^2}\int_{\partial G}\,K^{\dag}(z,\zeta)(\gamma_+\varphi)(\zeta)d\sigma(\zeta)\,,\quad z\in G.
\end{equation*}
\end{theorem}

\section{ Laurent polynomial type harmonic spinors  on\\  \(\mathbf{C}^2\setminus\{0\}\) and the associated 2-cocycle}

\subsection{Algebra of  Laurent polynomial type harmonic spinors on \(\mathbf{C}^2\setminus\{0\}\)
}

\subsubsection{ Quarternion trace and Quarternion residue}
  
We have the Laurent expansions of harmonic spinors, that is, 
a harmonic spinor \(\varphi\) on 
 \(\mathbf{C}^2\setminus\{0\}\) has an expansion by the basic spinors \(\{\,\phi^{\pm(m,l,k)}\}_{m,l,k}\,\):
 \begin{equation}
 \varphi(z)=\sum_{m,l,k}\,C_{+(m,l,k)}\phi^{+(m,l,k)}(z)+\sum_{m,l,k}\,C_{-(m,l,k)}\phi^{-(m,l,k)}(z),\label{Laurentspinor}
 \end{equation}
which is uniformly convergent on any compact subset of \(\mathbf{C}^2\setminus\{0\}\), \cite{Ko2}. 
The coefficients \(C_{\pm(m,l,k)}\) are given by the following formula:
\begin{equation}\label{coefficient}
C_{\pm(m,l,k)}=\,\frac{1}{2\pi^2}\int _{S^3}\, \langle \varphi,\,\phi^{\pm(m,l,k)}\rangle\,d\sigma,
\end{equation}
where \(\langle\,,\,\rangle\) is the inner product of \(S^+\).    

\begin{definition}
We call the series (\ref{Laurentspinor})  {\it a spinor of  Laurent polynomial type} on \(\mathbf{C}^2\setminus \{0\}\) if only finitely many coefficients \(C_{- (m,l,k)}\)  are non-zero .   The space of spinors of Laurent polynomial type is denoted by   \(\mathbf{C}[\phi ^{\pm}] \).  
\end{definition} 

The {\it quarternion trace} of \(\varphi\) is defined by 
\begin{equation}\label{qtrace1}
qTr\,\varphi\,=\,\left(\begin{array}{c}C_{+(0,0,1)}\\ -C_{+(0,0,0)}\end{array}\right)\,.
\end{equation}
By the quarternion notation it becomes 
\[qTr\,\varphi\,=\,C_{+(0,0,1)}\,+\,C_{+(0,0,0)}\,j\,.\]
We have 
\begin{equation}
qTr\,\varphi\,=\,\frac{1}{2\pi^2} \int_{\vert\zeta\vert=1}\varphi(\zeta)\sigma(d\zeta).
\end{equation}

\begin{proposition}\label{qtrace2}
\begin{equation}
q\,Tr\,(\phi_1\cdot\,\phi_2)\,=\,q\,Tr\,(\phi_2\cdot\,\phi_1)\,.
\end{equation}
\end{proposition}
In fact, for
\[\phi_i(z) = \sum_{m,l,k}\,C^{(i)}_{+(m,l,k)}\phi^{+(m,l,k)}(z)+\sum_{m,l,k}\,C^{(i)}_{-(m,l,k)}\phi^{-(m,l,k)}(z)\,,\quad i=1,2,\]
we have 
\[ q\,Tr( \phi_1\cdot\phi_2)=\,
\sum_m\,C^{(1)}_{+(m+3,\,,\,)}C^{(2)}_{-(m,\,,\,)}+
C^{(2)}_{+(m+3),\,,\,)}C^{(1)}_{-(m,\,,\,)}\,=\,q\,Tr( \phi_2\cdot\phi_1). \]

\begin{definition}\label{qres}
Let \(\varphi\) be a Laurent polynomial type spinor (\ref{Laurentspinor}).
We call the vector  \(\left(\begin{array}{c}C_{-(0,0,1)}\\-C_{-(0,0,0)}\end{array}\right)\)  the {\it quarternion residue of \(\,\varphi\) at \(z=0\) } and we denote it by \(qRes\,\varphi\).
\end{definition}
By the quarternion notation it becomes 
\[qRes\varphi\,=\,C_{-(0,0,1)}\,+\,C_{-(0,0,0)}\,j\,.\]
Since \(\phi^{-(m,l,k)}(z)\sim\,O(\vert z\vert^{-m-3})\),   \(\,qRes\,\varphi\)  is the coefficient of \(O(\vert z\vert^{-3})\) in the expansion of \(\varphi\).    We have the following 
\begin{proposition}\label{intrepres}
\begin{equation}
qRes\,\varphi=\frac{1}{2\pi^2}\int_{\vert\zeta\vert=1}\gamma_+(\zeta)\varphi(\zeta)\sigma(d\zeta).
\end{equation}
\end{proposition}
\begin{proof}
From (\ref{Laurentspinor}) and (\ref{coefficient}) we have
\begin{equation}
 \varphi(z)=\sum_{m,l,k}\,C_{+(m,l,k)}\phi^{+(m,l,k)}(z)+\sum_{m,l,k}\,C_{-(m,l,k)}\phi^{-(m,l,k)}(z),
 \end{equation}
with 
\begin{equation*}
C_{-(m,l,k)}=
\frac{1}{2\pi^2}\int_{\vert z\vert=1}\langle \varphi(z),\,\phi^{-(m,l,k)}(z)\rangle \sigma(dz).
\end{equation*}
Then 
\begin{eqnarray*}
C_{-(0,0,0)}&=&
\frac{1}{2\pi^2}\int_{\vert z\vert=1}
\langle (\gamma_+\varphi)(z),\,(\gamma_+\phi^{-(0,0,0)})(z)\rangle \sigma(dz)\\
&=&\frac{1}{2\pi^2}\int_{\vert z\vert=1} (\gamma_+\varphi )_2(z)\sigma(dz),
\end{eqnarray*}
since \(\gamma_+\varphi^{-(0,0,0)}=\left(\begin{array}{c}0\\1\end{array}\right)\), where \((\phi)_2\) is the 2nd component of \(\phi\).   Similarly we have 
\[
-\,C_{-(0,0,1)}
=\frac{1}{2\pi^2}\int_{\vert z\vert=1} (\gamma_+\varphi )_1(z)\sigma(dz).\]
\end{proof}

\begin{proposition}\label{resofproduct}
\begin{equation}
qRes\,(\phi_1\cdot\,\phi_2)\,=\,qRes\,(\phi_2\cdot\,\phi_1)\,.
\end{equation}
\end{proposition}
In fact, for
\[\phi_i(z) = \sum_{m,l,k}\,C^{(i)}_{+(m,l,k)}\phi^{+(m,l,k)}(z)+\sum_{m,l,k}\,C^{(i)}_{-(m,l,k)}\phi^{-(m,l,k)}(z)\,,\quad i=1,2,\]
the coefficient of the 
\(O(|z|^{-3})\)- terms in \(\phi_1\cdot\phi_2\)  is  
\[ qRes ( \phi_1\cdot\phi_2)=\,
\sum_m\,C^{(1)}_{+(m,\,\,)}C^{(2)}_{-(m,\,\,)}+
C^{(2)}_{+(m,\,\,)}C^{(1)}_{-(m,\,\,)}\,=\,qRes( \phi_2\cdot\phi_1). \]

We remark the following equivalence of the actions of the boundary Dirac operator and the normal derivation.   For a Laurent polynomial type spinor
\[\varphi(z) = \sum_{m,l,k}\,C_{+(m,l,k)}\phi^{+(m,l,k)}(z)+\sum_{m,l,k}\,C_{-(m,l,k)}\phi^{-(m,l,k)}(z)\,,\]
we have 
\[\frac{\partial}{\partial n}\varphi(z) =
\frac{1}{2|z|}\,\sum_{m,l,k}\,mC_{+(m,l,k)}\phi^{+(m,l,k)}(z)-\frac{1}{2|z|}\,\,\sum_{m,l,k}\,(m+3)C_{-(m,l,k)}\phi^{-(m,l,k)}(z)\,,\]
and 
\[
\Do\varphi(z) =
\frac{1}{2|z|}\,\sum_{m,l,k}\,mC_{+(m,l,k)}\phi^{+(m,l,k)}(z)-\frac{1}{2|z|}\,\,\sum_{m,l,k}\,(m+3)C_{-(m,l,k)}\phi^{-(m,l,k)}(z)\,.\]
Hence we have the following
\begin{proposition}~~~
For a Laurent polynomial type spinor \(\varphi\), 
\begin{equation}\label{bDequalnormal}
\Do\varphi(z) \,=\,\frac{\partial}{\partial n}\varphi(z) .
\end{equation}
\end{proposition}

Since there is no term of order \(O(\frac{1}{|z|^3})\) in the above expansions of 
\(\frac{\partial}{\partial n}\varphi\) and \(\Do\varphi\) we have the following
\begin{lemma}~~~
For a Laurent polynomial type spinor \(\varphi\), 
\begin{equation}\label{resDandn}
qRes\,(\,\Do \,\varphi\,)\,=\,0\,,\qquad 
 qRes\,(\frac{\partial}{\partial n}\varphi)=0\,.
 \end{equation}
 \end{lemma}

\subsubsection{ Algebra of currents on \(S^3\)}

The space \( \mathbf{C}[\,\phi^{\pm}\,]\) of spinors of Laurent polynomial type over \(\mathbf{C}^2\setminus \{0\}\) is not an algebra.   
\begin{definition}~~~
 The subalgebra of \(C^{\infty}(\mathbf{C}^2\setminus\{0\},\,S^+)\) that is generated by \( \mathbf{C}[\,\phi^{\pm}\,]\) is called the {\it algebra of currents} on \(\mathbf{C}^2\setminus\{0\}\) and is denoted by  \(\mathcal{L}\) .
\end{definition}

Every spinor \(\varphi\,\in \mathcal{L}\,\) is written as a \(\mathbf{C}\)-linear combination of basic spinors 
\(\,\phi_1\cdot \phi_2\cdot\,\cdots\,\phi_r\,\) for \(\,\phi_i=\,\phi^{\pm(m_i,l_i,k_i)}\,,\,0\leq m_i, 0\leq l_i\leq m_i\,,\,0\leq k_i\leq m_i+2,\,1\leq i\leq r 
\).

 We note that a spinor of \(\mathcal{L}\) is not necessarily a harmonic spinor on \(\mathbf{C}^2\setminus \{0\}\).   
For example, 
 \[
 \phi^{+(1,0,0)}\cdot \phi^{-(0,0,0)}\,=\,\sqrt{2}\,\left(\begin{array}{c}0\\-z_2\end{array}\right)\,\cdot \frac{1}{\vert z\vert^4}\,
 \left(\begin{array}{c}z_2\\ \overline{z}_1 \end{array}\right)\,
\,=\, \frac{\sqrt{2}}{\vert z\vert^4}\,\left(\begin{array}{c}\overline{z}_1\overline{z}_2\\ - z_2^2\end{array}\right)\, \]
 is not harmonic over \(\,\mathbf{C}^2\setminus \{0\}\).     
 
\begin{lemma}\label{multitosum}~~~
The product \(\,\phi^{\pm(m_1,\l_1,,k_1)}\cdot\phi^{\pm(m_2,\l_2,,k_2)}\) 
 of two spinors \(\phi^{\pm(m_1,\l_1,,k_1)}\) and \(\phi^{\pm(m_2,\l_2,,k_2)}\) is written by a linear combination of spinors  
 \(\vert z\vert^p\,\phi^{\pm(m,l,k)}\,\),  \(0\leq\,,m,\,l,\,k\,,\,p\,\leq m_1+m_2+2
\).
\end{lemma}
 \begin{proof}~~~
 The multiplication of spinors is defined in (\ref{spinormultip}).   
 By virtue of Proposition \ref{productformula} and the formula (\ref{wtov})  and  (\ref{complexconj}) we see that the  product \(\phi^{\pm (m_1,l_1,k_1)}\cdot \phi^{\pm (m_2,l_2,k_2)}\) of two spinors 
 \(\phi^{\pm (m_1,l_1,k_1)}\) and \( \phi^{\pm (m_2,l_2,k_2)}\)  is decomposed into a linear sum of the following spinors; 
  \[ \vert z\vert^{ p}\,\left(\begin{array}{c}v^k_{(l,m-l)}\\0\end{array}\right),\quad \vert z\vert^{ p}\,\left(\begin{array}{c}0\\v^{k}_{(l,m-l)}\end{array}\right)\,,\]
where   \(  k\,, \,l,\, m\,\) and \(\,p\,\) are numbers that are smaller  than or equal to \( m_1+m_2+2\,\).   
     On the other hand a spinor of the form 
 \(|z|^p\left(\begin{array}{c}v^k_{(l,m-l)}\\0\end{array}\right)\) or  \(|z|^p\left(\begin{array}{c}0\\v^{k+1}_{(l,m-l)}\end{array}\right)\)  is written as a linear combinations of \(|z|^{p^{\prime}}\phi^{+(m,l,k)}\) and \(|z|^{p^{\prime}}\phi^{-(m-1,k-1,l)}\).   For example,
 \[\left(\begin{array}{l}
 v^k_{(l,m-l)}\\0\end{array}\right)=A\phi^{+(m,l,k+1)}+B\phi^{-(m-1,k,l)}\,,\]
 with \( A= \sqrt{\frac{(k+1)!l!(m-l)!}{(m-k)!}}\frac{1}{m+1}\,\) and \(
 B=
 (-1)^l\frac{\sqrt{m-l}}{m+1}\,\).
 So any product  \(\phi^{\pm (m_1,l_1,k_1)}\cdot \phi^{\pm (m_2,l_2,k_2)}\) is written as a linear combination of \(\left\{\,|z|^p \phi^{\pm(m_1+m_2-n\,,\,\cdot\,,\,\cdot\,)}\,; \, 1\leq n,\, p\,\leq m_1+m_2+2\,\right\}\).   
 \end{proof}
We have seen that 
  the \(\mathbf{C}\)-algebra \(\mathcal{L}\,\) has the basis
\begin{equation}
\left\{\,\vert z\vert^p\,\phi^{\pm(m,l,k)}\,;\,0\leq p\,, \,0\leq m\,, 0\leq l\leq m,\,0\leq k\leq m+1\,\right\}.  
\end{equation}

 Moreover, we see from Proposition \ref{productformula}  that  any \(v^{k}_{(l,m-l)}\) is written by  a linear combination of \(\,|z|^{2r}\,v^{k_1}_{(l_1,m_1-l_1)}v^{k_2}_{(l_2,m_2-l_)}\) for \(1\leq r\leq m,\, \,0\leq m_1+m_2\leq m-1\,\),  \(0\leq l_1+l_2\leq l\) and \(0\leq k_1+k_2\leq k\), so that any  \(\,\phi^{\pm(m,l,k)}\)  is written by a linear combination of the products  \(\,|z|^{2r}\,\phi^{\pm(m_1,l_1,k_1)}\cdot \phi^{\pm(m_2,l_2,k_2)}\) for \(0\leq r\leq m-1,\,0\leq m_1+m_2\leq m-1\,\),  \(0\leq l_1+l_2\leq l\) and \(0\leq k_1+k_2\leq k\),  so that  \( \mathcal{L}\) is a graded algebra generated by the four spinors 
\(I=\phi^{ +(0,0,1)}\,,\,J=-\phi^{ +(0,0,0)}\,,\,\kappa=\phi^{ -(0,0,1)}\,,\,\lambda=\phi^{-(0,0,0)}\,\). 
    Thus we have proved the following:
\begin{theorem}\label{Laurentalgebra}~~~
The algebra of current \(\mathcal{L}\) becomes a graded algebra  with the generators given by the  spinors:
\begin{eqnarray*}
I &=\,\phi^{ +(0,0,1)}=\left(\begin{array}{c}1\\0\end{array}\right),\qquad
J &=-\,\phi^{ +(0,0,0)}=\left(\begin{array}{c}0\\1\end{array}\right),\\[0.2cm]
\kappa &= \phi^{ -(0,0,1)}=\frac{1}{\vert z\vert^4}\left(\begin{array}{c}-z_1\\\overline z_2\end{array}\right)\,,\quad
\lambda &=\phi^{-(0,0,0)}=\frac{1}{\vert z\vert^4}\left(\begin{array}{c}z_2\\ \overline z_1.\end{array}\right)\,.
\end{eqnarray*}
\end{theorem}
 
 {\bf Example}.  \\
As we saw in the above,  
\(\phi^{+(1,0,0)}\cdot \phi^{-(0,0,0)}\,\) is not a harmonic spinor.   But
\[|z|^4\, \phi^{+(1,0,0)}\cdot \phi^{-(0,0,0)}\,=\frac{1}{\sqrt{2}}\,\phi^{+(2,0,2)}+\sqrt{\frac{2}{3} }\,\phi^{+(2,0,0)}
 \in \mathbf{C}[\,\phi^{\pm}]\,\subset \mathcal{L},\]
and their restrictions to the boundary \(S^3=\{|z|=1\}\) are equal. 

 We put 
\begin{eqnarray}
\mathcal{K}&=&\{\phi\in\mathcal{L}\,;\,\phi=\left(\begin{array}{c}f\\0\end{array}\right),\quad\mbox{ for  } f\in C^{\infty}(\mathbf{C}^2\setminus\{0\},\,\mathbf{R})\},\\[0.2cm]
\mathcal{J}&=&\{\phi\in\mathcal{L}\,;\,\phi=\left(\begin{array}{c}\sqrt{-1}f\\ g+\sqrt{-1}h\end{array}\right),\quad\mbox{ for  } f, g, h \in C^{\infty}(\mathbf{C}^2\setminus\{0\},\,\mathbf{R})\}.
\end{eqnarray}
As we noticed at the beginning of Theorem\ref{Laurentalgebra}
  a spinors of the form 
\(\left(\begin{array}{c}v^k_{(l,m-l)}\\0\end{array}\right)\) or  \(\left(\begin{array}{c}0\\v^{k+1}_{(l,m-l)}\end{array}\right)\)  is written by a linear combinations of \(\phi^{+(m,l,k+1)}\) and \(\phi^{-(m-1,k,l)}\).     This fact and the relation (\ref{complexconj}) yield  that \(\mathcal{K}\) and 
\(\mathcal{J}\) are \(\mathbf{R}\)-linear suspaces of \(\mathcal{L}\) and    
 \(\mathcal{L}\)  is decomposed into the direct sum;  
  \(
 \mathcal{L}=\mathcal{K}\,\oplus\,\mathcal{J}\,\).     Evidently \(\mathcal{K}\) is a commutative subalgebra of \(\mathcal{L}\,\).    
 By the induced Lie algebra structure on \(\mathcal{L}\) we see that \(\mathcal{J}\) is an ideal of \(\mathcal{L}\) and 
\begin{equation}\label{Kcommute}
[\mathcal{K}\,,\,\mathcal{L}\,]=0\,,\quad 
 [\mathcal{L}\,,\,\mathcal{L}\,]\,=\,\mathcal{J}\,.
 \end{equation}
The quarternion trace (\ref{qtrace1}) on \(\mathcal{L}\) is non-degenerate and the restriction to  \(\mathcal{K}\) is also non-degenerate.  

\subsection{Radial derivative on \(\mathcal{L}\)}

The vector fields \(\theta_i\), \(i=1,2,3\,\), and the boundary Dirac operator \(\Do\,\), (\ref{tgDirac}),  are tangent to \(S^3=\{|z|=1\}\,\); 
 \[\theta_i\,|z|\,=0\,, \,i=1,2,3\,,\qquad \Do(\,|z|\,\varphi )\,=\,\vert z\vert\,\Do\,\varphi\,.\]
The normal vector field on \(\mathbf{R}^4\setminus\{0\}\) and  the normal derivation are defined in (\ref{normal}) and  (\ref{normalder}):
   \begin{equation*}
   \,{\mathbf n}\,=\nu+\overline{\nu},\qquad \frac{\partial}{\partial n}=\frac{1}{2|z|}{\mathbf n}\,.
 \end{equation*}  
  The action of  the normal vector field \({\bf n}\) on a spinor \(\varphi=\left(\begin{array}{c}u\\[0.2cm] v\end{array}\right)\in C^{\infty}(\mathbf{R}^4\setminus\{0\},\,S^+)\,\) is defined by 
 \begin{equation*}\label{radial}
{\mathbf n}\,\,\varphi\,=\,\left(\begin{array}{c}{\mathbf n}\,\,u\\[0.2cm] {\mathbf n}\,\,v\end{array}\right).
\end{equation*}

\begin{proposition}\label{derivationofL}~~
\begin{enumerate}
\item
\begin{equation}
{\mathbf n}( \phi_1 \cdot \phi_2 ) =\,(\,{\mathbf n}\,\phi _1) \cdot \phi_2 + \phi _1 \cdot (\,{\mathbf n}\,\phi _2)\,.  
\label{leibnitz}
\end{equation}
So  \(\,{\mathbf n}\,\) gives a derivation of the algebra \(\,\mathcal{L}\,\).    In particular \(\frac{\partial}{\partial n}\) is a derivation of \(\,\mathcal{L}\,\).
\item
\begin{equation}
{\mathbf  n}\,\,\phi^{+(m,l,k)} =\, m\, \phi^{+(m,l,k)}\,,\quad 
{\mathbf n}\,\,\phi^{-(m,l,k)}=\,- (m+3)\, \phi^{-(m,l,k)}\,.\label{normalderofbase}
\end{equation}
 \item
 If \(\varphi\) is a spinor of Laurent polynomial type on \(\mathbf{C}^2\setminus\{0\}\,\):
\begin{equation*}
 \varphi(z)=\sum_{m,l,k}\,C_{+(m,l,k)}\phi^{+(m,l,k)}(z)+\sum_{m,l,k}\,C_{-(m,l,k)}\phi^{-(m,l,k)}(z)\,, 
 \end{equation*}
  \({\bf n}\,\varphi\) is also a  spinor of Laurent polynomial type:
\begin{equation}
{\bf n}\,\varphi\,=\sum_{m,l,k}\,m\,C_{+(m,l,k)}\phi^{+(m,l,k)}(z)-\sum_{m,l,k}\,(m+3)\,C_{-(m,l,k)}\phi^{-(m,l,k)}(z)\,, 
 \end{equation}
 \end{enumerate}
\end{proposition}
\begin{proof}~~~
The formula (\ref{normalderofbase}) follows from the definition (\ref{basespinor}).      
\end{proof}

 \subsection{Homogeneous decomposition of \(\mathcal{L}\)}

Let  \(\mathcal{L}[m]\) be the subspace of \(\mathcal{L}\) consisting of those spinors 
 that are of homogeneous degree \(m\): 
  \[\varphi(z)=|z|^m\varphi(\frac{z}{|z|})\,. \]
\(\mathcal{L}[m]\) is spanned by the sum of the spinors  
\(\varphi=\,\phi_1\cdots\phi_n \) such that each \(\phi_i\) is equal to  \(\phi_i=\phi^{+(m_i,l_i,k_i)}\) or  \(\,\phi_i=\phi^{-(m_i,l_i,k_i)}\) , where \(m_i\geq 0\,\), \(\,0\leq l_i\leq m_i+1, \,0\leq k_i\leq m_i+2\), and such that 
\[m= \sum_{i:\,\phi_i=\phi^{+(m_i,l_i,k_i)}}\,m_i\,\,-\,\sum_{i:\,\phi_i=\phi^{-(m_i,l_i,k_i)}}\,(m_i+3).\]
It holds that \({\bf n}\,\varphi=\,m\varphi\), so the eigenvalues of \({\bf n}\,\) on \(\mathcal{L}\) are \(\left\{m;\,m\in\mathbf{Z}\,\right\}\) and 
\(\mathcal{L}[m]\) is the space of eigenspinors for the eigenvalue \(m\).    

{\bf Example }
\[\phi=\phi^{+(2,00)}\cdot\phi^{-(0,00)}\,\in \mathcal{L}[-1]\,,\quad\mbox{ and }\quad {\bf n}\phi\,=\,-\,\phi\,.\]
Here we note that \(-1\) is not an eigenvalue of the tangential Dirac operator \(\Do\).  
 
On the other hand \(\varphi\in\mathcal{L}\) is also expressed as a linear sum of the spinors of the form \(|z|^p\phi^{\pm(m,l,k)}\).    For this component we have 
\begin{eqnarray}\label{ncoef}
{\bf n}\,|z|^p\phi^{+(m,l,k)}\,&=&\,(p+ m)|z|^p\phi^{+(m,l,k)}\,,\nonumber \\[0.2cm] 
{\bf n}\,|z|^p\phi^{-(m,l,k)}\,&=&\,(p-m-3)|z|^p\phi^{-(m,l,k)}\,
\end{eqnarray}
so that  \(|z|^p\phi^{+(m,l,k)}\) is of homogeneous order \((p+ m)\) and \(|z|^p\phi^{-(m,l,k)}\) is of homogeneous order \((p-m-3)\).
The normal vector field \(\,{\bf n}\,\) acting on \(\mathcal{L}\,\) preserves the homogeneous degree.     Thus we have the eigenspace decomposition of \(\mathcal{L}\) by the normal derivative \({\bf n}\,\):
\begin{equation}\label{homog}
\mathcal{L}=\, \bigoplus_{m\in \mathbf{Z}}\,\mathcal{L}[m]\,\end{equation}

Let 
\begin{equation}
H_{[m]}\,=\left\{\varphi\in \mathcal{L}[m] ;\quad D\varphi=0\,\right\}.
\end{equation}

\begin{proposition}~~~
\begin{enumerate}
\item
\begin{equation*} H_{-1}=H_{-2}=0\,.
\end{equation*}
\item
\begin{equation}\label{harmonicdecomp}
\mathcal{L}[m]=\,\sum_{p\in \mathbf{Z}}\,|z|^{p}\,H_{m-p}
\end{equation}
\end{enumerate}
\end{proposition}
So any \(m\)-homogeneous spinor of \(\mathcal{L}\) is equal to a finite sum \(\sum_k|z|^k\phi_{(m-k)}\) with 
\(\phi_{(m-k)}\in H_{(m-k)}\).

{\it Proof}~~~
Since \(\phi(z)=|z|^m\phi(\frac{z}{|z|})\) for  \(\phi\in \mathcal{L}[m]\) , we have 
\[ D\phi(z)=
|z|^{m-1}\gamma_+(z)(\frac{m}{2}\phi-\Do\phi\,)(\frac{z}{|z|}).
\]
For \(\phi\in H_m\) we have 
\(\,\Do\phi=\frac{m}{2}\phi\) 
on \(S^3=\{|z|=1\}\).   
Since the eigenvalues of  the boundary Dirac operator \(\Do\) are \(\frac{m}{2}\) for \(m\geq 0\) and \(m\leq -3\), 
we have 
\[H_{-2}=H_{-1}=0.\]
We shall prove that 
\begin{equation}\label{intersection}
H_k\cap\,|z|\mathcal{L}[k-1]=0,\quad\forall k\in \mathbf{Z}\,,
\end{equation}
then by counting the dimensions we have
 \begin{equation}
 \mathcal{L}[k]=H_{k}\oplus |z|\mathcal{L}[k-1] .   
 \end{equation}
 It implies 
\begin{eqnarray*} 
|z|^{-k}\mathcal{L}[m+k]&=&|z|^{-k}H_{m+k}\oplus |z|^{-k+1}H_{m+k-1}\cdots\oplus
 |z|^{m}H_0\oplus\\[0.2cm]
 &\,&\qquad  |z|^{m+3}H_{-3}\oplus \cdots\oplus |z|^{k-1}H_{m-k+1}\oplus |z|^{k}\mathcal{L}[m-k].
 \end{eqnarray*}
 for any \(k\).     For the proof of (\ref{intersection}) we suppose that \(|z|\phi\in H_k\cap\,|z|\mathcal{L}[k-1]\) with \(\phi\in \mathcal{L}[k-1]\).    Let \(q\) be the biggest  
 number such that \(|z|\phi=|z|^q\psi\) for a \(\psi\in \mathcal{L}[k-q]\).   Then 
 \[D(|z|^q\psi)=\frac{q}{2}|z|^{q-1}\gamma_+\psi+|z|^qD\psi=0.\]
We have \(|z|\phi=|z|^{q+1}\alpha\) with \(\alpha=-\frac{2}{q}\gamma_-D\psi\in \mathcal{L}[k-q-1]\).    This is a contradiction.   
 \hfill\qed

\begin{corollary}
Let \(m=0,1,\cdots\).
\begin{enumerate}
\item
\(H_{m}\) has the base \(\{\phi^{+(m,l,k)};\,l=0,1,\cdots,m,\, k=0,\cdots m+2\}\)  .  
\item
\(H_{-(m+3)}\) has the base \(\{\phi^{-(m,l,k)};\,l=0,1,\cdots m,\, k=0,\cdots m+2\}\) 
.
\end{enumerate}
\end{corollary}

\begin{proposition}
For each \(m\in \mathbf{Z}\) there is a linear isomorphism
 \begin{equation}
 \mathcal{L}[m]\,\simeq\,\mathbf{C}[\phi^{\pm}]\,.\label{brestr}
 \end{equation}  
 \end{proposition}
 In fact,  we have the following isomorphism
 \begin{eqnarray*}
 |z|^{p}H_{m-p}\,&\simeq&\, \mathbf{C}[\phi^{+(m-p,l,k)};\,l=0,1,\cdots,m-p,\, k=0,\cdots m-p+1]\\[0.2cm]
  |z|^{p}H_{-(m-p+3)}\,&\simeq&\, \mathbf{C}[\phi^{-(m-p,l,k)};\,l=0,1,\cdots,m-p,\, k=0,\cdots m-p+1]
  \end{eqnarray*}
 that is given by letting \(|z|=1\) and the inverse is defined by 
 \begin{equation}
 \phi\longrightarrow |z|^m\phi(\frac{z}{|z|}).\label{extension}
 \end{equation}
 From  (\ref{harmonicdecomp}) we have our assertion.
   \hfill\qed

{\bf Example 1}\\
\(\phi^{+(1,0,1)}\phi^{+(1,1,1)}\in \mathcal{L}[2]\) is decomposed to 
\begin{eqnarray}
\phi^{+(1,0,1)}(z)\phi^{+(1,1,1)}(z)&=&
\,\frac{1}{\sqrt{2}}\phi^{+(2,1,1)}(z)-\frac{1}{3}\phi^{+(2,2,2)}(z)\,-
\frac{1}{3\sqrt{2}}|z|^6\phi^{-(1,1,2)}(z)\nonumber \\[0.2cm]
&&\in  H_2\oplus |z|^6H_{-4}\,.
\end{eqnarray}
The restriction to \(\{|z|=1\}\) becomes 
\begin{equation*}
\mathcal{L}[2]\ni \phi^{+(1,0,1)}\phi^{+(1,1,1)}\longrightarrow 
\varphi=
\,\,\frac{1}{\sqrt{2}}\phi^{+(2,1,1)}(z)-\frac{1}{3}\phi^{+(2,2,2)}(z)\,-
\frac{1}{3\sqrt{2}}\phi^{-(1,1,2)}(z)\,\in \mathbf{C}[\phi^{\pm}]
\end{equation*}
and the inverse is given by 
\[
\mathbf{C}[\phi^{\pm}]\ni\varphi\longrightarrow |z|^{2}\varphi(\frac{z}{|z|})\,\in \mathcal{L}[2]\,.\]

{\bf Example 2}\\
\(\phi^{+(1,0,1)}\phi^{-(0,0,0)}\in  \mathcal{L}[-2]\) is decomposed to 
\begin{eqnarray}
\phi^{+(1,0,1)}(z)\phi^{-(0,0,0)}(z)&=&
\frac{1}{|z|^4}(\,\frac{2}{3}\phi^{+(2,0,1)}(z)+\frac{\sqrt{2}}{3}\phi^{+(2,1,2)}(z)\,)+
\frac{1}{|z|^2}
\frac{1}{2}\phi^{+(0,0,1)}(z)\nonumber \\[0.2cm]
&&+
|z|^2(\,\frac{1}{6}\phi^{-(1,1,1)}(z)+\frac{1}{3\sqrt{2}}\phi^{-(1,0,0)}(z)\,)\nonumber
\\[0.2cm]
&&\in  \frac{1}{|z|^4}H_2\oplus \frac{1}{|z|^2}H_0\oplus |z|^2H_{-4}\,.\label{examp}
\end{eqnarray}
The restriction to \(\{|z|=1\}\) becomes 
\begin{eqnarray*}
\mathcal{L}[-2]&\ni &\phi^{+(1,0,1)}\phi^{-(0,0,0)}\longrightarrow \\[0.2cm]
\varphi&=&
\,\frac{2}{3}\phi^{+(2,0,1)}+\frac{\sqrt{2}}{3}\phi^{+(2,1,2)}\,+\frac{1}{2}\phi^{+(0,0,1)}+\frac{1}{6}\phi^{-(1,1,1)}+\frac{1}{3\sqrt{2}}\phi^{-(1,0,0)}\,\in \mathbf{C}[\phi^{\pm}]\end{eqnarray*}
and the inverse is given by 
\[
\mathbf{C}[\phi^{\pm}]\ni\varphi\longrightarrow |z|^{-2}\varphi(\frac{z}{|z|})\,\in \mathcal{L}[-2]\,.\]

By virtue of the above isomorphisms (\ref{brestr}) and (\ref{extension}) we 
extend the definitions of a quarternion trace and a quarternion residue that are defined for the  Laurent polynomial type spinors to those over the algebra of current: 
\[\mathcal{L}\,=\, \bigoplus_{m\in \mathbf{Z}}\,\mathcal{L}[m]\,,\quad   \mathcal{L}[m]=\sum_{p\in \mathbf{Z}}\,|z|^{p}\,H_{m-p}\,
.   \]

\begin{definition}\label{defforL}
Let \(\varphi=\,\sum_{m}\,\varphi_{[m]}\,\) with  \(\varphi_{[m]}\in \mathcal{L}[m]\) be the decomposition of \(\varphi\) into the homogeneous components \(\varphi_{[m]}\).   We denote \(\varphi_{[m]}=\sum_{p\in \mathbf{Z}}|z|^p\phi_{[m-p]}\) with \(\phi_{[m-p]}\in H_{m-p}\).    
\(qTr\,\varphi\) is defined by the term 
 \(\,\varphi_{[0]}\vert_{|z|=1}\,=\sum_{p\in\mathbf{Z}}\phi_{[-p]}\in\mathbf{C}[\phi^{\pm}]\,\): 
 \begin{equation}
 qTr\,\varphi\,=\,\frac{1}{2\pi^2}\int_{S^3}\,\varphi_{[0]}(\zeta)\sigma(d\zeta)\,.
 \end{equation}
And \( qRes\,\varphi\,\) is given by the term  
  \(\varphi_{[-3]}\vert_{|z|=1}\,=\sum_{p\in\mathbf{Z}}\phi_{[-3-p]}\in\mathbf{C}[\phi^{\pm}]\,\):
\begin{equation}
q Res\,\varphi\,=\,\frac{1}{2\pi^2}\int_{S^3}\,\gamma_+(\zeta)\varphi_{[-3]}(\zeta)\sigma(d\zeta)\,.
 \end{equation}
 \end{definition}

{\bf Example 1.}~~~
Let 
 \(\varphi=\phi^{+(1,1,2)}\cdot \phi^{+(1,1,2)}\cdot\phi^{-(1,0,0)}\,\) with 
  \(\,\, \phi^{+(1,1,2)}=\left(\begin{array}{c}-\sqrt{2}\overline{z}_2\\ 0\end{array}\right)\) and \( \phi^{-(1,0,0)}=\frac{\sqrt{2}}{|z|^6}\left( \begin{array}{c}z_2^2\\ \frac12z_2\overline{z}_1\end{array}\right)\,\).   Then 
\(\,\varphi(z)
=\frac{\sqrt{2}}{|z|^7}\,
\left(
\begin{array}{c}2|z_2|^4-|z_2|^2\overline{z}_2\overline{z}_1\\ 
2z_2^4+z_2^3\overline{z}_1\end{array}\right )
\,\in\mathcal{L}\,\) 
and 
\[\,qRes\,\varphi
=\,2\sqrt{2}\left(\begin{array}{c}1\\ 0\end{array}\right)\,.\]

{\bf Example 2.}~~~
Let \(\psi=|z|^3\phi^{+(1,1,2)}\cdot \phi^{+(1,1,2)}\cdot\phi^{-(1,0,0)}\,\).
\begin{equation*}
\psi=|z|^{3}\varphi(z)=\frac{\sqrt{2}}{|z|^4}\,
\left(
\begin{array}{c}2|z_2|^4-|z_2|^2\overline{z}_2\overline{z}_1\\ 
2z_2^4+z_2^3\overline{z}_1\end{array}\right )
\,\in\mathcal{L}\,
\end{equation*}
it holds that 
\[ qTr\,\psi\,=\,2\sqrt{2}\left(\begin{array}{c}1\\ 0\end{array}\right)\,.\]

\hspace{1pt}\\
\begin{proposition}\label{resrelation2}
Lemma \ref{resDandn} is valid for the algebra of current \(\mathcal{L}\):
\begin{equation}
qRes\,\Do\varphi\,=0\,, \quad    qRes\,|z|^{-k}\,{\bf n}\varphi=0\,,\quad \forall\, k\in \mathbf{N}.
\end{equation}
\end{proposition}
In particular \( qRes\,{\frac{\partial}{\partial n} }\varphi=0\,\).     But \( qRes\,{\bf n}\varphi=-3\,qRes\varphi \).

\begin{lemma}\label{res-tr}
  \begin{enumerate}
  \item
  \begin{equation}
  qRes\,|z|^{-k}\varphi\,=0\,,\quad k=1,2,\qquad\forall\varphi\in\mathcal{L}.
  \end{equation}
  \item
  \begin{equation}
  q Res\, |z|^{-3}\varphi\,=\,q Tr \varphi\,.
  \end{equation}
  \end{enumerate}
\end{lemma}
\begin{proof}~~~
First 
let \(\varphi\) be a Laurent polynomial type spinor; \(\varphi\,=\,\sum\,C_{-(m,\cdot,\cdot)}\phi^{-(m,\cdot,\cdot)}\,+\,
\,\sum\,C_{+(m,\cdot,\cdot)}\phi^{+(m,\cdot,\cdot)}\,\).    The quarternion trace of \(|z|^3\varphi\) is given by the coefficients of 
its homogeneous component of order \(0\), hence is equal to
\(qTr\,|z|^3\varphi\,=\,\left(\begin{array}{c}C_{-(0,0,1)}\\ C_{-(0,0,0)}
\end{array}\right)=\,q\,Res\,\varphi \).     Any spinor of \(\mathcal{L}\) is written by a sum of Laurent polynomial type spinors each summand of which may be multiplied by some radial length \(|z|^p\) so that we the proof follows from Definition\ref{defforL} .   
\end{proof}

\subsection{\(\mathbf{H}\)-valued 2-cocycle on \(\mathcal{L}\).
}

\subsubsection{2-cocycle on \(\mathcal{L}\)}

We shall introduce a {\it \(\mathbf{H}\)-valued }  2-cocycle on the \(\mathbf{R}-\)algebra  
\(\,\mathcal{L}\,\).   Remember the decomposition 
\begin{eqnarray*}
\mathcal{L}&=&\, \mathcal{K}\oplus\mathcal{J}, \\
\mathcal{K}&=&\,\mathbf{R}I\,,\quad \mathcal{J}\,=\{\phi=q\,i+r\,j+s\,k\in \mathcal{L}\,;\quad q,\,r\,, s\in C^{\infty}(\mathbf{C}^2\setminus{0},\mathbf{R})\,.
\end{eqnarray*}
We call  \(\phi\in \mathcal{L}\) a {\it basic form} if it takes one of the following forms:
\[\phi=pI,\,qi,\,rj,\,sk\quad   \mbox{ with } \,p,\,q,\,r,\,s\in C^{\infty}(\mathbf{C}^2\setminus{0},\mathbf{R})\,.\]
    The product of two basic forms is also a basic form.   So every spinor in \(\mathcal{L}\) is written as a sum of basic forms.   For basic forms \(\phi\,,\,\psi\,\in \mathcal{L}\), we put
\begin{equation}
\epsilon(\phi,\,\psi\,)=\left\{
\begin{array}{rl} 
+1, &\quad\mbox{ if \(\phi\) or \(\psi\) or \(\phi\cdot \psi\in \mathcal{K}\)} \\
-1, &\quad\mbox{ if \(\phi,\,\psi \,\mbox{ and } \phi\cdot\psi \in \mathcal{J}\)}.
\end{array}\right.
\end{equation}
That is, 
\begin{eqnarray*}
\epsilon(pI,tI)&=&\epsilon(pI,\,qi)=\epsilon(pI,rj)=\epsilon(pI,sk)=+1,\\
\epsilon(qi,\,ui)&=&\epsilon(rj,\,vj)=\epsilon(sk,\,wk)=\,+1,\\
\epsilon(qi,\,vj)&=&\epsilon(qj,\,wk)=\epsilon(rj,\,wk)\,=-1\,.\,.\end{eqnarray*}

\begin{lemma}\label{spinordeLambert}~~~ For basic forms \(\phi\) and \(\psi\) we have
\begin{equation}
\Do\phi\cdot \psi\,+\,\epsilon(\phi,\,\psi)\Do\psi\cdot\phi\,=\,\Do(\phi\cdot\psi)\,=\epsilon(\phi,\,\psi)\Do(\psi\cdot\phi).
\end{equation}
\end{lemma}
\begin{proof}
The formulae are proved for each couple of basic forms \(\phi,\,\psi\): 
For example, for \(\phi=pI\) and \(\psi=ui\), we have 
\begin{eqnarray*}
\Do\phi\cdot\psi\,&=&\, -(\theta_3p)uI+(\theta_1p)uk+(\theta_2p)uj\\
\Do\psi\cdot\phi\,&=&\, -(\theta_3u)pI+(\theta_1u)pk+(\theta_2u)pj\\
\Do\phi\cdot \psi\,+\,\Do\psi\cdot\phi\,&=&\, -(\theta_3up)I+(\theta_1up)k+(\theta_2up)j \,\\
&=& \Do(\,p\,u\,i)\,=\,\Do(\phi\cdot\psi)\,.
\end{eqnarray*}
For 
 \(\phi=rj\) and \(\psi=wk\), we have 
\begin{eqnarray*}
\Do\phi\cdot\psi\,&=&\, -(\theta_3r)wI+(\theta_1r)wk+(\theta_2r)wj\\
\Do\psi\cdot\phi\,&=&\, +(\theta_3w)rI-(\theta_1w)rk-(\theta_2w)rj\\
\Do\phi\cdot \psi\,-\,\Do\psi\cdot\phi\,&=&\, -(\theta_3rw)I+(\theta_1rw)k+(\theta_2rw)j \,\\
&=& \Do(\,r\,w\,i)\,=\,\Do(\phi\cdot\psi)\,.
\end{eqnarray*}
Others follow by similar calculations.\end{proof}

\begin{definition}\label{cocycle}
For \(\phi,\,\psi\,\in  \mathcal{L}\),  
we put  
\begin{equation}
A(\phi,\,\psi\,)\,=\,qRes\,\left( \,\Do\phi\,\cdot\,\psi\,-\,\Do\psi\,\cdot\,\phi\,\right)\,.
\end{equation}
\end{definition}
From Proposition \ref{intrepres} we have 
\begin{equation}
A(\phi,\,\psi\,)\,=\,\frac{1}{2\pi^2}\int_{\vert\zeta\vert=1}\gamma_+(\zeta)\left( \,\Do\phi\,\cdot\,\psi\,-\,\Do\psi\,\cdot\,\phi\,\right)\,\sigma(d\zeta).
\end{equation}

\(A(\phi,\,\psi\,)\) is an antisymmetric bilinear form of \(\phi,\,\psi\,\in   \mathcal{L}\).   
    From Proposition \ref{resrelation2}  and Lemma \ref{spinordeLambert} 
we see that \(A(\phi,\psi)=0\) for basic forms \(\phi,\,\psi\) with \(\epsilon(\phi,\psi)=\,-1\).

\begin{proposition}~~~
   For \(\phi_1,\,\phi_2\,,\,\phi_3\,\in   \mathcal{L}\),  we have
\begin{eqnarray}
&& A(\,\phi_1\,,\,\phi_2\,)\,+\,A(\,\phi_2\,,\,\phi_1\,)\,=\,0\,,\label{antisym}\\
&& A(\,\phi_1\cdot \phi_2\,, \,\phi_3)\,+\,A(\,\phi_2\cdot \phi_3\,, \,\phi_1)\,+
\,A(\,\phi_3\cdot \phi_1\,, \,\phi_2)\,=\,0\,.\label{precocycle}
\end{eqnarray}
\end{proposition}

\begin{proof}
Evidently \(A\) is  an antisymmetric bilinear form.   
We put \(\mathfrak{a}(\phi,\,\psi)=\left( \,\Do\phi\,\cdot\,\psi\,-\,\Do\psi\,\cdot\,\phi\,\right)\).   
The 2-cocycle property is proved by calculations  
exploiting the formula 
\(\mathfrak{a}(\phi_1\phi_2\,,\,\phi_3)\, \) 
for basic forms \(\phi_1,\phi_2,\,\phi_3\).
For these basic forms we have 
\begin{eqnarray*}
\mathfrak{a}(\phi_1\cdot\phi_2\,,\,\phi_3)&= &(\Do(\phi_1\cdot\phi_2))\cdot\phi_3\,- (\Do \phi_3)\cdot (\phi_1\cdot \phi_2)\\
&=&(\Do\phi_1)\cdot (\phi_2\cdot\phi_3)\,+\,\epsilon(\phi_1,\phi_2)(\Do\phi_2)\cdot(\phi_1\cdot\phi_3)
- (\Do \phi_3)\cdot (\phi_1\cdot \phi_2)\,.\\
\mathfrak{a}(\phi_2\cdot \phi_3\,,\,\phi_1)&= &(\Do(\phi_2\cdot\phi_3))\cdot\phi_1\,- (\Do \phi_1)\cdot (\phi_2\cdot \phi_3)\\
&=&(\Do\phi_2)\cdot\phi_3\cdot\phi_1\,+\,\epsilon(\phi_2,\phi_3)(\Do\phi_3)\cdot(\phi_2\cdot\phi_1)
- (\Do \phi_1)\cdot (\phi_2\cdot \phi_3)\,.\\
\mathfrak{a}(\phi_3\cdot \phi_1\,,\,\phi_2)&=&(\Do\phi_3)\cdot\phi_1\cdot\phi_2\,+\,\epsilon(\phi_3,\phi_1)(\Do\phi_1)\cdot(\phi_3\cdot\phi_2)
- (\Do \phi_2)\cdot (\phi_3\cdot \phi_1)\,.
\end{eqnarray*}
Then 
\begin{eqnarray*}
&&\mathfrak{a}(\,\phi_1\cdot \phi_2\,, \,\phi_3)\,+\,\mathfrak{a}(\,\phi_2\cdot \phi_3\,, \,\phi_1)\,+
\,\mathfrak{a}(\,\phi_3\cdot \phi_1\,, \,\phi_2)\,\\
&=& \,\epsilon(\phi_1,\phi_2)(\Do\phi_2)\cdot(\phi_1\cdot\phi_3)\,+\,
\epsilon(\phi_2,\phi_3)(\Do\phi_3)\cdot(\phi_2\cdot\phi_1)\,+
\epsilon(\phi_3,\phi_1)(\Do\phi_1)\cdot(\phi_3\cdot\phi_2)\,\\
&=&\,k\,\Do(\phi_1\cdot\phi_2\cdot\phi_3)\,,
\end{eqnarray*}
for a constant \(k\).
Since 
\[A(\phi,\psi)\,=\,q\,Res.\mathfrak{a}(\phi,\psi)\,,\quad\mbox{ for \( \phi\,,\psi\in \mathcal{L}\,\),}\]
we have from Proposition \ref{resrelation2} 
\begin{equation}
\begin{split}
&A(\,\phi_1\cdot \phi_2\,, \,\phi_3)\,+\,A(\,\phi_2\cdot \phi_3\,, \,\phi_1)\,+
\,A(\,\phi_3\cdot \phi_1\,, \,\phi_2)\,\\
&=\,\,qRes.\,\left(
\mathfrak{a}(\,\phi_1\cdot \phi_2\,, \,\phi_3)\,+\,\mathfrak{a}(\,\phi_2\cdot \phi_3\,, \,\phi_1)\,+
\,\mathfrak{a}(\,\phi_3\cdot \phi_1\,, \,\phi_2)\,\right)d\sigma\,\\
&=\,\,qRes.\,\left(\,k\,\Do(\phi_1\cdot\phi_2\cdot\phi_3)\,\right)=0\,,
\end{split}
\end{equation}
for basic forms \(\phi_1,\phi_2\) and \(\phi_3\).       
  Then by the additivity of \(\,A(\phi,\,\psi)\) for 
the variables \(\phi\) and \(\psi\) we conclude that \(A\) satisfies the condition (\ref{precocycle}).
\end{proof}
The bilinear form \(A(\phi\,,\,\psi\,)\) of Definition \ref{cocycle} on \( \mathcal{L}\) gives a \(\mathbf{H}\)-valued 2-cocycle.

By the definition of quarternion residue we have the following:
\begin{eqnarray}
A(\phi^{+(m_1,l_1,k_1)}\,,\,\phi^{+(m_2,l_2,k_2)}\,)\,&=&\,0\\[0.2cm]
A(\phi^{-(m_1,l_1,k_1)}\,,\,\phi^{-(m_2,l_2,k_2)}\,)\,&=&\,0\,.
\end{eqnarray}

\begin{proposition}\label{cocycleandnormalder}
\begin{equation}
A(\,\frac{\partial}{\partial  n}\,\phi\,,\,\psi\,)\,+\,A(\,\phi\,,\,\frac{\partial}{\partial  n}\psi\,)\,=\,0\,.
\end{equation}
\end{proposition}

 \begin{proof}~~~
 Let \(\,\mathfrak{a}(\phi,\,\psi)=\,\Do\phi\,\cdot\,\psi\,-\,\Do\psi\,\cdot\,\phi\,\).  
It holds from (\ref{DNcommute}) that 
\(\, \frac{\partial}{\partial  n}\,\Do\,=\Do\,\frac{\partial}{\partial  n} - \Do\,\) on \(\mathcal{L}\).   
So we have
\begin{eqnarray*}
\mathfrak{a}(\frac{\partial}{\partial  n}\phi\,,\,\psi\,)\,+\,\mathfrak{a}(\,\phi\,,\,\frac{\partial}{\partial  n}\psi\,)\,
&=&\Do(\,\frac{\partial}{\partial  n}\phi\,)\,\cdot \psi -\Do\psi \,\cdot (\,\frac{\partial}{\partial  n}\phi\,)\,+
\Do\phi \,\cdot (\,\frac{\partial}{\partial  n}\psi\,) \, -\Do(\,\frac{\partial}{\partial  n}\psi\,)\,\cdot \phi\,\\[0.2cm]
&=&\,\frac{\partial}{\partial  n}\,(\Do\phi)\,\cdot \,\psi \,-\,  \,\frac{\partial}{\partial  n}\,(\Do\psi)\,\cdot \,\phi \,  +\,\Do\phi\cdot(\,\frac{\partial}{\partial  n}\,\psi)\,-\, 
\Do\psi\cdot(\,\frac{\partial}{\partial  n}\phi)\\[0.2cm]
&\,& +\frac{1}{2|z|}\left(\Do\phi\cdot\psi-\Do\psi\cdot\phi\right)\\[0.2cm]
&=&\,
\,\frac{\partial}{\partial  n}\,(\Do\phi\,\cdot\psi\,-\,\Do\psi\cdot\phi\,) 
+\frac{1}{2|z|}\left(\Do\phi\cdot\psi-\Do\psi\cdot\phi\right)
\\[0.2cm]
&=& \,\frac{\partial}{\partial  n}\,\mathfrak{a}(\phi,\psi)\,+\,\frac{1}{2|z|}\mathfrak{a}(\phi,\psi)\,.
\end{eqnarray*}
By virtue of Proposition \ref{resrelation2}, we have 
\[A(\frac{\partial}{\partial  n}\phi\,,\,\psi\,)\,+\,A(\,\phi\,,\,\frac{\partial}{\partial  n}\psi\,)\,=\,qRes \left [ \frac{\partial}{\partial  n}\,\mathfrak{a}(\phi,\psi)\,+\,\frac{1}{2|z|}\mathfrak{a}(\phi,\psi)\,\right]=\,0\,.\]
\end{proof}

 \section{\(\mathfrak{g}\)-current algebras on \(S^3\)}

\subsection{Quarternification of a Lie algebra}
 
    Hitherto we have prepared the spaces \( C^{\infty}(S^3,\,S^+)\) and \(\mathcal{L}\)  that will play the role of {\it coefficients of current algebras}  which we shall discuss in  the next section.   These are {\it complex} vector spaces.     As we have seen they are given a \(\mathbf{C}-\)algebra structure.    \(C^{\infty}(S^3,S^+)\simeq C^{\infty}(S^3,\mathbf{H})\) has also a \(\mathbf{H}\)-module structure, while our basic interest is on  \(\,\mathcal{L}\) that is  endowed with the {\it real Lie algebra} structure.    In such a way  it is frequent that we deal with the fields \(\mathbf{H}\),  \(\mathbf{C}\) and \(\mathbf{R}\) in one formula.   
So to prove a steady point of view for our subjects we introduce here the concept of {\it quarternion Lie algebras}, \cite{Kq}.    

  First we note that a quarternion module \(V=\mathbf{H}\otimes_{\mathbf{C}}V_o=V_o+JV_o\,\), \(V_o\) being a \(\mathbf{C}\)-module, has two  involutions \(\sigma\) and \(\tau\):\[\sigma(u+Jv)=u-Jv\,,\quad \tau(u+Jv)=\overline u+J\overline v\,,\quad u,v\in V_o\,.\]
 \[\sigma\,\left(\begin{array}{c}u\\
v\end{array}\right)\,=\left(\begin{array}{c}u\\
-v\end{array}\right)\, \mbox{ and } \quad
\tau \left(\begin{array}{c}u\\
v\end{array}\right) =\left(\begin{array}{c}\overline u\\
\overline v\end{array}\right)\,.
\]
  We have \(\mathcal{L}= \mathcal{K}\oplus  \mathcal{J}\) with 
 \begin{eqnarray*}
 \mathcal{K}&=&\{\varphi\in\mathcal{L};\,\sigma\varphi=\tau\varphi=\varphi\},\\
 \mathcal{J}&=&\{\varphi\in\mathcal{L};\,\sigma\varphi=-\varphi\quad\mbox{ or }\quad\tau\varphi=-\varphi\,\}.
 \end{eqnarray*}

\begin{definition} ~~~ A {\it quarternion Lie algebra} \(\,\mathfrak{q}\) is a \(\mathbf{C}\)-submodule of a \(\mathbf{H}-\)module \(V\) that is endowed with a real Lie algebra structure compatible with the involutions \(\sigma\) and \(\tau\):
\begin{eqnarray*}
&\sigma\mathfrak{q}\,\subset\mathfrak{q}\,,\\[0.2cm] &\sigma [x\,,y]\,=[\sigma x\,,\sigma y]\,,\quad \tau [x\,,y]\,=[\tau x\,,\tau y]\,\quad \mbox{ for }\, x,y\in \mathfrak{q} . \label{involutions}
\end{eqnarray*}
 For a complex Lie algebra \(\mathfrak{g}\) the {\it quarternification} of \(\mathfrak{g}\) is a quarternion Lie algebra \(\mathfrak{g}^q\) that is generated ( as a real Lie algebra ) by  \(\mathbf{H}\otimes_{\mathbf{C}}\mathfrak{g}\).     
 \end{definition}
 For example, 
\(\mathfrak{so}^{\ast}(2n)=\mathbf{H}\otimes_{\mathbf{C}}\mathfrak{so}(n,\mathbf{C})\) is the quarternification of \(\mathfrak{so}(n,\mathbf{C})\,\).   And 
 \(\mathfrak{sl}(n,\mathbf{H})\) is the quarternification of  \(\mathfrak{sl}(n,\mathbf{C})\) while 
\(\mathbf{H}\otimes_{\mathbf{C}}\mathfrak{sl}(n,\mathbf{C})\) is not a Lie algebra.     

\begin{theorem}~~~
 The algebra of current \(\mathcal{L}\,\) is a quarternion Lie algebra by its associated Lie bracket.   
\end{theorem}
     In fact \(\mathcal{L}\,\) is a real submodule of \(S^3\mathbf{H}=C^{\infty}(S^3,\,\mathbf{H})\) that is invariant under the involutions \(\sigma\) and \(\tau\).     \(\mathcal{L}\,\) is an associative complex algebra with the multiplication (\ref{spinormultip}) and has an induced  Lie algebra structure over \(\mathbf{R}\).   The involutions \(\sigma\) and \(\tau\) are given by 
 \[\sigma\,\left(\begin{array}{c}u\\
v\end{array}\right)\,=\left(\begin{array}{c}u\\
-v\end{array}\right)\, \mbox{ and } \quad
\tau \left(\begin{array}{c}u\\
v\end{array}\right) =\left(\begin{array}{c}\overline u\\
\overline v\end{array}\right)\,.
\]
and satisfy \(\sigma([x,y])=[\sigma\,x,\,\sigma\,y]\) and  \(\tau([x,y])=[\tau\,x,\,\tau\,y]\,\).   

  We have \(\mathcal{L}= \mathcal{K}\oplus  \mathcal{J}\) with 
 \begin{eqnarray*}
 \mathcal{K}&=&\{\varphi\in\mathcal{L};\,\sigma\varphi=\tau\varphi=\varphi\},\\
 \mathcal{J}&=&\{\varphi\in\mathcal{L};\,\sigma\varphi=-\varphi\quad\mbox{ or }\quad\tau\varphi=-\varphi\,\}.
 \end{eqnarray*}
     
\medskip

\subsection{\(\mathfrak{g}-\)Current algebras and its subalgebras }

 We remember that we called  the algebra  \(\,\mathcal{L}\)  {\it  the algebra of current on \(S^3\)}.   By virtue of Theorem \ref{Laurentalgebra} \(\,\mathcal{L}\) is a \(\mathbf{C}\)-algebra generated by the spinors 
\[I=\phi^{+(0,0,1)},\,J=-\phi^{+(0,0,0)},\,\kappa=\phi^{-(0,0,1)},\,\lambda=\phi^{-(0,0,0)}\,.\] 
In the sequel we denote \(S^3E\) for a \(\mathbf{C}\)-module \(E\)  the space of \(E\)-valued smooth functions on \(S^3\): \(S^3E=C^{\infty}(S^3,\,E)\).    The space of positive spinors \(C^{\infty}(S^3, S^+)\) is identified with 
\(S^3\mathbf{H}\).     

\(S^3\,\mathfrak{gl}(n,\mathbf{H})\,=\,S^3\mathbf{H}\otimes_{\mathbf{C}}\mathfrak{gl}(n,\mathbf{C})\) becomes a quarternion Lie algebra  with the  Lie bracket defined by 
 \begin{equation}\label{glbracet}
 [\,\phi_1 \otimes X_1\, , \,\phi_2  \otimes X_2\,]=  ( \phi_1\cdot\phi_2 ) \otimes X_1X_2\,
- (\phi_2 \cdot \phi_1 ) \otimes X_2X_1\,,
\end{equation}
for \( \phi_1,\phi_2\in 
S^3\,\mathbf{H}\) and \(X_1,X_2\in \mathfrak{gl}(n,\mathbf{C})\).
Here the right hand side is in the tensor product of the \(\mathbf{C}\)-algebra \(S^3\mathbf{H}\) and the matrix algebra \(\mathfrak{gl}(n,\mathbf{C})\).   

Let \((\,\mathfrak{g}\,,\,[\,\,,\,\,]\,)\) be a complex Lie algebra that we suppose to be a subalgebra of \(\mathfrak{gl}(n,\mathbf{C})\) for some  \(n\).  
   Then \(\mathcal{L}\otimes_{\mathbf{C}}  \mathfrak{g}\,\) becomes a \(\mathbf{C}\)-submodule of  \( S^3\mathfrak{gl}(n,\mathbf{H})=\,S^3\mathbf{H}\otimes_{\mathbf{C}}\mathfrak{gl}(n,\mathbf{C})\).   
The involutions \(\sigma\) and \(\tau\) on \(\,\mathcal{L}\) are extended to \(\,\mathcal{L}\otimes_{\mathbf{C}}\mathfrak{g}\,\)  by
 \[\sigma(\phi\otimes X)=\sigma(\phi)\otimes X\,\quad\mbox{ and }\quad \tau(\phi\otimes X)=\tau(\phi)\otimes X\,\]
 respectively  for \(\phi\in \mathcal{L}\) and \(X\in \mathfrak{g}\,\).      
Thus \(\mathcal{L}\otimes_{\mathbf{C}}  \mathfrak{g}\,\) endowed with 
the  bracket (\ref{glbracet}) generates a quarternion Lie algebra.    
 \begin{definition}~~~
 The quarternification of \(
\left(\,\mathcal{L}\otimes_{\mathbf{C}}  \mathfrak{g}\,,\,[\,\,,\,\,]\,\right)
\), that is, 
the quarternion Lie algebra generated by 
\(
\left(\,\mathcal{L}\otimes_{\mathbf{C}}  \mathfrak{g}\,,\,[\,\,,\,\,]\,\right)
\) 
is called {\it \(\mathfrak{g}\)-current algebra} and is denoted by \(\mathcal{L}\mathfrak{g}\).  
\end{definition}
As the following examples show \(\mathcal{L}\otimes_{\mathbf{C}}\mathfrak{g}\) is not necessarily a Lie algebra so that the Lie algebra \(\mathcal{L}\mathfrak{g}\) is defined as the {\it quarternification } generated by \(\mathcal{L}\otimes_{\mathbf{C}}\mathfrak{g}\) in the \(\mathbf{H}\)-module 
\(S^3\mathfrak{gl}(n,\mathbf{H})\).     

{\bf Examples}: ~~~ The following elements are in 
\(\, \mathcal{L}\mathfrak{g}\,\ominus \,(\mathcal{L}\otimes_{\mathbf{C}}\mathfrak{g})\,\).
\begin{enumerate}
\item
\begin{equation*}
\sqrt{-1}(X_1X_2+X_2X_1)\,\in \mathcal{L}\mathfrak{g}\,,\quad\mbox{ for }\forall X_1,\,X_2\in \mathfrak{g}\,.
\end{equation*}
In fact we have 
\[\mathcal{L}\mathfrak{g}\,\ni\,
[\,J\otimes X_1\,,\,\sqrt{-1}J\otimes X_2\,]\,=\,\sqrt{-1}I\otimes(X_1X_2+X_2X_1)\,.\]
Here the right hand-side is calculated in \(S^3\mathfrak{gl}(n,\mathbf{H})\) which gives the left-hand side element of \(\mathcal{L}\mathfrak{g}\,\).  
\item
\begin{equation*}
\sqrt{-1}J\otimes (X_1X_2+X_2X_1)\,\in \mathcal{L}\mathfrak{g}\,,\quad\mbox{ for }\forall X_1,\,X_2\in \mathfrak{g}\,.
\end{equation*}
In fact 
\[\mathcal{L}\mathfrak{g}\,\ni\,
[\,J\otimes X_1\,,\,\sqrt{-1}I\otimes X_2\,]\,=\,\sqrt{-1}J\otimes(X_1X_2+X_2X_1)\,.\]
\end{enumerate} 

Remember that the quarternification of a complex Lie algebra \(\mathfrak{g}\) is the quarternion Lie algebra \(\mathfrak{g}^q\) generated by \(\mathbf{H}\otimes_{\mathbf{C}}\mathfrak{g}=\mathfrak{g}+J\mathfrak{g}\).    The latter is not a Lie algebra in general.  Since 
 \(I=\phi^{+(0,0,1)}=\left(\begin{array}{c}1\\0\end{array}\right)\) and \(J=-\phi^{+(0,0,0)}=\left(\begin{array}{c}0\\1\end{array}\right)\) are in \(\,\mathcal{L}\),  
\(\mathfrak{g}^q\) becomes a subspace of \(\mathcal{L}\mathfrak{g}\).   
We have the following relations:
\begin{equation*}
S^3\mathfrak{g}^q\supset S^3\mathfrak{g}+J\,(S^3\mathfrak{g})\supset \mathcal{L}\mathfrak{g}\,\supset \mathfrak{g}^q,
\end{equation*}
where \(S^3\mathfrak{g}+J\,(S^3\mathfrak{g})\) is not necessarily a Lie algebra and   
\(S^3\mathfrak{g}^q\) is the Lie algebra with bracket (\ref{glbracet}).    

The following examples \(1\sim 3\) show the case where both \(S^3\mathbf{H}\otimes_{\mathbf{C}}\mathfrak{g}\) and \(\mathbf{H}\otimes_{\mathbf{C}}\mathfrak{g}\) become Lie algebras (over \(\mathbf{C}\)).

{\bf Examples}
\begin{enumerate}
\item
\[\mathfrak{gl}(n,\mathbf{H})=\mathbf{H}\otimes_{\mathbf{C}}\mathfrak{gl}(n,\mathbf{C})
\subset \mathcal{L}\mathfrak{gl}(n,\mathbf{C})
\subset S^3\mathbf{H}\otimes_{\mathbf{C}}\mathfrak{gl}(n,\mathbf{C})\,= S^3
\mathfrak{gl}(n,\mathbf{H})
\]
\item
\[
so^{\ast}(2n)=\mathbf{H}\otimes_{\mathbf{C}}\mathfrak{so}(n,\mathbf{C})\,\subset\, 
\mathcal{L}\mathfrak{so}(n,\mathbf{C})\,\subset S^3\mathbf{H}\otimes \mathfrak{so}(n,\mathbf{C})=S^3\mathfrak{so}^{\ast}(2n)  \]
\item
\[
\mathfrak{sp}(2n)=\mathbf{H}\otimes_{\mathbf{C}}\mathfrak{u}(n)\,\subset \mathcal{L}\mathfrak{u}(n)\,\subset \, S^3\mathbf{H}\otimes_{\mathbf{C}}\mathfrak{u}(n)\,=\,S^3\mathfrak{sp}(2n)    .\]
\item~~~
 \(\mathfrak{sl}(n,\mathbf{H})\) is the quarternification of  \(\mathfrak{sl}(n,\mathbf{C})\) which is not in \(\mathcal{L}\otimes_{\mathbf{C}}\mathfrak{sl}(n,\mathbf{C})\). \\
   In fact, 
 let \(\{h_i=E_{ii}-E_{i+1\,i+1};\,1\leq i\leq n-1\,,\quad E_{ij},\,i\neq j\,\}\) be the basis of \(\mathfrak{g}=\mathfrak{sl}(n,\mathbf{C})\).   Then 
\([\sqrt{-1}Jh_1,Jh_2\,]=-2\sqrt{-1}E_{22}\,\in \mathfrak{g}^q\subset\mathcal{L}\mathfrak{g}\) but not in \(\mathcal{L}\otimes_{\mathbf{C}}\mathfrak{g}\).  
\end{enumerate}

\subsection{Root space decomposition of \(\mathfrak{g}\)-current algebras}

\subsubsection{ }
Let \(\mathfrak{g}\) be a simple Lie algebra with Cartan matrix \(A=\left( c_{ij}\right )\).    Let \(\mathfrak{h}\) be a Cartan subalgebra, \(\Phi\) the corresponding root system.   Let \(\Pi=\{\alpha_i;\,i=1,\cdots,l=\dim\,\mathfrak{h}\}\subset \mathfrak{h}^{\ast}\) be the set of simple roots and  \(\{h_i=\alpha_i^{\vee}\,;\,i=1,\cdots,l\,\}\subset \mathfrak{h}\) be the set of simple coroots.   The Cartan matrix
 \(A=(\,c_{ij}\,)_{i,j=1,\cdots,r}\) is given by \(c_{ij}=\left\langle \alpha_j ,\,\alpha_i^{\vee}\right\rangle\).       
 \(\alpha(h)\) is real if  \(h\in\mathfrak{h}\) is real.   The Killing form 
 \( (x\vert \,y) =tr( ad\,x\,\,ad\,y )\) gives the symmetric invariant bilinear form on \(\mathfrak{g}\).   We have an isomorphism \(h\longrightarrow h^{\ast}\) from \(\mathfrak{h}\) to \(\mathfrak{h}^{\ast}\) given by \(\langle h^{\ast}, x \rangle = ( h\vert x )\).
 Let \(\,\mathfrak{g}_{\alpha}=\{\xi\in\mathfrak{g}\,;\,ad(h)\xi\,=\,\alpha(h)\xi, \quad\forall h\in \mathfrak{h}\}\) be the root space of \(\alpha\in\Phi\).   Then  \(\dim_{\mathbf{C}}\,\mathfrak{g}_{\alpha}=1\).    
  Let \(\Phi_{\pm}\) be the set of positive ( respectively negative )  roots of \(\mathfrak{g}\) and put 
\[\mathfrak{e}=\sum_{\alpha \in \Phi_{+}}\,\mathfrak{g}_{\alpha}\,,\quad \mathfrak{f}=\sum_{\alpha \in \Phi_{-}}\,\mathfrak{g}_{\alpha}\,.\]
   Fix a standard set of generators \(\,h_i\in \mathfrak{h}\,, \,e_i\in \mathfrak{g}_{\alpha_i}\),  \(f_i\in \mathfrak{g}_{-\alpha_i}\).      
\(\mathfrak{g}\) is generated by the set 
\( \{e_i,\,f_i,\,h_i\,;\,i=1,\cdots,l\,\}\), and these generators satisfy the relations:
\begin{equation}\label{S1}
[\,h_i,\,h_j\,]\,=\,0\,, \quad   [\,e_i\,,\,f_j\,] \,=\,\delta_{ij}h_i\,,\quad
  [\,h_i\,,\,e_j\,]\,=\,c_{ji}e_j\,,\quad
[\,h_i\,,\,f_j\,]\,=\,-\,c_{ji}f_j\,.
\end{equation}
This is a presentation of \(\mathfrak{g}\) by generators and relations which depend only on the root system \(\Phi\).     
     The triangular decomposition of the simple Lie algebra \(\mathfrak{g}\) becomes 
\(\mathfrak{g}=\mathfrak{f}+ \mathfrak{h}+ \mathfrak{e}\), direct sum , with the space of positive root vectors \(\mathfrak{e}\) and  
the space of negative root vectors \(\mathfrak{f}\).   
 
\( \mathfrak{g}\) is considered as a quarternion Lie subalgebra of  the \(\mathfrak{g}\)-current algebra \(\,\mathcal{L}\mathfrak{g}\,\); 
\begin{eqnarray}\label{includ}
& i\,:\,\mathfrak{g}\,\ni \,X\,\longrightarrow \,\phi^{+(0,0,1)}\otimes X\,\in\,\mathcal{L}\mathfrak{g}\,,\\[0.2cm]
& \left[\phi^{+(0,0,1)}\otimes X,\,\phi^{+(0,0,1)}\otimes Y\right]_{\mathcal{L}\mathfrak{g}} =
\left[X,Y\right]_{\mathfrak{g}}\,.\nonumber
\end{eqnarray}
We adopt the following abbreviated notations:   
For \(\phi_{i}\in \mathcal{L}\). \(\,x_{i}\in \mathfrak{g}\,\), 
\(i=1,\cdots,t\),  we put
\begin{eqnarray}
 x_{12\cdots t}&=&[\, x_{1},\,[\, x_{2},\,[\,\cdots\,\cdots\, x_{t}\,]\,]\cdots\,]\,,\nonumber\\[0.2cm]
\,\phi_{12\cdots t}\ast x_{12\cdots t}&=&[\phi_{1}\otimes x_{1},\,[\phi_{2}\otimes x_{2},\,[\,\cdots\,\cdots\,,\,\phi_{t}\otimes x_{t}\,]\,]\cdots\,]\,.\label{abbreviation}
\end{eqnarray}
Every element of \(\mathcal{L}\mathfrak{g}\) is expressed by a linear combination of  \(\,\phi_{12\cdots t}\ast x_{12\cdots t}\)'s.       
We have a projection from \(\mathcal{L}\mathfrak{g}\) to \(\mathfrak{g}\) that extends the correspondence:
\begin{equation}\label{project}
\pi:\,\mathcal{L}\mathfrak{g}\ni \,\phi_{12\cdots t}\ast x_{12\cdots t}\,\longrightarrow \,
 x_{12\cdots t}\,\in \mathfrak{g}.
 \end{equation}
 It is obtained by letting all \(\phi_i\)'s in (\ref{abbreviation}) equal to \(\phi^{+(0,0,1)}\).  
 
\subsubsection{ The adjoint representation \(ad_{\mathcal{K}\mathfrak{h}}:\,\mathcal{K}\mathfrak{h}\longrightarrow\,End(\mathcal{L}\mathfrak{g})\)}


We shall investigate the triangular decomposition of \(\mathfrak{g}\)-current algebra \(\mathcal{L}\mathfrak{g}\).  
 \begin{definition}~~~
 Let \(\mathcal{L}\,\mathfrak{h}\), 
 \(\,\mathcal{L}\,\mathfrak{e}\) and \(\mathcal{L}\,\mathfrak{f}\) respectively be the  Lie subalgebras of  the \(\mathfrak{g}\)-current algebra \(\mathcal{L}\mathfrak{g}\) that are generated by \(\,\mathcal{L}\otimes_{\mathbf{R}}\mathfrak{
 h}\,\), \(\,\mathcal{L}\otimes_{\mathbf{R}}\mathfrak{e}\,\) and \(\,\mathcal{L}\otimes_{\mathbf{R}}\mathfrak{f}\,\) respectively. \\
 Let \(\mathcal{K}\, \mathfrak{h}\) 
 and \(\,\mathcal{J}\, \mathfrak{h}\)
  be the  Lie subalgebra of \(\mathcal{L}\,\mathfrak{g}\,\) generated by  \(\mathcal{K}\otimes_{\mathbf{R}}\mathfrak{h}\,\)
   and \(\mathcal{J}\otimes_{\mathbf{R}}\mathfrak{h}\) respectively.
 \end{definition}
 
  \(\mathcal{L}\mathfrak{e}\) 
consists of linear combinations of elements of the form \(\phi_{12\cdots t}\ast e_{12\cdots t}\) for \(\phi_j\in\mathcal{L}\) and \(e_j\in\mathfrak{e}\), \(j=1,2,\cdots,\,t\).   Similarly \(\,\mathcal{L}\mathfrak{f}\,\) is generated by \(\phi_{12\cdots t}\ast f_{12\cdots t}\) with \(\phi_j\in\mathcal{L}\,\) and \(\,f_j\in\mathfrak{f}\), \(j=1,2,\cdots,\,t\,\).     
 Later we shall see that \(\mathcal{L}\mathfrak{e}=\mathcal{L}\otimes_{\mathbf{R}}\mathfrak{e}\) and \(\mathcal{L}\mathfrak{f}=\mathcal{L}\otimes_{\mathbf{R}}\mathfrak{f}\), viewed as real Lie algebras.      
 
 \begin{lemma}
 It holds that 
  \begin{equation}
  [\phi\otimes x,\,\psi\otimes y]=(\phi\psi)\otimes [\,x,\,y\,]\,
  \end{equation}
  for any \(\phi\in\mathcal{K}\,, \psi\in\mathcal{L}\,\), and \(x,\,y\in\mathfrak{g}\,\).
  \end{lemma}
  This is an immediate consequence of  (\ref{Kcommute}) and (\ref{glbracet}).
  
 \begin{lemma}\label{heigen}~~~
 \begin{enumerate}
 \item
 \begin{equation}\label{Kh}
 \mathfrak{h}\subset \mathcal{K}\mathfrak{h}\,\quad\mbox { and }\quad\,\mathcal{K}\mathfrak{h}=\,\mathcal{K}\otimes_{\mathbf{R}}\mathfrak{h}\,.
 \end{equation}
 \item
 \(\mathcal{K}\,\mathfrak{h}\) 
 is a commutative subalgebra of \(\mathcal{L}\,\mathfrak{g}\,\), and 
 \(N(\mathcal{K}\mathfrak{h})\,=\,\mathcal{K}\mathfrak{h}\,\).     That is, \( \mathcal{K}\mathfrak{h}\) is a Cartan subalgebra of \(\mathcal{L}\mathfrak{g}\,\), where \(N(\mathcal{K}\mathfrak{h})=\{\,\xi\in\mathcal{L}\mathfrak{g};\,[\kappa,\xi]\in\mathcal{K}\mathfrak{h},\,\forall\kappa\in\mathcal{K}\mathfrak{h}\,\}\) is the normalizer of \(\mathcal{K}\mathfrak{h}\).   
  \item
  \begin{equation*}
 [\,\mathcal{K}\,\mathfrak{h}\,,\,\mathcal{L}\mathfrak{h}\,]\,=\,0\,, \quad
 [\,\mathcal{K}\,\mathfrak{h}\,,\,\mathcal{L}\,\mathfrak{e}\,]\,=\,\mathcal{L}\mathfrak{e}\,,\quad
 [\,\mathcal{K}\,\mathfrak{h}\,,\,\mathcal{L}\,\mathfrak{f}\,]\,=\,\mathcal{L}\,\mathfrak{f}\,.
 \end{equation*}
 \end{enumerate}
 \end{lemma}
 \begin{proof}~~~
Let \(\phi_i\in\mathcal{K}\) and \(h_i\in\mathfrak{h}\), \(i=1,2\).     We have \(\,[\phi_1\otimes h_1\,,\,\phi_2\otimes h_2\,]\,=\,(\phi_1\phi_2)\,[h_1,h_2]=0\,\).    So \(\mathcal{K}\mathfrak{h}=\mathcal{K}\otimes_{\mathbf{R}} \mathfrak{h}\), and \(\mathcal{K}\mathfrak{h}\) is a commutative subalgebra of \(\mathcal{L}\mathfrak{g}\,\).      
  Now the first assertion follows from the definitions; \(\phi^{+(0,0,1)}\otimes\mathfrak{h}\subset \mathcal{K}\mathfrak{h}\).    We shall prove 
 \(N(\mathcal{K}\mathfrak{h})\,=\,\mathcal{K}\mathfrak{h}\,\).      Let \(\psi\otimes x\in( \mathcal{L}\otimes \mathfrak{g})\cap\,N(\mathcal{K}\mathfrak{h})\).     By hypothesis 
 \([\phi\otimes h,\,\psi\otimes x]=(\phi\psi)\otimes [h,x]\) is in \(\mathcal{K}\mathfrak{h}=\mathcal{K}\otimes\mathfrak{h}\) for any \(\phi\in\mathcal{K}\) and \(h\in\mathfrak{h}\).   Then \(\phi\psi\in \mathcal{K}\) for all \(\phi\in\mathcal{K}\), so \(\psi\in N(\mathcal{K})\), which implies \(\psi\in\mathcal{K}\).  And also  \([h,x]\in \mathfrak{h}\) for all \(h\in\mathfrak{h}\).    \(\mathfrak{h}\) being a Cartan subalgebra it follows \(x\in\mathfrak{h}\).     Hence \(\psi\otimes x\in \mathcal{K}\mathfrak{h}\).      \(N(\mathcal{K}\mathfrak{h})\) being generated by \(( \mathcal{L}\otimes \mathfrak{g})\cap\,N(\mathcal{K}\mathfrak{h})\),  it follows  
 \(N(\mathcal{K}\mathfrak{h})\,=\mathcal{K}\mathfrak{h}\).    We proceed to the proof of the 3rd assertion.   
 Let \(\phi\otimes h\in\mathcal{K}\otimes\mathfrak{h}\) and \(\psi\otimes h^{\prime}\in \mathcal{L}\otimes\mathfrak{h}\) with 
\(\phi\in \mathcal{K}\,,  \psi\in \mathcal{L}\) and \(h\,,\,h^{\prime}\in\mathfrak{h}\).    We have 
\(
[\phi\otimes h,\,\psi\otimes h^{\prime}]=(\phi\psi)\otimes\,[h.h^{\prime}]=0\,
\).
Jacobi identity yields \([\,\phi\otimes h,\,[\psi_1\otimes h_1,\,\psi_2\otimes h_2]\,]=0\) for \(\psi_i\in\mathcal{L},\,h_i\in \mathfrak{h}\), \(i=1,2\), and \(\,[\,\phi\otimes h,\,\psi_{12\cdots t}\ast h_{12\cdots t}\,]=0\,\).   Hence \([\,\mathcal{K}\,\mathfrak{h}\,,\mathcal{L}\mathfrak{h}\,]=0\) .   Let \(\psi \otimes e_j\in\mathcal{L}\otimes \mathfrak{e}\).    We have \([\phi\otimes h_i\,,\,\psi\otimes e_j\,]\,
 =\,(\phi\psi)\otimes [h_i,e_j]\,=(\phi\psi)\otimes c_{ji}e_j\,\in \mathcal{L}\mathfrak{e}\,\).    The similar argument with Jacobi identity yields 
 \begin{equation}\label{eigen}
 [ \phi\otimes h_i,\,\psi_{j_1\cdots j_t}\ast e_{j_1\cdots j_t}\,]\,=\,\left (c_{j_1i}+\cdots c_{j_ti}\right )(\phi \psi_{j_1}\psi_{j_2}\cdots \psi_{j_t})\otimes e_{j_1\cdots j_t} \in \mathcal{L}\mathfrak{e}\,.
 \end{equation}
   So we have \([ \phi\otimes h_i\,,\, \mathcal{L}\mathfrak{e}\,]\subset  \mathcal{L}\mathfrak{e}\), hence \([\,\mathcal{K}\mathfrak{h},\,\mathcal{L}\mathfrak{e}\,]\subset  \mathcal{L}\mathfrak{e}\).    Similarly \(\,[\,\mathcal{K}\,\mathfrak{h},\,\mathcal{L}\mathfrak{f}\,]\,\subset  \mathcal{L}\mathfrak{f}\,\).    Conversely any element 
\(\psi_{j_1\cdots j_t}\ast e_{j_1\cdots j_t}\,\in\,\mathcal{L}\mathfrak{e}\) satisfies the relation (\ref{eigen}) for all \(\phi\otimes h\in \mathcal{K}\mathfrak{h}\) with non-zero \(
 (c_{j_1i}+\cdots c_{j_ti} )\) hence \([\,\mathcal{K}\mathfrak{h}\,,\,   \mathcal{L}\mathfrak{e}\,]\,=\,\mathcal{L}\mathfrak{e}\,\).    Similarly \([\,\mathcal{K}\mathfrak{h}\,,\,   \mathcal{L}\mathfrak{f}\,]\,=\,\mathcal{L}\mathfrak{f}\,\).   
\end{proof}

Let 
\(i:\mathfrak{h}\hookrightarrow\mathcal{K}\mathfrak{h}\) be the embedding (\ref{includ}).   Let \(\,\pi:\mathcal{L}\mathfrak{g}\longrightarrow\mathfrak{g}\) be the projection (\ref{project}) and  \(\pi_o: \mathcal{L}\longrightarrow\mathbf{C}\) be the projection to the homogeneous degree \(0\) terms i.e. the trace of Laurent polynomial type spinors,(\ref{qtrace1}).

 \(\mathcal{K}\mathfrak{h}\) being a Cartan subalgebra of \(\mathcal{L}\mathfrak{g}\),  we shall investigate the adjoint representation 
 \(ad_{\mathcal{K}\mathfrak{h}}\,\in End_{\mathbf{R}}(\mathcal{L}\mathfrak{g})\,\)  
  and the associated weight space decomposition.    
The adjoint representation \(ad_{\mathcal{K}\mathfrak{h}}\) is written as follows:  \begin{eqnarray}\label{adK}
 ad_{\phi\otimes h}\,(\psi\otimes x)&=&\,(\phi\, \psi)\,\otimes ad_hx\,,\\[0.2cm]
ad_{\phi\otimes h}(\psi_{1\cdots m}\ast x_{1\cdots m})&=&
\sum_{i=1}^m\,[\psi_1\otimes x_1, [\psi_2\otimes x_2,\,\cdots[\,(\phi\psi_i )\otimes ad_hx_i\,, [\,\psi_{i+1}\otimes x_{i+1},\cdots,\psi_m\otimes x_m]\cdots],\nonumber
\end{eqnarray}
for \(\phi\otimes h\in\mathcal{K}\mathfrak{h}\,\) and  \(\psi\otimes x\,\), \(\,
\psi_{1\cdots m}\ast x_{1\cdots m}\in \mathcal{L}\mathfrak{g}\,\).   

To a \(\lambda\,\in Hom_{\mathbf{R}}(\mathcal{K}\mathfrak{h}\,, \mathcal{L})\) there corresponds 
  \(\alpha=\pi_o\circ \lambda\circ i\,\in Hom(\,\mathfrak{h},\,\mathbf{C}) =\mathfrak{h}^{\ast}\).     Conversely given \(\alpha\in \mathfrak{h}^{\ast}\), we have \(\lambda \,\in Hom_{\mathbf{R}}(\mathcal{K}\mathfrak{h}\,, \mathcal{L})\) that associates  to \(\kappa=\phi\otimes h\in \mathcal{K}\mathfrak{h}\,\) 
 with an element 
\(\alpha(h)\phi\in \mathcal{L}\).   Then  \(\lambda(ih)=\alpha(h)\phi^{+(0,0,1)}\), that is,  \(\alpha=\pi_o\circ\lambda\circ i\).   
We put 
\begin{equation}\label{dualofKh}
Hom_{\mathbf{R}}(\,\mathcal{K}\mathfrak{h}\,,\,\mathcal{L})\,=\,\left\{\lambda\,: \mathcal{K}\mathfrak{h}\,\longrightarrow \mathcal{L}\,,\quad \lambda(\kappa)= \pi_o\circ \lambda\circ i(\,h)\,\phi,\quad \forall \kappa=\phi\otimes h\,\right\}.
\end{equation}
 \(\lambda(\kappa)\) yields a one dimensional representation of the Lie algebra \(\mathcal{L} \mathfrak{g}\,\), that is 
 given by the multiplication of \(\lambda(\kappa)=\alpha(h)\phi\in \mathcal{L}\,\):
 \[\mathcal{L}\mathfrak{g}\ni \xi= \psi\otimes x\longrightarrow \lambda(\kappa)\xi=\alpha(h)\,(\phi\psi)\otimes x\,\in \mathcal{L}\mathfrak{g}\,,\qquad  \kappa=\phi\otimes h\,.\]

 For each \(\lambda\in Hom_{\mathbf{R}}(\,\mathcal{K}\,\mathfrak{h}\,,\,\mathcal{L}) \), we define  
 \begin{equation}
  (\mathcal{L}\mathfrak{g})_{\lambda}=\{\xi\in \mathcal{L}\mathfrak{g}\,;\quad 
   ad_{\kappa}\,\xi\,=\,\lambda(\kappa)\,\xi\,,\quad \forall \kappa\in\mathcal{K}\mathfrak{h}\}\,.
  \end{equation}
  \(\lambda\in Hom_{\mathbf{R}}(\,\mathcal{K}\mathfrak{h}\,,\,\mathcal{L}) \) is called a {\it weight} of \(ad_{\mathcal{K}\mathfrak{h}}\) whenever  \((\mathcal{L}\mathfrak{g})_{\lambda}\neq 0\).    \((\mathcal{L}\mathfrak{g})_{\lambda}\) is called the {\it weight space of weight \(\lambda\)}.      The set of the non-zero weights is denoted by 
   \[{\Phi}_{\mathcal{L}}=\,\{\lambda\in \, Hom_{\mathbf{R}}(\,\mathcal{K}\mathfrak{h}\,,\,\mathcal{L}) \,;\,\lambda\neq 0\,\}\,.\]
The relevance of the root space decomposition \(\mathfrak{g}=\sum_{\alpha \in \Phi}\,(\mathfrak{g})_{\alpha}\,\) is shown by the following proposition.

\begin{proposition}~~~
 \(\mathcal{L}\mathfrak{g}\) is the direct sum of the weight spaces:
\begin{equation}\label{Khdecomp}
\mathcal{L}\mathfrak{g}\,=\,(\mathcal{L}\mathfrak{g})_0\,\oplus_{\lambda\in\Phi_{\mathcal{L}}}
\,(\mathcal{L}\mathfrak{g})_{\lambda}\,.
\end{equation}
\end{proposition}

We have 
\begin{equation}
ad_{\kappa}\,[\,\xi_1\,,\,\xi_2\,]\,=\,[\,ad_{\kappa}\xi_1\,,\,\xi_2\,]\,+\,[\,\xi_1\,,\,ad_{\kappa}\xi_2\,]\, ,
\end{equation}
for all \(\kappa\in\mathcal{K}\mathfrak{h}\,\), \(\xi_i\in\mathcal{L}\mathfrak{g}\), \(i=1,2\).
This follows inductively from the definition of \(ad_{\mathcal{K}\mathfrak{h}}\,\), (\ref{adK}).   \\

It holds that if \(\xi,\,\eta\in \,\mathcal{L}\mathfrak{g}\) are weight vectors of weights \(\lambda,\,\mu\) then \([\xi\,,\,\eta]\) is a weight vector of weight 
\(\lambda+\mu\,\):   
\begin{equation}
[\,(\mathcal{L}\mathfrak{g})_{\lambda}\,,\,(\mathcal{L}\mathfrak{g})_{\mu}\,]\,\subset\,(\mathcal{L}\mathfrak{g})_{\lambda+\mu}\,.
\end{equation}

\begin{proposition}\label{extofad}
The adjoint representation \(ad_{\mathfrak{h}}\) of \(\mathfrak{g}\) extends to the adjoint representation  \(ad_{\mathcal{K}\mathfrak{h}}\) of \(\mathcal{L}\mathfrak{g}\)
.
\end{proposition}
\begin{proof}~~~ 
 \(\phi^{+(0,0,1)}\in\mathcal{K}\) and 
the abbreviation \(\phi^{+(0,0,1)}\otimes \mathfrak{h}\,\simeq \mathfrak{h}\) imply the embedding \(i:\,\mathfrak{h}\longrightarrow\,\mathcal{K}\mathfrak{h}\).    The  adjoint representation  \(ad_{\mathcal{K}\mathfrak{h}}\) restricts to the adjoint representation  of \(\mathfrak{h}\) on \(\mathfrak{g}\) if we take \(\phi=\psi=\phi^{+(0,0,1)}\) in (\ref{adK}).    Then we have 
\begin{equation}
ad_h\circ \pi\,=\,
\pi \circ ad_{ih}\,,\quad \forall h\in \mathfrak{h}.
\end{equation}
Conversely we see from (\ref{adK}) that the action of the representation \(ad_{\mathcal{K}\mathfrak{h}}\) on \(\mathcal{L}\mathfrak{g}\)  comes from \(ad_{\mathfrak{h}}\in End(\mathfrak{g})\).    
If \(ad_{h}y=0\,\) for a \(h\in\mathfrak{h}\) and a \(y\in\mathfrak{g}\,\)  then \(ad_{\phi\otimes h}\psi\otimes y=0\) for all \(\phi\in \mathcal{K}\) and \(\psi\in\mathcal{L}\) .   In fact, since \([\mathcal{K},\mathcal{L}]=0\) we have  \([\phi\otimes h,\,\psi\otimes y]=(\phi\cdot\psi)\otimes [h,y]=0\).
\end{proof}

\begin{theorem}~~~
\begin{enumerate}
\item
The root spaces of the adjoint representation \(ad_{\mathcal{K}\mathfrak{h}}\) on \(\mathcal{L}\mathfrak{g}\) and that of \(ad_{\mathfrak{h}}\) on \(\mathfrak{g}\) correspond bijectively:  \(\Phi_{\mathcal{L}}\simeq\Phi\).
\item
For \(\lambda\in\Phi_{\mathcal{L}}\,\), hence \(\alpha=\pi_o\circ\lambda\circ i\in\mathfrak{h}^{\ast}\), 
\begin{equation}\label{weightvectors}
(\mathcal{L}\mathfrak{g})_{\lambda}\,=\,\mathcal{L}\otimes\,\mathfrak{g}_{\alpha}\,.
\end{equation}
And 
\begin{equation}\label{weight0vectors}
(\mathcal{L}\mathfrak{g})_0\,=\,\mathcal{L}\mathfrak{h}\,= \mathcal{K}\mathfrak{h}\,\oplus\,\mathcal{J}\mathfrak{h}\,.
\end{equation}
\item
\(\mathcal{L}\mathfrak{g}\) is the direct sum of the weight spaces:
\begin{equation}
\mathcal{L}\mathfrak{g}\,=\,\mathcal{L}\mathfrak{h}\ \oplus\,\oplus_{\lambda\in \Phi}\,(\mathcal{L}\otimes \,\mathfrak{g}_{\lambda})\,.
\end{equation}

\end{enumerate}
\end{theorem}
 \begin{proof} 
 Let \(\lambda\in \Phi_{\mathcal{L}}\,\).   There exists a weight vector \(\xi\in\mathcal{L}\mathfrak{g}\) with the weight \(\lambda\):    
 \(\,[\phi\otimes h\,,\,\xi\,]\,=\,\lambda(\phi\otimes h)\xi\,\) for any \( \phi\otimes h\in\mathcal{K}\mathfrak{h}\,\).   We define  \(\check\lambda\in Hom(\,\mathfrak{h}\,, \mathbf{R}\,)\) by the formula 
  \( \check\lambda(h)=\lambda(\phi^{+(0,0,1)}\otimes h)\).    Then \(\check\lambda\) becomes a root of the representation \(ad_{\mathfrak{h}}\) on \(\mathfrak{g}\):
   \(\,[h,x]=[\, \phi^{+(0,0,1)}\otimes h\,,\,\phi^{+(0,0,1)}\otimes x\,]\,=\,\check\lambda(h)x \,\).       Conversely let \( \xi=\psi_{1\cdots m}\ast x_{1\cdots m}\in \mathcal{L}\mathfrak{g}\).   We  suppose that each \(x_i\in \mathfrak{g}\) is a weight vector with root \(\beta_i\in \Phi\), \(i=1,\cdots,m\).    General elements of \(\mathcal{L}\mathfrak{g}\) are linear combinations of such vectors.    It follows from (\ref{adK}) that 
 \[ad_{\phi\otimes h}\xi=\,\left(\Sigma_{i=1}^m\beta_i(h)\phi\right)\,\xi\,,\quad\forall \phi\otimes h\in\mathcal{K}\mathfrak{h}.\]
Hence \(\Sigma_{i=1}^m\beta_i(h)\phi \in \Phi_{\mathcal{L}}\), and \(\xi\) is a weight vector of \(ad_{\phi\otimes h}\).   The relation extends linearly to \(
  \mathcal{L}\mathfrak{g}\).    Thus we have proved the first assertion.   From (\ref{adK}) we have 
\( \mathcal{L}\otimes\mathfrak{g}_\alpha\,\subset \,(\mathcal{L}\mathfrak{g})_{\alpha}\,\) for any \(\alpha\in\Phi\).     
Lemma \ref{heigen} shows that 
    \(\mathcal{L}\mathfrak{h}\,\subset \,(\mathcal{L}\mathfrak{g})_{0}\).   Then        
(\ref{eigen})  yields that  \(\phi_{i_1i_2\cdots i_t}\otimes e_{i_1i_2\cdots i_t}\) and  \(\phi_{i_1i_2\cdots i_t}\otimes f_{i_1i_2\cdots i_t}\) are 
 weight vectors.   Thus all Lie products of generators 
\(\left\{\,\phi\otimes e_i\,,\,\phi\otimes f_i\,,\,\phi\otimes h_i\,;\, \phi\in \mathcal{L}\,,\, i=1,\cdots,l\,\,\right\}\) 
 are weight vectors.    Since every element of \(\mathcal{L}\mathfrak{g}\) is a linear combination of products of these weight vectors we deduce from (\ref{Khdecomp}) and the fact \(\Phi\simeq \Phi_{\mathcal{L}}\) that  
 \begin{equation}\label{sum}
 \mathcal{L}\mathfrak{g}\,=\,(\mathcal{L}\mathfrak{g})_0\,\oplus\,\oplus_{\alpha\in\Phi}(\mathcal{L}\mathfrak{g})_{\alpha}\,.
 \end{equation}
    Now the simple roots \(\alpha_1,\cdots,\alpha_l\in \Phi\) are linearly independent, so the only monomials which have weight \(\alpha_j\) are the weight vectors of \(\mathcal{L}\otimes\mathfrak{g}_{\alpha_j}\).   We conclude 
   \begin{equation}
   \,(\mathcal{L}\mathfrak{g})_{\alpha_j}\,=\,\mathcal{L}\otimes_{\mathbf{C}}\mathfrak{g}_{\alpha_j}\,.
   \end{equation}
 Hence \( \,(\mathcal{L}\mathfrak{g})_{\alpha}\,=\,\mathcal{L}\otimes_{\mathbf{C}}\mathfrak{g}_{\alpha}\,\) for all \(\alpha\in \Phi\).    Therefore (\ref{sum}) becomes 
 \begin{equation}
 \mathcal{L}\mathfrak{g}\,=\,(\mathcal{L}\mathfrak{g})_{0}\,\oplus\,\oplus_{\alpha\in\Phi}(\mathcal{L}\otimes_{\mathbf{C} }\mathfrak{g}_{\alpha})\,.
 \end{equation}
Now we shall prove \((\mathcal{L}\mathfrak{g})_{0}=\mathcal{L}\mathfrak{h}\,\).   We regard \(\mathcal{L}\mathfrak{g}\) as a \(\mathcal{K}\mathfrak{h}\)-module.   Hence  \(\mathcal{L}\mathfrak{h}\) is a \(\mathcal{K}\mathfrak{h}\)-submodule.
\(\mathcal{L}\mathfrak{h}\) is contained in  \((\mathcal{L}\mathfrak{g})_0\) by Lemma \ref{heigen}.   If \(\,\mathcal{L}\mathfrak{h}\neq (\mathcal{L}\mathfrak{g})_0\,\) the 
\( \mathcal{K}\mathfrak{h}\)-module 
\( (\mathcal{L}\mathfrak{g})_0/\mathcal{L}\mathfrak{h} \) will have a 1-dimensional submodule \(M/\mathcal{L}\mathfrak{h}\) on which \(\mathcal{K}\mathfrak{h}\) acts with weight \(0\).   That is, \([\,\mathcal{K}\mathfrak{h},\,M/
\mathcal{L}\mathfrak{h}\,]=0\).   Then  \([\,\mathcal{K}\mathfrak{h},\,M\,]\subset\,
\mathcal{L}\mathfrak{h}\,\) and  \(M\) is a
 \( \mathcal{K}\mathfrak{h}\)-submodule of \(\mathcal{L}\mathfrak{h}\).   That is a contradiction.
\end{proof}

  We know that any weight \(\lambda\in \Phi\) is of the form \(\sum_{i=1}^l\,k_i\alpha_i\), \(k_i\in\mathbf{Z}\).   Moreover a non-zero weight \(\lambda\) has the form \(\lambda=\sum_{i=1}^lk_i\alpha_i,\,k_i\in \mathbf{Z}\), with all \(k_i\geq 0\) or all \(k_i\leq 0\).     Therefore 
  \begin{eqnarray}
  \mathcal{L}\mathfrak{e}&=&\sum_{\lambda\in\Phi^+}\mathcal{L}\otimes_{\mathbf{R}}\mathfrak{g}_{\lambda}\\[0.2cm]
  \mathcal{L}\mathfrak{f}&=&\sum_{\lambda\in\Phi^-}\mathcal{L}\otimes_{\mathbf{R}}\mathfrak{g}_{\lambda}
  \end{eqnarray}

From the above discussion we have the following 
\begin{theorem}\label{triangulardecomp}
The \(\mathfrak{g}\)-current algebra \(\mathcal{L}\mathfrak{g}\) has the following triangular decomposition 
 \begin{eqnarray*}
\mathcal{L}\mathfrak{g}&=&\,\mathcal{L}\mathfrak{e}\,\oplus\,\mathcal{L}\mathfrak{h}\,\oplus\,
\mathcal{L}\mathfrak{f}\,.
\\[0.2cm]
\mathcal{L}\mathfrak{e}\,&=&\,
\mathcal{L}\otimes\,\mathfrak{e}\,,\quad \mathcal{L}\mathfrak{f}\,=\, 
\mathcal{L}\otimes\,\mathfrak{f}\,,\\[0.2cm]
\mathcal{L}\mathfrak{h}\,&=&\,\mathcal{K}\mathfrak{h}\,\oplus\,\mathcal{J}\mathfrak{h}\,.
\end{eqnarray*}
\end{theorem}

It follows from (\ref{Kh}) that 
\begin{corollary}\label{nonproduct}
\begin{equation}
\mathcal{L}\mathfrak{g}\ominus\,(\,\mathcal{L}\otimes\mathfrak{g}\,)\,=\,\mathcal{J}\mathfrak{h}\,.
\end{equation}
\end{corollary}

\subsubsection{ Symmmetric invariant bilinear form on \(\mathcal{L}\mathfrak{g}\)}

The  symmetric invariant bilinear form on \(\mathcal{L}\) is defined by
\begin{equation}
(\phi_1\,\vert\,\phi_2\,)\,=\,qRes\,(\phi_1\phi_2\,)\,.
\end{equation}
from Definition \ref{qres}.     
The invariant bilinear form on \(\mathcal{L}\mathfrak{g}\) is given by the following formula:
\begin{equation}\label{invariantform}
\left( \phi_1\otimes x_1\,\vert \,\phi_2\otimes x_2\right)\,=\,(\phi_1\vert \phi_2)\,(x_1\vert x_2)\,,\quad \mbox{ for } \,\phi_i\in\mathcal{L},\,x_i\in \mathfrak{g}\,,\,i=1,2\,,\end{equation}
where \((x_1\vert x_2)\) is the Killing form of \(\mathfrak{g}\).   
In fact, we have
 \begin{eqnarray*}
 &(\,[\phi_1\otimes x_1,\,\phi_2\otimes x_2\,]\,\vert\, \phi_3\otimes x_3\,)\,
 =\,(\,\phi_1 \phi_2\otimes\,x_1x_2\,\vert\, \phi_3\otimes x_3\,)\,-\,(\phi_2\phi_1\otimes x_2x_1\,\,\vert\, \phi_3\otimes x_3\,)\\[0.2cm]
 &=\,(\,\phi_1\phi_2\,\vert\,\phi_3 )(x_1x_2\vert  x_3\,)
\,-\, 
(\,\phi_2\phi_1\vert\, \phi_3)(\,x_2x_1\vert x_3\,)\\[0.2cm]
&=
\,(\,\phi_1\vert\,\phi_2\phi_3 )(x_1\vert x_2x_3\,)
\,-\, (\,\phi_1\vert\,\phi_3\phi_2 )(x_1\vert\,x_3 x_2\,)\\[0.2cm]
&=
\,(\,\phi_1\otimes x_1\vert\,\phi_2\phi_3\otimes x_2\,x_3\,-\,\phi_3\phi_2\otimes\,x_3x_2\,),
\end{eqnarray*}
where the calculation relies on the fact that the Lie algebra \(\mathfrak{g}\) is a subalgebra of \(\mathfrak{gl}(n,\mathbf{C})\), so in particular \((x_1x_2\vert x_3)=(x_1\vert x_2x_3)\).   \\
 The bilinear form \(\,(\xi\,\vert\,\eta\,)\),\(\,\xi,\,\eta\in\mathcal{L}\mathfrak{g}\),  is non-degenerate.   
As an immediate consequences we have the following
\begin{proposition}~~~
\begin{enumerate}
\item
\((\mathcal{L}\mathfrak{g})_\lambda\) and \((\mathcal{L}\mathfrak{g})_\mu\) are orthogonal with respect to the bilinear form (\ref{invariantform}) unless \(\mu+ \lambda=0\).
\item
\begin{equation}
(\,\mathcal{K}\mathfrak{g}\,\vert\,\mathcal{J}\mathfrak{g}\,)\,=0\,.
\end{equation}
\end{enumerate}
\end{proposition}
\begin{proof}
Suppose \(\lambda+\mu\neq 0\) and 
let \( \xi\in (\mathcal{L}\mathfrak{g})_{\lambda}\), \(\eta\in (\mathcal{L}\mathfrak{g})_{\mu}\).   Choose \(\kappa\in\mathcal{K}\mathfrak{h}\) with \((\lambda+\mu)(\kappa)\neq 0\).   Then 
\[ (\,[\xi\,,\kappa\,]\,\vert\,\eta\,)=(\,\xi\,\vert [\kappa\,,\,\eta\,]\,)\]
implies \(-\lambda(\kappa)\,(\,\xi\vert \,\eta\,)=\mu(\kappa)\,(\,\xi\vert\,\eta\,)\,\) that is \((\lambda+\mu)(\kappa)\,(\xi\,\vert\,\eta\,)=0\,\).   Hence \((\xi\,\vert\,\eta\,)=0\,\).   The second assertion is trivial from the definition; \(q\,tr(\phi\cdot\psi)=0\) for \(\phi\in\mathcal{K}\) and \(\psi\in\mathcal{J}\).
\end{proof}~~~
Suppose \(\lambda+\mu\neq 0\)  
For any \(\alpha\in \Phi\) there is a unique element \(h_{\alpha}\in\mathfrak{h}\) such that \(\alpha(h)=(h_{\alpha}\vert h)\) for all \(h\in\mathfrak{h}\,\).     Similar assertion holds for \(\mathcal{K}\mathfrak{h}\).    In fact, let \(\lambda\in \Phi_{\mathcal{L}}\) be the weight of \(ad_{\mathcal{K}\mathfrak{h}}\,\) corresponding to \(\,\alpha=\pi_o\circ\lambda \circ i\).    Then for any \(\kappa=\phi\otimes h\in\mathcal{K}\mathfrak{h}\,\)  it holds   
\(\lambda(\kappa)=\phi\otimes\alpha(h)=(h_{\alpha}\vert h)\phi\,\).   Hence,  \(\kappa_{\lambda}=\phi^{+(0,0,1)}\otimes h_\alpha\in \mathcal{K}\mathfrak{h}\) is the unique element  that represents \(\lambda\in \Phi_{\mathcal{L}}\):  
\begin{equation}
\,\lambda(\kappa)=(\kappa_{\lambda}\,\vert\,\kappa)\quad\mbox{
 for }\, \forall\kappa\in\mathcal{K}\mathfrak{h}\,.\label{dualelement}
 \end{equation}
We know that \(h_{\alpha}=\,[x,\,y]\in [\mathfrak{g}_{\alpha}\,,\mathfrak{g}_{-\alpha}]\,\subset\mathfrak{h}\) for some \(x\in \mathfrak{g}_{\alpha}\) and \(y\in \mathfrak{g}_{-\alpha}\) , \cite{C} .    But \(\kappa_{\lambda}\) can not have an analogous formula.   Though it holds that 
 \[[\,(\mathcal{L}\mathfrak{g})_{\lambda}\,,\,(\mathcal{L}\mathfrak{g})_{-\lambda}\,]\,=(\mathcal{L}\mathfrak{g})_0=\mathcal{L}\mathfrak{h}=\mathcal{K}\mathfrak{h}\oplus\mathcal{J}\mathfrak{h}\, ,\] 
 from (\ref{weight0vectors}).

 \begin{proposition}~~~ 
 \begin{enumerate}
 \item
 Let \(\lambda\in \Phi_{\mathcal{L}}\) be a weight of \(ad_{\mathcal{K}\mathfrak{h}}\) and \(\alpha=\pi_o\circ\lambda\circ i\in \Phi\) the corresponding root of \(\,\mathfrak{g}\).     
 Then the vector \(\kappa_{\lambda}=\phi^{+(0,0,1)}\otimes h_\alpha\in \mathcal{K}\mathfrak{h}\) gives the \(\mathcal{K}\mathfrak{h}\)-component of \( [\,(\mathcal{L}\mathfrak{g})_{\lambda}\,,\,(\mathcal{L}\mathfrak{g})_{-\lambda}\,]\,\).    
 \item
 We have the relation:
 \begin{equation}
 [\,\xi\,,\,\eta\,]\,=\,\left(\,\xi\,\vert\,\eta\,\right)\,\kappa_{\lambda}\,,
 \end{equation}
 for \(\xi\in\,(\mathcal{L}\mathfrak{g})_{\lambda}\,\) and  \(\eta \in\,(\mathcal{L}\mathfrak{g})_{-\lambda}\,\).   
  \end{enumerate}
  \end{proposition}
  \begin{proof}~~~
 In fact, let \(\lambda\in\Phi_{\mathcal{L}}\) and let \(\alpha\in\Phi\) be the corresponding element: \(\alpha=\pi_0\circ\lambda\circ i\), and let   \(\mathbf{R}e_{\alpha}\) be the 1-dimensional \(\mathfrak{h}\)-submodule contained in  \(\,\mathfrak{g}_{\alpha}\).     
 We have \([h, e_{\alpha}]=\alpha(h)e_{\alpha}\) for all \(h\in \mathfrak{h}\).      Similarly for   \(\epsilon_{\lambda}=\phi^{+(0,0,1)}\otimes e_{\alpha}\in\mathcal{L}\otimes e_{\alpha}=(\mathcal{L}\mathfrak{g})_{\lambda}\) and \(\kappa=\phi\otimes h\in \mathcal{K}\mathfrak{h}\), we have 
 \( [\kappa\,,\,\epsilon_{\lambda} ]\,=\,\alpha(h)\phi\,\otimes e_{\alpha}=\lambda(\kappa)\,\epsilon_{\lambda}\) .
 Let \(y\in\mathfrak{g}_{-\alpha}\) be such that \((e_{\alpha}\vert y)\neq 0\).   Such a \(y\in \mathfrak{g}_{-\alpha}\) certainly exists.   
  Then \([e_{\alpha},\,y]\in [\mathfrak{g}_{\alpha},\mathfrak{g}_{-\alpha}]\subset \mathfrak{h}\).   For any \(\psi\in\mathcal{L}\), we have \([\epsilon_{\lambda}, \psi\otimes y]=
  \psi\otimes [e_{\alpha},y]\in 
 \mathcal{L}\mathfrak{h}\,=\mathcal{K}\mathfrak{h}\oplus \mathcal{J}\mathfrak{h}\).    We shall verify that the \(\mathcal{K}\mathfrak{h}\) part of \([\epsilon_{\lambda},\,\psi\otimes y]\) is given by \((\epsilon_{\lambda}\vert \psi\otimes y\,)\kappa_{\lambda}\).   In fact let \(\xi=[\epsilon_{\lambda}, \psi\otimes y]-(\epsilon_{\lambda}\vert \psi\otimes y\,)\kappa_{\lambda}\).    
Then
\[ (\kappa\vert \xi)=(\kappa\vert [\epsilon_{\lambda}, \psi\otimes y]\,)-\,(\epsilon_{\lambda}\vert \psi\otimes y)\,(\kappa\vert \kappa_{\lambda})\,=([\kappa,\epsilon_{\lambda}]\,\vert\,\psi\otimes y)-\lambda(\kappa)(\epsilon_{\lambda}\vert\psi\otimes y)=0 
  \]
  for any \(\kappa=\phi\otimes h\in \mathcal{K}\mathfrak{h}\).    Thus \(\xi\in \mathcal{J}\mathfrak{h}\).   Hence \((\epsilon_{\lambda}\vert \psi\otimes y\,)\kappa_{\lambda}\) is the projection of \([\epsilon_{\lambda},\,\psi\otimes y]\) to \(\mathcal{K}\mathfrak{h}\).   The first assertion is proved.   
    Now for the proof of the second assertion we consider \(\,[\,\xi\,,\,\eta\,]\,-\,\left(\,\xi\,\vert\,\eta\,\right)\,\kappa_{\lambda}\).   For all \(\kappa\in \mathcal{K}\mathfrak{h}\) we have 
\begin{eqnarray*}
& \left([\,\xi\,,\,\eta\,]\,-\,\left(\,\xi\,\vert\,\eta\,\right) \kappa_{\lambda}\,\vert\,\kappa\,\right)
=\left(\,[\,\xi,\,\eta\,]\,\vert\,\kappa\right) -\,\left(\xi\,\vert\,\eta\right)\left( \kappa_{\lambda}\,\vert\kappa\right)\,\\[0.2cm]
&=\, \left(\,\xi\,\vert\,[\,\eta,\,\kappa\,]\right)\,-\,\lambda(\kappa)\,\left(\xi\,\vert\,\eta\right)\,=0\,.
\end{eqnarray*}
  Since the form is non-degenerate on \(\mathcal{K}\mathfrak{h}\,\) we deduce 
  \([\,\xi\,,\,\eta\,]\,-\,\left(\xi\vert\eta\right)\kappa_{\lambda}=0\,.\)
  \end{proof}

\section{ Central extension of the \(\mathfrak{g}\)-current algebra}

\subsection{ Central extension of the \(\mathfrak{g}\)-current algebra \(\mathcal{L}\mathfrak{g}\) }

  Let \((V,\,[\,\cdot\,,\,\cdot\,]_V\,)\) be a quarternion Lie algebra.    A {\it central extension} of   \((V,\,[\,\cdot\,,\,\cdot\,]_V\,)\) is a quarternion Lie algebra \((W, \,[\,\cdot\,,\,\cdot\,]_W\,)\)  such that \(W=V\oplus Z\) ( direct sum ) and  \(Z\) is contained in the center of \(W\);
\[\,Z\,\subset \{w\in W\,:\,[w,x]_W=0\,, \forall x\in W\}\,,\]
and such that  \(\,[\,\cdot\,,\,\cdot\,]_W\)  restricts to  \(\,[\,\cdot\,,\,\cdot\,]_V\).   
  
  Let \(\mathfrak{g}\) be a simple Lie algebra which we suppose to be a subalgebra of a matrix algebra, and let \(\mathcal{L}\mathfrak{g}\) be the \(\mathfrak{g}\)-current algebra.      
 There exists a non-degenerate symmetric bilinear form \((\cdot\,\vert\cdot\,)\)  on  \( \mathfrak{g}\) ( Killing form ), which is given given by 
 \((x \vert y)=\,Trace\,(\,x\,y) \,\).     
 The invariance means; \(([x,y]\vert z)=(x\vert [y,z])\) for all \(x,y,z\in\mathfrak{g}\).   
 
 In Proposition \ref{cocycle}  we introduced a 2-cocycle \(A\) on the space of current \(\mathcal{L}\) that takes values in \(\mathbf{H}\).       
We extend them to the 2-cocycle on the \(\mathfrak{g}\)-current algebra   \( \mathcal{L}\mathfrak{g}\) by 
\begin{equation}\label{tricocycle}
A(\,\phi\otimes x\,,\,\psi\otimes y\,)=\,(x \vert y)\,A(\phi,\psi)\,
\end{equation}
for \(\phi,\,\psi\in \mathcal{L}\) and \(x,\,y\in\mathfrak{g}\).  Then we have a \(\mathbf{H}\)-valued bilinear form on \(\mathcal{L}\mathfrak{g}\,\)  that satisfy cocycle conditions:
\begin{eqnarray*}
& A(\,u\,,\,v\,)=\, -A(\,v\,,\,u\,)\\[0.2cm]
& A(\,[u,v]\,,\,w\,)\,+\, A(\,[v,w]\,,\,u\,)\,+ A(\,[w,u]\,,\,v\,)\,=0\qquad\mbox{ for }  \,u,v,w\in\mathcal{L}\mathfrak{g}.
\end{eqnarray*}
In fact it is enough to check these conditions for \(u=\phi\otimes x,\,v=\psi\otimes y,\,w=\pi\otimes z\,\), with \(\phi,\psi,\pi\in \mathcal{L}\) and \(x,y,z\in\mathfrak{g}\).   The first follows from (\ref{antisym}) and the symmetry of \((\cdot\,\vert\,\cdot)\).   The second property follows from (\ref{precocycle}) and  the symmetry and invariance of \((\cdot\,\vert\,\cdot)\).  Indeed we have 
 \begin{eqnarray*}
 A([u,v]\,,w)&=&A((\phi\psi)\otimes xy\,, \pi\otimes z)-A((\psi\phi)\otimes yx,  \pi\otimes z)=
 (xy\vert z)\,A(\phi\psi\,,\pi)-(yx\vert z)A(\psi\phi\,,\pi).\\[0.2cm]
 A([v,w]\,,u)&=&A((\psi\pi)\otimes yz, \phi\otimes x)-\,A((\pi\psi)\otimes zy,  \phi\otimes x)=
 (yz\vert x)A(\psi\pi,\phi)-\,(zy\vert x)A(\pi\psi,\phi).\\[0.2cm]
 A([w,u],v)&=&A((\pi\phi)\otimes zx, \psi\otimes y)-A((\phi\pi)\otimes xz,  \psi\otimes y)=
 (zx\vert y)A(\pi\phi,\psi)-(xz\vert y)A(\phi\pi,\psi).
 \end{eqnarray*}
\((\cdot\vert\cdot)\) being symmetric invariant bilinear form we have  \((xy\vert z)=(yz\vert x)=(zx\vert y)\) etc., then 
\begin{eqnarray*}
 &&A([u,v]\,,w)+A([v,w]\,,u)+A([w,u]\,,v)=\\[0.2cm]
 &&(xy\vert z)\left\{
 A(\phi\psi,\pi)+A(\psi\pi,\phi)+A(\pi\phi,\psi)\right\}
- (yx\vert z)\left\{
 A(\psi\phi,\pi)+A(\pi\psi,\phi)+A(\phi\pi,\psi)\right\}.
\end{eqnarray*}
By (\ref{precocycle}) the last formula vanishes.

     Associated to the 2-cocycle \(A\), we have the central extension of \( \mathcal{L}\mathfrak{g}\,\).    
     
 \begin{theorem}~~ Let \(\rm c\) be a indefinite number.    Put  
 \begin{equation}
 \mathcal{L}\mathfrak{g}(\rm c)\,=\, \mathcal{L}\mathfrak{g}\oplus \mathbf{H}c\,.
\end{equation}
We endow \((\mathcal{L}\otimes \mathfrak{g})\oplus \mathbf{H}c\,\) with the following bracket: 
\begin{eqnarray}\label{bracSH}
 [\,\phi \otimes x\, , \,\psi  \otimes y\,]^{\rm c}
  &=&   [ \phi\otimes x\,,\,\psi\otimes y\,]
+A(\phi \otimes x\,, \psi \otimes y){\rm c}\, \,,\nonumber
\\[0,3cm] 
 [\,{\rm c}\,, \,\phi\otimes x\,]^{\rm c}&=&0\, ,
  \end{eqnarray}
for  \(\phi\otimes x\,,\, \psi\otimes y \in \mathcal{L}\otimes\mathfrak{g}\).   The bracket is extended to \(\mathcal{L}\mathfrak{g}({\rm c})\) and 
 \(\mathcal{L}\mathfrak{g}({\rm c})\) becomes a quartenion Lie algebra with 
the conjugation automorphism \(\sigma\) extended by \(\,\sigma {\rm c}= {\rm c}\). 
\end{theorem}  
 
 We shall further complete the central extension of the current algebra \(\mathcal{L}\mathfrak{g}(\rm c)\) by adjoining the normal derivation coming from the normal vector field \(\).    First we extend \(\mathbf{n}\)  
 to an outer derivation of the 
 Lie algebra \(\mathcal{L}\mathfrak{g}\,\) by 
\begin{equation}
   \,\mathbf{n}(\phi \otimes x\,)\,=\,(\,\mathbf{n}\phi \,)\otimes x,\qquad \,\phi\in \mathcal{L}\,,\,x\in \mathfrak{g}\,.
\end{equation}  
Then we extend \(\mathbf{n}\)  further to \(\mathcal{L}\mathfrak{g}({\rm c})\) by killing the \(\rm c\,\); \(\mathbf{n}{\rm c}=0\,\).   In fact we have the following
\begin{lemma}\label{outerderivation}
Let \(\frac{\partial}{\partial n}=\frac{1}{2|z|}\mathbf{n}\) be the normal derivative.   We have 
\begin{equation}
\,[\,\frac{\partial}{\partial n}(\phi_1\otimes x_1)\,,\,\phi_2\otimes x_2\,]^{\rm c}\,+\,
[\,\phi_1\otimes x_1\,,\,\frac{\partial}{\partial n}(\phi_2\otimes x_2)\,]^{\rm c}
=\frac{\partial}{\partial n}\left(\,[\,\phi_1\otimes x_1\,,\,\phi_2\otimes x_2\,]^{\rm c}\,\,\right)\,.
\end{equation}  
\end{lemma}
\begin{proof}~~~
From  
Propositions \ref{derivationofL} and \ref{cocycleandnormalder}
we have
\begin{eqnarray*}
&&\,[\,\frac{\partial}{\partial n}(\phi_1\otimes x_1)\,,\,\phi_2\otimes x_2\,]^{\rm c}\,+\,
[\,\phi_1\otimes x_1\,,\,\frac{\partial}{\partial n}(\phi_2\otimes x_2)\,]^{\rm c}\nonumber\\[0.2cm]
&& \,=\,(\frac{\partial}{\partial n}\phi_1 \cdot\phi_2)\otimes\,x_1x_2\,-\,(\phi_2\cdot\frac{\partial}{\partial n}\phi_1 )
 \otimes\,x_2x_1\, 
  \,+\,(\phi_1 \cdot \frac{\partial}{\partial n}\phi_2)\otimes\, x_1x_2\,\, -\,( \frac{\partial}{\partial n}\phi_2\cdot\phi_1)\otimes\, x_2x_1\,\nonumber \\[0.2cm]
&&  +\,(x_1\vert x_2)\,\left ( A(\frac{\partial}{\partial n}\phi_1,\,\phi_2)\,+\,A(\phi_1,\,\frac{\partial}{\partial n}\phi_2)\right ){\rm c}\,\nonumber \\[0.2cm]
&&\,=\,\frac{\partial}{\partial n}(\phi_1\cdot\phi_2)\otimes x_1x_2\,-\,\frac{\partial}{\partial n}(\phi_2\cdot\phi_1)\otimes x_2x_1\,
=\frac{\partial}{\partial n}\left(\,[\,\phi_1\otimes x_1\,,\,\phi_2\otimes x_2\,]^{\rm c}\,\,\right)\,.
\end{eqnarray*}  
\end{proof}
We have shown that  \(\frac{\partial}{\partial n}\) acts on the Lie algebra \( \mathcal{L}\mathfrak{g}({\rm c})\).    We remark that Lemma \ref{outerderivation} is valid for the normal vector field \(|z|^{-k}\mathbf{n}\), \(k\geq 1\), but not  for \(|z|^k\mathbf{n}\), \(k\geq 0\).    \(|z|^{-k}\mathbf{ n}\),  \(k\geq 1\),  acts on the Lie algebra \( \mathcal{L}\mathfrak{g}({\rm c})\):
\[
\,[\,|z|^{-k}\mathbf{n}(\phi_1\otimes x_1)\,,\,\phi_2\otimes x_2\,]^{\rm c}\,+\,
[\,\phi_1\otimes x_1\,,\,|z|^{-k}\mathbf{n}(\phi_2\otimes x_2)\,]^{\rm c}
=|z|^{-k}\mathbf{n}\left(\,[\,\phi_1\otimes x_1\,,\,\phi_2\otimes x_2\,]^{\rm c}\,\,\right)\,.
\]

 \begin{theorem}~~~
 Let \(\,{\rm d}\) be an indefinite element. We consider the \(\mathbf{R}\)-vector space:
\begin{equation}
\widehat{\mathfrak{g}\,}\,=\, \mathcal{L}\mathfrak{g}\oplus( \mathbf{H}\, {\rm c} )\oplus (\mathbf{C}\,{\rm d})\,.\label{quatgl}
\end{equation}
We endow  \(\,\widehat{\mathfrak{g}}\,\) with the following extended bracket:  
 \begin{eqnarray}
&& [\,\phi \otimes x\, , \,\psi  \otimes y\,]_{\widehat{\mathfrak{g}} } \,=\,
  [\,\phi \otimes x\, , \,\psi  \otimes y\,]^{\rm c}    \nonumber\\[0.2cm]
  &&\quad =  \, [\,\phi \otimes x\, , \,\psi  \otimes y\,]\,+\, (x|y)\, A(\phi , \psi )\,{\rm c}\, , 
\label{braqgl}
\\[0,2cm] 
    &&\, [\,{\rm c}\,, \,\phi\otimes x\,] _{\widehat{\mathfrak{g}} }\,=0\,,\\[0.2cm]
    &&\,  [\,{\rm d}\,,\, \phi \otimes x\,] _{\widehat{\mathfrak{g}} }=\,\frac{\partial}{\partial n}\,\phi \otimes x\,, \label{derbra}\\[0.2cm]
&&   [\,{\rm d}\,,\,{\rm c}\,]_{\widehat{\mathfrak{g}} }\,=0\,,\nonumber
  \end{eqnarray}
  for \(x,y\in \mathfrak{g}\) and \(\phi,\,\psi\,\in \,\mathcal{L}\, \).      The involution \(\sigma\) is extended to \(\widehat{\mathfrak{g}}\,\) by 
\begin{equation*}
\sigma(\,\phi\otimes x)=\sigma\phi\otimes x\,,\quad \sigma {\rm c}\,=0\,, \quad \sigma {\rm n}\,={\rm n}\,.
\end{equation*}
Then we get a quarternion Lie algebra
 \( \left(\, \widehat{\mathfrak{g}}  \, , \, [\,\cdot,\cdot\,]_{\widehat{\mathfrak{g}} } \,\right) \).
\end{theorem}
\begin{proof}~~~ We write simply \([\,\,,\,\,]\) instead of  \([\,\,,\,\,]_{\widehat{\mathfrak{g}} }\,\).   It is enough to prove the following Jacobi identity:
 \begin{equation*}
 [\,[\,{\rm d}\,, \,\phi_1   \otimes x_1 \,]\,,\, \phi_2 \otimes x_2\,]
+[\,[\phi_1  \otimes x_1 , \phi_2 \otimes x_2\,]\,,\,{\rm d}\,]
\,+\,[\,[\phi_2 \otimes x_2 , \,{\rm d}\,] ,\, \phi _1   \otimes x_1\,]=0.
\end{equation*}    
From the defining equation (\ref{derbra}) and Lemma  \ref{outerderivation} 
 the sum of the 1st and the 3rd terms is equal to  
\[[\,[{\rm d},\,\phi_1\otimes x_1]\,,\phi_2\otimes x_2\,]\,+\,
[\phi_1\otimes x_1\,, [{\rm d}\,,\,\phi_2\otimes x_2]\,]\,=
\,\frac{\partial}{\partial n}\,\left(\,[\,\phi_1\otimes x_1\,,\,\phi_2\otimes x_2\,]\,\,\right)\,,
\]
which is equal to  \(\,-\,[\,[\phi_1  \otimes x_1 , \phi_2  \otimes x_2\,]\,,\,{\rm d}\,]\).
 \end{proof}
  
 \par\medskip

 \begin{proposition}
  The centralizer of \(\,{\rm d}\in \,\widehat{\mathfrak{g}}\,\) is given by
\begin{equation*}
(\,\mathcal{L}[0]\, \mathfrak{g}\,)\,\oplus \,\mathbf{H}{\rm c}\,\oplus\mathbf{C}{\rm d}\,.
 \end{equation*}
 \end{proposition}
 Here \(\,\mathcal{L}[0]\) is the subspace in \(\mathcal{L}\) generated by \(\phi_1\cdots\phi_n\) with \(\phi_i\) being \( \phi_i=\phi^{\pm(m_i,l_i,k_i)} \) such that 
 \[\sum_{i;\,\phi_i=\phi^{+(m_i,l_i,k_i)} }\,m_i-\sum_{i;\,\phi_i=\phi^{-(m_i,l_i,k_i)}}\,(m_i+3)=0\,,\]
and  \(\,\mathcal{L}[0] \mathfrak{g}\,\) is the subalgebra of \(\widehat{\mathfrak{g}}\,\) generated by \(\,\mathcal{L}[0]\,\otimes_{\mathbf{C}} \mathfrak{g}\,\).   The proposition follows from the definition (\ref{derbra}).

\begin{definition}\label{affinecurrent}
We call the quarternion Lie algebra \(\,\widehat{\mathfrak{g}}\,\)  {\it the affine current algebra over \(\mathfrak{g}\) }:
\begin{equation}
\widehat{\mathfrak{g}\,}\,=\, \mathcal{L}\mathfrak{g}\oplus\,\mathbf{H}\,{\rm c}\oplus \,\mathbf{C}\,{\rm d}\,.
\end{equation}
\end{definition}

\subsection{ Root space decomposition of the current algebra \(\,\widehat{\mathfrak{g}}\,\) }
 
Let \(\widehat{\mathfrak{g}\,}\,=\, \mathcal{L}\mathfrak{g}\oplus\,\mathbf{H}{\rm c} \oplus \mathbf{C}\,{\rm d}\,\)
 be the affine current algebra over \(\mathfrak{g}\), Definition \ref{affinecurrent}.   
We introduce the subalgebra 
\begin{equation}
\widehat{\mathfrak{h}}\,
  =\,\mathfrak{h} \oplus \mathbf{H}{\rm c}\,\oplus \,\mathbf{C} {\rm d}\,,
\end{equation}
where we applied the identification \(\mathfrak{h}\ni h\stackrel{\simeq}{\longrightarrow} \phi^{+(0,0,1)}\otimes h\in \mathcal{L}\mathfrak{g}\).     \(\widehat{\mathfrak{h}}\,\) is a commutative subalgebra of \(\widehat{\mathfrak{g}}\).     The adjoint action of  \(\widehat{\mathfrak{h}}\,\) over \(\widehat{\mathfrak{g}}\,\) is written as follows.   
From the discussion in previous sections, in particular by virtue of Theorem \ref{triangulardecomp}, 
Corollary \ref{nonproduct}, (\ref{weightvectors}) and  (\ref{quatgl}),  we see that any element \(\xi\,\in \widehat{\mathfrak{g}}\) is written in the form:
\begin{eqnarray}\label{anyLg}
\xi&=&\, x\,+\,p\,{\rm c}\,+\,q {\rm d}\,,\quad x\in\mathcal{L}\mathfrak{g}\,,\quad p \in \mathbf{H}\,,\, q\in \mathbf{C}\,\\[0.2cm]
x &=&\, y+\,\sum_{\alpha\in\Phi}\,\varphi_{\alpha}\otimes x_{\alpha}\,, \quad \varphi_{\alpha} \in
\mathcal{L}\,,\quad 
 x_{\alpha}\in \mathfrak{g}_{\alpha}\,,\nonumber
 \\[0.2cm]
  y&=& \kappa+z \in \mathcal{L}\mathfrak{h}\,,\quad \kappa\in\mathcal{K}\mathfrak{h}\,,\quad z\in \mathcal{J}\mathfrak{h}\,\nonumber
  \end{eqnarray}
   On the other hand any element of \(\,\widehat{\mathfrak{h}}\,\)  is written in the form 
   \[\,\hat h= \phi^{+(0,0,1)}\otimes h+\, s{\rm c} +\,t \,{\rm d}\,,\quad h\in \mathfrak{h}\,,\,s\in\mathbf{H}\,, t\in\mathbf{C}. \]
From Lemma \ref{heigen} follows \([\phi\otimes h\,, y\,]=0\) for any \(\phi\in \mathcal{K}\), \(h\in \mathfrak{h}\) and \(y\in \mathcal{L}\mathfrak{h}\),  in particular \([\phi^{+(0,0,1)}\otimes h\,, y\,]=0\).     
So we see that the adjoint action of \(\,\hat h=h+\, s{\rm c} +t {\rm d} \in\widehat{\mathfrak{h}}\,\) on \(\,\xi=y+\,\sum_{\alpha}\,\varphi_{\alpha}\otimes x_{\alpha}+\,p{\rm c} +\,q {\rm d}\in\widehat{\mathfrak{g}}\,\) becomes 
\begin{equation}\label{hathad}
ad(\hat h)(\xi)\,=\,\sum_{\alpha}\,\alpha(h)\varphi_{\alpha}\otimes x_{\alpha}
\,+\,  t\, \sum_{\alpha}\,(\,\frac{\partial}{\partial n}\,\varphi_{\alpha})\otimes x_{\alpha}\,+\,t \,[{\rm d}\,,\,y]\,.
\end{equation}

Let \(\widehat{\mathfrak{h}}^{\ast}\) be the dual space of \(\widehat{\mathfrak{h}}\):
\[\widehat{\mathfrak{h}}^{\ast}\,=\,Hom_{\mathbf{C}}(\widehat{\mathfrak{h}}\,,\mathbf{C})\,.\]
   An element  $\alpha$ of the dual space  \(\mathfrak{h}^*\) of \(\mathfrak{h}\) is regarded as a  element
of $\,\widehat{\mathfrak{h}}^{\,\ast}$ by putting
\[
\left\langle \,\alpha,\,{\rm c} \, \right\rangle= 
\left\langle \,\alpha , {\rm d} \,\right\rangle = 0\,.
\]
So  $\Phi \subset \mathfrak{h}^*$ is seen to be a subset of $\,\widehat{\mathfrak{h}}^{\,*}$.    
We define  $\delta\,,\,\Lambda\, \in \widehat{\mathfrak{h}}^{\,*}$,   by
\begin{align}
\left\langle\delta , \alpha _i ^{\vee} \,\right\rangle &= \left\langle\,\Lambda , \alpha _i ^{\vee} \,\right\rangle = 0,  \quad
 1 \leqq i \leqq  l\,, \nonumber\\[0.2cm] 
 \left\langle\,\delta , {\rm c} \,\right\rangle &= 0\,,  \qquad \left\langle\,\delta , {\rm d}\,\right\rangle = 1\,,\\[0.2cm]
 \left\langle\,\Lambda , \,{\rm c} \,\right\rangle &= 1\,,  \qquad \left\langle\,\Lambda , {\rm d}\,\right\rangle = 0\,.\nonumber
\end{align}  
Then \(\alpha_1,\,\cdots\,,\alpha_l,\,\delta,\,\Lambda,\,\) give the basis of \(\widehat{\mathfrak{h}}^{\ast}\).

We shall investigate the decomposition of \(\,\widehat{\mathfrak{g}}\,\) into a direct sum of the simultaneous eigenspaces of \(ad\,(\hat h)\,\), \(\hat h\in \widehat{\mathfrak{h}}\,\).
 For a 1-dimensional representation \(\lambda\in \widehat{\mathfrak{h}}^{\ast}\) we put 
\begin{equation}
\widehat{\mathfrak{g}}_{\lambda}\,=\,\left\{\xi\in \widehat{\mathfrak{g}}\,;\quad \, [\,\hat h,\,\xi\,]_{\widehat{\mathfrak{g}}}\,=\,\langle \lambda, \hat h\rangle\,\xi\quad\mbox{ for }\, \forall\hat h\in \,\widehat{\mathfrak{h}}\,\right\}.
\end{equation}
\(\lambda\)  is called a {\it root } of  the representation \(\left(\,\widehat{\mathfrak{g}}\,,\, ad(\widehat{\mathfrak{h}}\,)\right)\) if \(\lambda\neq 0\) and \(\, \widehat{\mathfrak{g}}_{\lambda}\neq 0\).     \(\,\widehat{\mathfrak{g}}_{\lambda}\) is called the {\it root space} of  \(\lambda\,\).     Let \( \widehat{\Phi}\) be the set of roots: 
 \begin{equation*}
 \widehat{\Phi}=\left\{\lambda=
 \alpha+\,m\,\Lambda\,+\,k_0\delta\in \widehat{\mathfrak{h}}^{\ast}\,;\, \alpha=\sum_{i=1}^l\,k_i\alpha_i\in\Phi, \, k_i,\,m\in \mathbf{Z},\,0\leq i\leq l \,\right\}.
 \end{equation*}
The set  \(\widehat{\Pi}=\{\,\alpha _1,\cdots,\alpha_l , \,\Lambda ,\,\delta \,\} \)  forms a fundamental basis of  $\,\widehat{\Phi}\,$.    
 Thus we have the root space decomposition of \(\widehat{\mathfrak{g}}\) with respect to  \(\widehat{\mathfrak{h}}\) :
 \begin{equation}
 \widehat{\mathfrak{g}}\,=\,\widehat{\mathfrak{g}}_0\,\oplus\,\left(\oplus_{\lambda\in 
 \widehat{\Phi}}\,\widehat{\mathfrak{g}}_{\lambda}\,\right)\,
\,.
 \end{equation}
There are following types of roots:  
\[
(i)\,\lambda=\alpha+k\delta, \, 0 \neq \alpha\in\Phi\,,\quad (ii)\,
 \lambda=k\delta,\quad k\neq 0,\quad (iii)
  \,\lambda=0\delta\,  \quad \mbox{and }\, (iv)\,\lambda=0\,.
 \]
  
\begin{proposition}~~~
We have the following relations: 
\begin{enumerate}
\item
\begin{equation}
\left[\, \widehat{\mathfrak{g}}_{\frac{m}{2}\delta+\alpha}\,,\, \widehat{\mathfrak{g}}_{\frac{n}{2}\delta+\beta}\,\right ]_{ \widehat{\mathfrak{g}}}\,\subset \,
 \widehat{\mathfrak{g}}_{\frac{m+n}{2}\delta+\alpha+\beta}\,\,,
 \end{equation}
 for \(\alpha,\,\beta \in \Phi\) and for  \(m,n\in\mathbf{Z}\).
\item
\begin{equation}
\left[\, \widehat{\mathfrak{g}}_{\frac{m}{2}\delta}\,,\, \widehat{\mathfrak{g}}_{\frac{n}{2}\delta}\,\right ]_{\widehat{\mathfrak{g}}}\,\subset \,
 \widehat{\mathfrak{g}}_{\frac{m+n}{2}\delta}\,, 
 \end{equation}
  for  \(m,n\in\mathbf{Z}\).
\end{enumerate}
\end{proposition}
The Proposition is proved by a standard argument using the properties of Lie bracket.

 We now describe each root space  \(\,\widehat{\mathfrak{g}}_{\lambda}\,\).   
We may assume that the weight vector \(\xi\in \widehat{\mathfrak{g}}\) of each weight \(\lambda\) takes the form \(\xi=y+\sum_{\alpha\in\Phi}\varphi_\alpha\otimes x_{\alpha}\) because others do not contribute to give weight, see (\ref{hathad}).   Let \(x\in \mathfrak{g}_{\alpha}\) for \(\alpha\in \Phi\), \(\alpha\neq 0\), and let  \(\varphi\in \mathcal{L}[m]\)  for \(m\in \mathbf{Z}\), that is,  \(\varphi\) is \(m\)-homogeneous, (\ref{homog}). 
    From (\ref{hathad}) we have  
\begin{eqnarray*}
[\,\phi\otimes h , \,\varphi \otimes x\,] _{ \widehat{\mathfrak{g}}}&=&
( \phi\,\varphi) \otimes [\,h,\,x\,] \,
= \left\langle \alpha , h\right\rangle \varphi\otimes x,
\\[0.2cm]
[\,{\rm d}, \,\varphi \otimes x\,]_{ \widehat{\mathfrak{g}}}&=& (\frac{m}{2}\varphi)\otimes x ),
\end{eqnarray*}
for any \(\phi\otimes h\in \mathcal{K}\mathfrak{h}\).
That is, 
\begin{equation*}
[\,\hat h\,, \varphi \otimes x]_{ \widehat{\mathfrak{g}}} = \left\langle \frac{m}{2}\delta + \alpha \,, \,\hat h\,\right\rangle (\varphi\otimes x)\,,
\end{equation*}
for every \(\hat{h} \in \widehat{\mathfrak{h}}\).   
It implies the relation;  
\[ \mathcal{L}[ m]\otimes \mathfrak{g}_{\alpha}\,\subset \widehat{\mathfrak{g} }_{\frac{m}{2}\delta+ \alpha} \,.\]
    Now let \(y\in \mathcal{L}\mathfrak{h}\).    
 It is written by a linear combination of terms of the form  \(y^{\prime}=\phi_{i_1 i_2\cdots i_t}\otimes h_{i_1i_2\,\cdots i_t}\) with 
 \(h_j\in \mathfrak{h}\) and \(\phi_j\in \mathcal{L}[m_j\,]\), \(j=i_1,\cdots, i_t\), so that 
 \[\frac{\partial}{\partial n}y^{\prime}= (\frac12\,\sum_{k=1}^t\,m_k\,\,) \phi_{i_1 i_2\cdots i_t}\otimes h_{i_1i_2\,\cdots i_t}\,,\]
 and we find that 
 \(y^{\prime}\in \widehat{\mathfrak{g} }_{\frac m2\delta}\) with \(m=\sum_{k=1}^t\,m_k\in\mathbf{Z}\).
Hence 
\[\mathcal{L}\mathfrak{h}\,\subset\, \,\widehat{\mathfrak{g}}_{0\delta}\oplus \oplus_{m\neq 0} \,\widehat{\mathfrak{g}}_{\frac{m}{2}\delta}\,,\]
with \(\,\widehat{\mathfrak{g}}_{0\delta}=
\mathcal{L}[ 0]\otimes\mathfrak{h}\,\), and \( \,\widehat{\mathfrak{g}}_{\frac{m}{2}\delta}=\mathcal{L}[m]\otimes\mathfrak{h}\,\).

\begin{theorem}   
\begin{enumerate}
\item
\begin{eqnarray}
\widehat{\Pi} &=& \left\{ \frac{m}{2} \,\delta + \alpha\,;\quad \alpha \in \Pi\,,\,m\in\mathbf{Z}\,\right\}\nonumber \\[0.2cm]
&& \bigcup \left\{ \frac{m}{2} \,\delta\, ;\quad  m\in \mathbf{Z} \, \right\}  \,.
\end{eqnarray}
is a base of  \(\,\widehat{\Phi}\).   
\item
For \(\alpha\in \Phi\), \(\alpha\neq 0\) and \(m\in \mathbf{Z}\), we have 
 \begin{equation}
\widehat{ \mathfrak{g}}_{\frac{m}{2}\delta+ \alpha}\,=\mathcal{L}[m] \otimes_{\mathbf{C}}\mathfrak{g}_{ \alpha}\,.
 \end{equation}
 \item
 \begin{eqnarray}  
  \,\widehat{\mathfrak{g}}_{0\delta}&=&
\mathcal{L}[0]\otimes_{\mathbf{C}} \mathfrak{h}\,\supset\widehat{\mathfrak{h}},
\\[0.2cm]
 \widehat{ \mathfrak{g}}_{\frac{m}{2}\delta}&=&  \,
 \mathcal{L}[m]\otimes_{\mathbf{C}} \mathfrak{h}\,,\quad\mbox{for  \(0\neq  m\in\mathbf{Z} \) . }\,
 \end{eqnarray}
  \item
 \(\widehat{ \mathfrak{g}}\) has the following decomposition:
\begin{equation}
\widehat{ \mathfrak{g}}\,=\, 
\widehat{\mathfrak{g}}_{0\delta}\oplus\left( 
\oplus_{0\neq m\in\mathbf{Z} }\, \widehat{ \mathfrak{g}}_{\frac{m}{2}\delta}\,\right)\oplus\,\,\left(\oplus_{\alpha\in \Phi,\,  m\in\mathbf{Z} }\, 
\widehat{ \mathfrak{g}}_{\frac{m}{2}\delta+\alpha}\,\right)\,.
\end{equation}
\item
\begin{equation}
\widehat{\mathfrak{g}}_{0\delta}\oplus\left( 
\oplus_{0\neq m\in\mathbf{Z} }\, \widehat{ \mathfrak{g}}_{\frac{m}{2}\delta}\,\right)
=\mathcal{L}\mathfrak{h}=\mathcal{K}\mathfrak{h}\oplus \mathcal{J}\mathfrak{h}\,.
\end{equation}
\end{enumerate}
\end{theorem}
\begin{proof} ~~~ 
First we prove the second assertion.    We have already proved  
\( \mathcal{L}[m]\otimes \mathfrak{g}_{\alpha}\,\subset \widehat{\mathfrak{g}}_{\frac{m}{2}\delta+ \alpha}\).    Conversely,  for \(m\in\mathbf{Z}\)  and \(\xi\in \widehat{\mathfrak{g}}_{\frac{m}{2}\delta+ \alpha}\), we shall show that \(\xi\)  has the form  \(\,\phi\otimes x\,\) with \(\phi\in  \mathcal{L}[m]\) and \(x\in \mathfrak{g}_ \alpha\,\).   
 Let \(\xi=\psi \otimes x\,+\sum\, p_k a_k+q {\rm n}\).   Then
\begin{eqnarray*}
&&[\hat h,\xi]_{ \widehat{\mathfrak{g}}} =[\,\phi^{+(0,0,1)}\otimes h+ \sum s_ka_k+t{\rm n}\,, \,\psi \otimes x\,+\sum\, p_k a_k+q {\rm n}\,]_{ \widehat{\mathfrak{g}}} 
 =\,\psi\otimes [\,h\,,\,x\,]\\[0.2cm]
 &&\qquad + \,t (\,\sum_{n\in \mathbf{Z}} \,\frac{n}{2} \psi_n  \,\otimes x\,)
\end{eqnarray*}
for any \(\hat h=\phi^{+(0,0,1)}\otimes h+\sum s_ka_k+tn\in \widehat{\mathfrak{h}}\,\), where \(\psi=\sum_n\psi_n\) is the homogeneous decomposition of \(\psi\).
From the assumption we have 
\begin{eqnarray*}
 [\,\hat h,\xi\,]_{ \widehat{\mathfrak{g}}} \,&=&\,\langle\, \frac{m}{2}\delta+ \alpha\,,\,\hat h\,\rangle\, \xi\,\\[0.2cm]
&=& < \alpha,\,h>\psi\otimes x\, +(\frac{m}{2}t+< \alpha,\,h>)(\sum p_k a_k+q{\rm n})\,\\[0.2cm]
&& \quad+\,
\frac{m}{2}t\,
  (\sum_{k} \,\psi_{k} )\otimes x.
\end{eqnarray*}
Comparing the above two equations we have \(p_k=q=0\), and $\psi_k = 0$ for all \(k\) 
except for $k= m$.      Therefore    $ \psi \in\mathcal{L}[m]$.   We also have  $[\hat h , \xi]_{ \widehat{\mathfrak{g}}}  =\psi\otimes [h,x]= \langle \alpha ,\,h\rangle \,\psi \otimes x$ for any \(\hat h=\phi^{+(0,0,1)}\otimes h+\sum s_ka_k+td \in\widehat{\mathfrak{h}}\).     Hence  \(x \) has weight \(\alpha\) and 
  \(\xi= \psi_m\otimes x \in \widehat{\mathfrak{g}}_{\frac{m}{2}\delta+ \alpha}\,\).   We have proved
  \begin{equation*}
  \widehat{\mathfrak{g}}_{\frac{m}{2}\delta+\alpha}=\mathcal{L}[m] \otimes_{\mathbf{C}} \mathfrak{g}_{ \alpha}\,.
 \end{equation*}
 Now we shall show 
 \[\mathcal{L}\mathfrak{h}\,\supset\, \,\widehat{\mathfrak{g}}_{0\delta}\oplus \oplus_{m\neq 0} \,\widehat{\mathfrak{g}}_{\frac{m}{2}\delta}\,.\]
where \(\,\widehat{\mathfrak{g}}_{0\delta}=
\mathcal{L}[ 0]\otimes_{\mathbf{C}}\mathfrak{h}\,\), and \( \,\widehat{\mathfrak{g}}_{\frac{m}{2}\delta}=\mathcal{L}[m]\otimes_{\mathbf{R}}\mathfrak{h}\,\).   The converse implication has been proved before, so both sides coincide.     Let \(\xi=\in \widehat{\mathfrak{g}}_{0\delta}\oplus \oplus_{m\neq 0} \,\widehat{\mathfrak{g}}_{\frac{m}{2}\delta}\,\) which we may assume to be the form \(\xi=y+\sum\, p_k a_k+q {\rm n}\).    It satisfies 
\[
 [\,\hat h,\xi\,]_{ \widehat{\mathfrak{g}}} \,=\,\langle\, \frac{m}{2}\delta\,,\,\hat h\,\rangle\,\xi\,,\quad \forall\widehat h\in\widehat{\mathfrak{h}}\,,\]
  for \(m=0\) or \(m\neq 0\).  From (\ref{hathad}) we find \(\xi=y\in \mathcal{L}[m]\mathfrak{h}
 \).      The above discussion yields the first and the fourth assertions.
  \end{proof}

\begin{corollary}
\[\oplus_{ \Phi\ni\alpha\neq 0}\, 
\widehat{ \mathfrak{g}}_{\frac{m}{2}\delta+\alpha}\,=\,\mathcal{L}[m] \otimes_{\mathbf{C}} \mathfrak{g}.\]
\end{corollary}


\subsection{ Standard invariant bilinear form on \(\widehat{\mathfrak{g}}\)}

 Let  \((\cdot\,\vert\cdot\,)\) be the standard invariant form on  \(\widehat{\mathfrak{g}}\)  and let \(\theta\) be the highest root of the root system \(\Phi\).   
Normalize the form \((\cdot\,\vert\cdot\,)\) on \(\mathfrak{g}\) by the conditionn
\((\theta\vert\theta)=2\) and extend it to the whole \(\widehat{\mathfrak{g}}\) by 
\begin{eqnarray*}
(\phi_1\otimes x_1\vert\,\phi_2\otimes x_2) &=&\,(\phi_1\vert \phi_2)\,(x_1\vert x_2)\,,\quad x_i\in \mathfrak{g},\, \phi_i\in\mathcal{L}\,,\quad  i=1,2\,\\[0.2cm]
(\,\,\mathbf{H}{\rm c} +\mathbf{C} {\rm d}\,\vert \, \mathcal{L}\mathfrak{g}\, )&=&0\,,\quad 
({\rm c} \vert {\rm c} )=({\rm d}\vert {\rm d})=0\,,\quad ({\rm c} \vert {\rm d})=1\,,
\end{eqnarray*}
where we have defined  
\[ (\phi\,\vert\psi\,)=\,q Res(\,\phi\cdot \psi)\,.\]
It extends the invariant bilinear form (\ref{invariantform} ) of  \( \mathcal{L}\mathfrak{g}\).   
Here we shall check only the following invariance property.   Since the rests are easy to prove.

\begin{lemma}
\begin{equation}\label{d-action}
 (\,[\,{\rm d}\,, \phi\otimes x\,]\,\vert\, \psi\otimes y\,)\,=\,(\,{\rm d}\,\vert\,[\phi\otimes x\,,\,\psi\otimes y\,]\,)\,
 \end{equation}
 for \(\phi,\,\psi\,\in \mathcal{L}\).
 \end{lemma}
\begin{proof}
The left hand side of the above equation is 
\begin{eqnarray*}
 (\,[{\rm d}\,,\,\phi\otimes x]\,\vert \,\psi\otimes y\,)
 &=&
(\,\frac{\partial}{\partial n}\phi\,\otimes x\,\vert\,\psi\otimes y\,) \,=
(\,\frac{\partial}{\partial n}\phi\,\vert\, \psi )\,(x\vert\, y)\\[0.2cm]
&=&  qRes(\frac{\partial}{\partial n} \phi\cdot\psi\,)\,(x\vert\, y)=\frac{1}{2}\left( qRes(\frac{\partial}{\partial n} \phi\cdot\psi\,)-qRes(\frac{\partial}{\partial n} \psi\cdot\phi)\right)(x\vert y)\,.
\end{eqnarray*}
Here we used the fact: \(q\,Res\,\frac{\partial}{\partial n}(\psi\cdot\phi)=0\).
The right hand side of (\ref{d-action}) is equal to
\begin{equation*}
\,(\,{\rm d}\,\vert\,[\phi\otimes x\,,\,\psi\otimes y\,]^{\rm c}\,)\,
=\,({\rm d}\,\vert \,[\phi\otimes x, \psi\otimes y]\,+\,(x\vert y)\,A(\phi,\psi){\rm c}\,)=\,A(\phi,\psi)\, (x\vert y)\,.
\end{equation*}



Proposition\ref{resDandn} implies 
 \[qRes(\,\frac{\partial}{\partial n}\phi\,\cdot\, \psi\,)=\,qRes(\Do\phi\,\cdot\, \psi\,)\,\, \mbox{ and }\,qRes(\,\frac{\partial}{\partial n}\psi\,\cdot\, \phi\,)=\,qRes(\Do\psi\,\vert\, \phi\,)\,.\]
 Hence 
\[ qRes(\frac{\partial}{\partial n} \phi\cdot\psi\,)-\,qRes(\frac{\partial}{\partial n} \psi\cdot\phi)=qRes(\Do\phi\,\cdot \psi)\,-\,qRes(\Do\psi\,\vert\phi)
\,=A(\phi,\psi).\]
\end{proof}

 \subsection{ Chevalley generators of \(\,\widehat{\mathfrak{g}}\)}
 
We consider the root spaces \(\mathfrak{g}_{\theta}\) and \(\mathfrak{g}_{-\theta}\) where \(\theta\) is the highest root of \(\mathfrak{g}\).   We have \(\dim \mathfrak{g}_{\theta}=\dim \mathfrak{g}_{-\theta}=1\), and the bilinear form \((\cdot\vert\cdot)\) restricted on \(\mathfrak{g}_{\theta}\times \mathfrak{g}_{-\theta}\), that is, the form restricted on 
\((\mathcal{L}\mathfrak{g})_{\theta}\times (\mathcal{L}\mathfrak{g})_{-\theta}\), is non-degenerate.   
Let \(\omega^0\) be the Chevalley involution of \(\mathfrak{g}\).    

By the natural embedding of \(\mathfrak{g}\) in \(\widehat{\mathfrak{g}}\) we have the vectors 
\begin{eqnarray}
h_i&=&\phi^{+(0,0,1)}\otimes h_i\,\in \widehat{\mathfrak{h}},\,\nonumber\\[0.2cm]
e_i&=&\phi^{+(0,0,1)}\otimes e_i\,\in \widehat{\mathfrak{g}}_{0\delta+\alpha_i},\quad f_i=\phi^{+(0,0,1)}\otimes f_i\,\in \widehat{\mathfrak{g}}_{0\delta-\alpha_i},\qquad i=1,\cdots,l\,.
\nonumber\end{eqnarray}
Then 
\begin{eqnarray}
\left[e_i\,,f_j\,\right]_{ \widehat{\mathfrak{g}}} &=&\,\delta_{ij}\,h_i\,,\nonumber\\[0.2cm]
\left[h_i\,,e_j\,\right ]_{ \widehat{\mathfrak{g}}}&=&\,a_{ij}\,e_j,\quad 
\left[h_i\,,f_j\,\right]_{ \widehat{\mathfrak{g}}} =\,- a_{ij}\,f_j\,,\quad 1\leq i\,,\,j\leq l .
\end{eqnarray}
We have obtained a part of generators of \(\widehat{\mathfrak{g}}\) that come naturally from \(\mathfrak{g}\).   We want to augment these generators to the Chevalley generators of  \(\widehat{\mathfrak{g}}\).    
We take the following generators of  the algebra \(\mathcal{L}\):
\begin{eqnarray}
 I\,&=\,\phi^{+(0,0,1)}\,=\,\left(\begin{array}{c} 1\\ 0\end{array}
 \right)\,, \quad 
 J\,&= \,\phi^{ +(0,0,0)}\,=\,\left(\begin{array}{c}  0\\ -1\end{array}\right)\,,\\[0.2cm]
  \kappa\,&=\,\phi^{+(1,0,1)}\,=\,\left(\begin{array}{c} z_2\\-\overline{z}_1\end{array}
 \right)\,, \quad 
 \lambda\,&= \,\phi^{ -(0,0,0)}\,=\,\frac{1}{\vert z\vert^4}\,\left(\begin{array}{c}  z_2\\ \overline{z}_1\end{array}\right)\lvert_{|z|=1}\,,
 \end{eqnarray}

  We put
\begin{eqnarray*}
\kappa_{\ast}&=&\,-\,\frac{1}{\sqrt{2}}\phi^{+(1,1,2)}-\,\frac{1}{2}\phi^{+(1,0,1)}\,+\,\frac{1}{2}\,\phi^{ -(0,0,0)}\,\\[0.2cm]
&=& \,
\left(\begin{array}{c}\overline{z}_2\\0\end{array}\right)-\frac{1}{2}\left(\begin{array}{c}z_2\\ -\overline{z}_1\end{array}\right)
+\frac{1}{2|z|^4}\left(\begin{array}{c}z_2\\ \overline{z}_1\end{array}\right)\lvert_{|z|=1}\,=\left(\begin{array}{c}\overline{z}_2\\ \overline{z}_1\end{array}\right)\,.\\[0.2cm]
\lambda_{\ast}&=&\,\,\frac{1}{\sqrt{2}}\phi^{+(1,0,2)}\,+\,\frac{1}{2}\phi^{+(1,1,1)}\,+\,\frac{1}{2}\,\phi^{ -(0,0,1)}\,\\[0.2cm]
&=&\,\left(\begin{array}{c}\overline{z}_1\\0\end{array}\right)+\frac{1}{2}\left(\begin{array}{c}z_1\\ \overline{z}_2\end{array}\right)
+\frac{1}{2|z|^4}\left(\begin{array}{c}-z_1\\ \overline{z}_2\end{array}\right)\lvert_{|z|=1}\,=\left(\begin{array}{c}\overline{z}_1\\ \overline{z}_2\end{array}\right)\,.
\end{eqnarray*}
It holds that 
\[ \kappa\,\in \mathcal{L}[1]\,,\quad \lambda\,\in \mathcal{L}[-3]\,,\quad 
\kappa_{\ast}\,,\,\lambda_{\ast}\,\in \mathcal{L}[1]\oplus\mathcal{L}[-3]\,
\]

\begin{lemma}~~
\begin{enumerate}
\item
\begin{eqnarray}
\kappa\,\cdot \kappa_{\ast}\,&=\,\kappa_{\ast}\,\cdot \kappa \,=\left( \begin{array}{c} 1\\ 0 \end{array}\right) \,,\\[0.2cm]
\,\lambda_{\ast}\,\cdot\lambda \,&=\,\left( \begin{array}{c} 0\\ 1 \end{array}\right) \,.
\end{eqnarray}
\item
\begin{equation}
\,A(\kappa,\kappa_{\ast})\,=\,\left(\begin{array}{c}1\\ 0\end{array}\right)\,,\qquad 
A(\lambda,\lambda_{\ast})=\,\,\left(\begin{array}{c}0\\ 1\end{array}\right)\,\,.
\end{equation}
\end{enumerate}
\end{lemma}
\begin{proof}~~~
We have the following equations
\begin{eqnarray*}
\Do\,\kappa\,&=&\,\frac{1}{2}\kappa\,,\qquad
\Do\,\lambda\,
=\,-\frac{3}{2}\,\lambda\,,\\[0.2cm]
 \Do\,\kappa_{\ast}\,&=&\,
\left(\begin{array}{c}-\frac12\overline{z}_2-z_2\\ -\frac12\overline{z}_1\end{array}\right)\,
=\,\frac{1}{2}\kappa_{\ast}\,-\lambda\,,\qquad
 \Do\,\lambda_{\ast}\,=\,\frac{1}{2}
\left(\begin{array}{c}\overline z_1\\ \overline{z}_2\end{array}\right)\,+\,\left(\begin{array}{c} z_1\\ - \overline{z}_2\end{array}\right)
=\,\frac{1}{2}\lambda_{\ast}\,-\,\phi^{-(0,0,1)}.
\end{eqnarray*}
By virtue of these equations we have 
\[\,\Do\kappa\cdot\kappa_{\ast}\,-\,\Do\kappa_{\ast}\cdot\kappa\,=\,\lambda\cdot\kappa\,=\,\left(\begin{array}{c} z_2^2+\vert z_1\vert^2\\ \overline z_1(z_2-\overline z_2)\end{array}\right)\,.\]
The quarternion residue of the last term is \(\left(\begin{array}{c} 1 \\ 0 \end{array}\right)\), so \(A(\kappa\,,\,\kappa_{\ast}\,)=\left(\begin{array}{c} 1 \\ 0 \end{array}\right)\).

Similarly we have
\[
\Do\lambda\cdot\lambda_{\ast}\,-\,\Do\lambda_{\ast}\cdot\lambda\,=\,
\,\left(\begin{array}{c} - z_1(\overline z_2+ z_2)  \\  -\overline z_2^2+\vert z_2\vert^2-\frac12    \end{array}\right)\,. \]
The quarternion residue of the last term is \(\left(\begin{array}{c} 0 \\ 1 \end{array}\right)\), so \(A(\lambda\,,\,\lambda_{\ast}\,)=\left(\begin{array}{c} 0 \\ 1 \end{array}\right)\).

\end{proof}

Let \(\theta\) be the highest root of \(\mathfrak{g}\) and suppose that \(e_\theta\in\mathfrak{g}_{-\theta}\) and \(f_\theta\in\mathfrak{g}_{\theta}\) satisfy the relations \([e_\theta\,,\,f_{\theta}]\,=\,h_\theta\) and \((e_\theta\vert f_\theta)=1\).    We introduce the following vectors of \(\,\widehat{\mathfrak{g}}\,\);
\begin{align}
 f_{J}&=J\otimes f_{\theta}\,\in \widehat{\mathfrak{g}}_{0\delta+\theta}\,,\quad &
  e_{J}&=(-J)\otimes e_{\theta} \,\in \widehat{\mathfrak{g}}_{0\delta-\theta}\,,\\[0.2cm]
   f_{\kappa}&=\kappa\otimes f_{\theta}\,\in \widehat{\mathfrak{g}}_{\frac12\delta+\theta}\,,\quad &
  e_{\kappa}&=\kappa_{\ast} \otimes e_{\theta} \,\in \widehat{\mathfrak{g}}_{-\frac32\delta-\theta}\oplus \widehat{\mathfrak{g}}_{\frac12\delta-\theta}\,,\\[0.2cm]
  f_{\lambda}&=\lambda \otimes f_{\theta}\,\in \widehat{\mathfrak{g}}_{-\frac32\delta+\theta}\,,\quad &
  e_{\lambda}&=\lambda_{\ast}\otimes e_{\theta} \,\in \widehat{\mathfrak{g}}_{-\frac{3}{2}\delta-\theta}\oplus \widehat{\mathfrak{g}}_{\frac12\delta-\theta}\,.
\end{align}
Then we have the generators of \(\mathcal{L}\mathfrak{g}\oplus\,\mathbf{H}{\rm c}\)  that are given by the following triples:
\begin{eqnarray}
&& \left(\,\widehat{e}_i,\widehat{f}_i,h_i \right) \quad  i=1,2,\cdots,l,\nonumber\\[0.2cm]
&&\left( \widehat{e}_{\lambda}, \widehat{f}_{\lambda},h_{\theta}\right),\quad 
\left( \widehat{e}_{\kappa}, \widehat{f}_{\kappa},h_{\theta}\,\right),\quad
 \,\left( \widehat{e}_{J}, \widehat{f}_{J},h_{\theta}\right)\, \,.
\end{eqnarray}
These triples satisfy the following relations.
\begin{proposition}~~~
\begin{enumerate}
\item
\begin{equation}
\left[\,e_{\pi}\,,\,f_i\,\right]_{\widehat{\mathfrak{g}}}=\,
\left[\,f_{\pi}\,,\,e_i\,\right]_{\widehat{\mathfrak{g}}} =0\,,\quad\mbox{for } \,1\leq i\leq l ,\,\mbox{ and }\, \pi=J,\,\kappa,\,\lambda\,.\label{i}
\end{equation}
\item
\begin{equation}
\left[\,e_{J}\,,\,f_J\,\right]_{\widehat{\mathfrak{g}}}=\,\widehat  h_\theta\,,
\end{equation}
\item
\begin{equation}\quad
\left[\,e_{\lambda}\,,\,f_{\lambda}\,\right]_{\widehat{\mathfrak{g}}}=\sqrt{-1}\,\widehat  h_\theta\,-{\rm c} ,\quad
\left[\,e_{\kappa}\,,\,f_{\kappa}\,\right]_{\widehat{\mathfrak{g}}}=\sqrt{-1}\,\widehat  h_\theta\,-{\rm c}\,.
\end{equation}
\end{enumerate}
\end{proposition}
Adding the element \(n\) to these generators of \(\mathcal{L}\mathfrak{g}\oplus\,\mathbf{H}{\rm c}\)  we have obtained the Chevalley generators of 
\(\widehat{\mathfrak{g}}\).


\begin{thebibliography}{99}
\bibitem[A]{A} J.F.Adams, Lectures on Lie Groups, W.A.Benjamin, Inc. New York, 1969.
\bibitem[B]{B} N. Bourbaki, Groupes et Alg\`ebres de Lie,   Elements de Math\'ematique,   Hermann, Paris, 1973.
\bibitem[C]{C} R. W. Carter, Lie Algebras of Finite and Affine Type, Cambridge Studies in Advanced Mathematics 96, Cambridge University Press, Cambridge, 2005.
\bibitem[D]{D} J. Dixmier,   Alg\`ebres enveloppantes,   Cahiers scientifiques XXXVII, Gauthier-Villars, Paris, 1974.
\bibitem[D-S-Sc]{D-S-Sc} R. Delanghe, F. Sommen and V. Sou\v cek, Clifford algebra and spinor-valued functions: a function theory for the Dirac operator, Kluwer, Dordrecht, 1992.
\bibitem[F]{F} T. Friedrich, Dirac Operators in Riemannian Geometry,   Graduate Studies in Mathematics, American Mathematical Society, Providence, 1997.
\bibitem[G-M]{G-M} J. Gilbert and M. Murray, Clifford algebras and Dirac operators in harmonic analysis,
Cambridge University Press, Cambridge, 1991.
\bibitem[H]{H} J-E. Humphreys,   Introduction to Lie Algebras and Representation Theory,  Springer-Verlag,  Berlin, 1972.
\bibitem[K]{K} V. Kac, Infinite dimensional Lie algebras, Cambridge University Press, Cambridge, 1983.
\bibitem[K-W]{K-W}B. Khesin and R. Wendt,
The geometry of infinite-dimensional groups,
 A Series of Modern Surveys in Mathematics 51, 
 Ergebnisse der Mathematik und ihrer Grenzgebiete. 3. Folge. Springer-Verlag, Berlin, 2009. 
\bibitem[Ko0]{Ko0} T. Kori, Lie algebra of the infinitesimal automorphisms on $S^3$ and its central extension, J. Math. Kyoto Univ. {\bf 36} (1996), no. 1, 45-60. 
\bibitem[Ko1]{Ko1} T. Kori, Index of the Dirac operator on $S^4$ and the infinite-dimensional Grassmannian on $S^3$,
 Japan. J. Math. (N.S.), {\bf  22} (1996), no. 1, 1-36.
\bibitem[Ko2]{Ko2}T. Kori, Spinor analysis on $\mathbf{C}^2$ and on conformally
 flat 4-manifolds, Japan. J. Math., {\bf 28} (1)(2002), 1-30. 
 \bibitem[Ko3]{Ko3}T. Kori,  Cohomology groups of harmonic spinors on conformally flat manifolds, Trends in Mathematics: Advances in Analysis and Geometry, 209-225, Birkh{\"a}user Verlag, Basel, 2004.
 \bibitem[K-I]{K-I} T. Kori and Y. Imai,  Lie algebra extensions of current algebras on \(S^3\), International Journal of Geometric Methods in Modern Physics vol. 12(2015), 1550087, World Scientific Publishing Company.
 \bibitem[Kc]{Kc} T. Kori,  \(\mathfrak{sl}(n,\mathbf{H})\)-Current Algebra on \(S^3\), arXiv:1710.09712v1[math.DG].
\bibitem[Kq]{Kq} T. Kori, Quarternification of complex Lie algebras, arXiv: 2008.02913v2 [math.RT].
 \bibitem[M]{M}
J. Mickelsson,  Current Algebras and Groups, Plenum Press, New York and London, 1989.
  \bibitem[P-S]{P-S}A. Pressley and G. Segal, Loop groups, Oxford Mathematical Monographs. Oxford Science Publications. The Clarendon Press, Oxford University Press, New York, 1986.
 \bibitem[S]{S} J-P. Serre, Algebres de Lie semi-simples complexes, W.A.Benjamin, New-York, 1966.
 \bibitem[W]{W} M. Wakimoto, Infinite-Dimensional Lie Algebras, Translations of Mathematical Monographs, 195.
 Iwanami Series in Modern Mathematics. American Mathematical Society, Providence, RI, 2001.

\end{thebibliography}
\end{document}